\documentclass[12pt,oneside,reqno]{amsart}
\usepackage[margin=1in]{geometry}
\usepackage{color}
\usepackage{esint,amssymb}
\usepackage{graphicx}
\usepackage{MnSymbol}
\usepackage{mathtools}
\usepackage[colorlinks=true, pdfstartview=FitV, linkcolor=blue, citecolor=blue, urlcolor=blue,pagebackref=false]{hyperref}
\usepackage{microtype}
\usepackage{amsmath}
\usepackage{xifthen}
\usepackage{verbatim}
\definecolor{darkgreen}{rgb}{0,0.5,0}
\definecolor{darkblue}{rgb}{0,0,0.7}
\definecolor{darkred}{rgb}{0.9,0.1,0.1}

\newtheorem{theorem}{Theorem}
\newtheorem{proposition}[theorem]{Proposition}
\newtheorem{lemma}[theorem]{Lemma}
\newtheorem{corollary}[theorem]{Corollary}

\theoremstyle{definition}
\newtheorem{remark}[theorem]{Remark}

\newtheorem{definition}[theorem]{Definition}

\newcommand{\tref}[1]{Theorem~\ref{t.#1}}
\newcommand{\pref}[1]{Proposition~\ref{p.#1}}
\newcommand{\lref}[1]{Lemma~\ref{l.#1}}
\newcommand{\cref}[1]{Corollary~\ref{c.#1}}

\newcommand{\sref}[1]{Section~\ref{s.#1}}

\newcommand{\eref}[1]{(\ref{e.#1})}

\numberwithin{equation}{section}
\numberwithin{theorem}{section}

\newcommand{\N}{\mathbb{N}}
\newcommand{\R}{\mathbb{R}}

\renewcommand{\a}{\bar{a}_{ij}}
\newcommand{\cm}{\bar{c}}

\newcommand{\eps}{\varepsilon}

\newcommand{\test}[1][]{%
\ifthenelse{\equal{#1}{}}{omitted}{given}%
}

\newcommand{\derv}[3]{\partial_x^{\alpha{#1}}\partial_v^{\beta{#2}}Y^{\sigma{#3}}}
\newcommand{\der}{\derv{}{}{}}

\renewcommand{\d}[1]{\ensuremath{\operatorname{d}\!{#1}}}
\DeclarePairedDelimiter{\norm}{\lVert}{\rVert}
\DeclareMathOperator{\sm}{small}

\DeclareMathOperator{\boot}{Boot}

\DeclareMathOperator{\ini}{in}
\newcommand{\jap}[1]{\langle {#1} \rangle}

\renewcommand{\bar}{\overline}
\renewcommand{\tilde}{\widetilde}

\renewcommand{\part}{\partial}

\usepackage{dsfont}

\begin{document}
\title[Stability of vacuum for the Landau Equation with hard potentials]{Stability of vacuum for the Landau Equation\break with hard potentials}
\begin{abstract}
We consider the spatially inhomogeneous Landau equation with Maxwellian and hard potentials (i.e with $\gamma\in[0,1)$) on the whole space $\R^3$. We prove that if the initial data $f_{\ini}$ are close to the vacuum solution $f_{\text{vac}}=0$ in an appropriate weighted norm then the solution $f$ exists globally in time. This work builds up on the author's earlier work on local existence of solutions to Landau equation with hard potentials.

Our proof uses $L^2$ estimates and exploits the null-structure established by Luk [Stability of vacuum for the Landau equation with moderately soft potentials, Annals of PDE (2019) 5:11]. To be able to close our estimates, we have to couple the weighted energy estimates, which were established by the author in a previous paper [Local existence for the Landau equation with hard potentials, https://arxiv.org/abs/1910.11866], with the null-structure and devise new weighted norms that take this into account.
 \end{abstract}
\author[S. Chaturvedi]{Sanchit Chaturvedi}
\address[Sanchit Chaturvedi]{450 Jane Stanford way, Bldg 380, Stanford, CA 94305}
\email{sanchat@stanford.edu}
\subjclass[2010]{}
\date{\today}

\maketitle
\section{Introduction}
We study the Landau equation for the particle density $f(t,x,v)\geq 0$ in the whole space $\R^3$. $t\in \R_{\geq 0}, x\in \R^3$ and $v \in \R^3$. The Landau equation is as follows\\
\begin{equation}\label{e.landau_collision}
\part_t f+\part_{v_i}f=Q(f,f),
\end{equation}
where $Q(f,f)$ is the collision kernel given by\\
\begin{equation*}
Q(f,f)(v):=\part_{v_i}\int a_{ij}(v-v_*)(f(v_*)(\part_{v_j}f)(v)-f(v)(\part_{v_j}f)(v_*))\d v_*
\end{equation*}
and $a_{ij}$ is the non-negative symmetric matrix defined by
\begin{equation}\label{e.matrix_a}
a_{ij}(z):=\left(\delta_{ij}-\frac{z_iz_j}{|z|^2}\right)|z|^{\gamma+2}.
\end{equation}

In all the expressions above (and in the rest of the paper), we use the convention that repeated lower case Latin indices are summed over $i,j=1,2,3.$

The quantity $\gamma$ encodes the type of interaction potential between the particles. Physically, the relevant regime is $\gamma\in[-3,1]$. We will be concerned with the regime $\gamma\in[0,1]$, that is the Maxwellian ($\gamma=0$) and all of the hard potentials ($\gamma\in (0,1]$).

For us it will be convenient to work with a slightly modified but equivalent version of \eref{landau_collision} which is as follows,
\begin{equation}\label{e.landau}
\part_t f+v_i\part_{x_i}f=\a\part^2_{v_iv_j}f-\bar{c}f,
\end{equation}
where $c:=\part^2_{z_iz_j}a_{ij}(z)$, $\a:=a_{ij}*f$ and $\bar{c}:=c*f$, where $*$ is the convolution operator.

Our main result is that for sufficiently regular and sufficiently small initial data, we have global existence and uniqueness for \eref{landau}. That is we have a unique, non-negative solution to the Cauchy problem for all times.

\begin{theorem}\label{t.global}
Let $\gamma \in [0,1)$, $0<\delta<\frac{1}{8}$ and $d_0>0$. There exists an $\eps_0=\eps_0(\gamma,d_0)>0$ such that if
\begin{equation*}
\sum\limits_{|\alpha|+|\beta|+|\sigma|\leq 10}\norm{\jap{v}^{20-\frac{3}{2}|\alpha|-\frac{1}{2}|\beta|-\frac{3}{2}|\sigma|}\jap{x-v}^{20-\frac{3}{2}|\sigma|-\frac{1}{2}|\beta|-\frac{1}{2}|\alpha|}\part_x^\alpha\part_v^\beta(\part_x+\part_v)^\sigma(e^{2d_0\jap{v}}f_{\ini})}_{L^2_xL^2_v}^2<\eps
\end{equation*}
for some $\eps\in[0,\eps_0]$ then there exists a global and non-negative solution $f$ to \eref{landau} with $f(0,x,v)=f_{\ini}(x,v)$. In addition, $fe^{(d_0(1+(1+t)^{-\delta})\jap{v}}$, is unique in the energy space $\tilde E_T^4\cap C^0([0,T);\tilde Y^4_{x,v})$ for all $T\in (0,\infty).$

Moreover, $fe^{(d_0(1+(1+t)^{-\delta})\jap{v}}\in E_T\cap C^0([0,T);Y_{x,v})$ for any $T\in(0,+\infty)$.
\end{theorem}
See \sref{notation} for definitions of the relevant energy spaces.

\begin{remark}
The global existence problem for $\gamma=1$, can also be done in a similar way. Since the local result for the same requires a slightly different  hierarchy, we only mention global existence result for $\gamma\in [0,1)$ for a consice presentation. 
\end{remark}
\begin{remark}
In an identical manner as in \cite{Lu18}, we can use \tref{global} to show that the long time limit exists and that the long time dynamics of the solution is governed by the transport part when considered in a weaker topology.\\
 To formalize this define $f^\sharp(t,x,v)=f(t,x+tv,v)$. With this definition we have that $f$ satisfies the tranport equation, $$\part_t f+v_i\part_{x_i}f=0,$$ if and only if $f^\sharp(t,x,v)$ is independent of $t$. Our claim then, is that there is a unique $f^\sharp_\infty:\R^3\times\R^3\to \R$ such that $f^\sharp$  converges to $f^\sharp_\infty$ in an appropriate topology.

In addition, the long time asymptotics of macroscopic quantities, such as the density $\rho=\int_{\R^3}f(t,x,v)\d v$, the mass $m_i(t,x)=\int_{\R^3}v_i f(t,x,v)\d v$ and the energy $e(t,x)=\frac{1}{2}\int_{\R^3}|v|^2f(t,x,v)$, is governed by the macroscopic quantities associated to $f^\sharp(x-tv,v).$

Finally, it can also be proved that there exist initial data $f_{\ini}$ such that $f^\sharp_\infty$ is neither $0$ nor a traveling global Maxwellian. This happens despite Boltzmann's H-theorem, which stipulates that the entropy functional $H[f]=\int_{\R^3}\int_{\R^3} f\log f\d v\d x$ is decreasing along the flow of the Landau equation. Further it is known that the solutions to \eref{landau} for which the H-functional is constant are the traveling global Maxwellians. We do not present the proof here and refer the reader to \cite{Levermore}, \cite{To87}, \cite{BaGaGoLe16} and \cite{Lu18} for more details on such issues.
\end{remark}
In the case of inhomogeneous equations, this is the first global existence result for Landau equation with hard potentials such that the solutions are not close to Maxwellians. We note that a major difficulty in getting a global result for hard potentials was the velocity weight issue which was already present in the local existence problem. Hence the present work builds up on the author's past work on local existence \cite{Cha19}. That said, the global problem not only requires us to overcome the moment loss issue but also to gain enough time decay to be able to close the estimates.
\subsection{Related Works} In this section we give a rather inexhaustive list of pertinent results that deal
with a long range interaction potential. 
\begin{enumerate}
\item \textbf{Stability of vacuum for collisional kinetic models:} The earlier works were concerned with the Boltzmann equation with an angular cut-off. The first work was by Illner-Shinbrot \cite{IlSh84}. There were many other follow up works of \cite{IlSh84}; see for instance \cite{Ar11,BaDeGo84,BeTo85,Guo01,Ha85,HeJi17,Po88,To86,ToBe84}. Perturbations to travelling global Maxwellians were studied in \cite{AlGa09,Go97,To88, BaGaGoLe16} and it was shown that the long-time dynamics is governed by dispersion.

The first stability problem for vacuum result with long-range interactions was only recently obtained by Luk, in \cite{Lu18}, who proved the result in the case of moderately soft potentials ($\gamma \in (-2,0)$). Luk combined $L^\infty$ and $L^2$ methods to prove global existence of solutions near vacuum. Moreover, Luk proves that the main mechanism is dispersion and showed that the long-time limit of the solution to Landau equation solves the transport equation. Furthermore, Luk also showed the solution does not necessarily approach global Maxwellains. An interesting feature of Luk's paper is the use of a hierarchy in his norm which is necessary for gaining enough time decay. We use a similar weighted norm inspired from his techniques.
\item \textbf{Spatially homogenous solutions to Landau Equation:} The theory is rather well developed in the spatially homogeneous case.
\begin{itemize}
\item \textbf{Soft-Potentials, $\gamma\in[-3,0)$:} Global existence was studied in \cite{ArPe77,De15,Vi98.2} and uniqueness in \cite{Fo10,FoGu09}. The problem of global uniqueness remains a notable open problem for $\gamma\in [-3,-2]$. See \cite{AlLiLi15,De15,Si17,Wu14} for more a priori estimates.
\item \textbf{Maxwellian molecule, $\gamma=0$:} See \cite{Vi98} for results on this case.
\item \textbf{Hard potentials, $\gamma>0$:} This case was studied in detail by Desvillettes--Villani in \cite{DeVi00,DeVi00.2}, who showed existence and smoothness of solution with suitable initial data, as well as the appearance and propagation of various moments and lower bounds. 
\end{itemize}
\item \textbf{Spatially inhomogenous solutions to Landau Equation:} Whether regular solutions give rise to globally regular solutions remains an outstanding open problem . There are perturbative results to this end but the general result still seems elusive.
\begin{itemize}
\item \textbf{Global nonlinear stability of Maxwellians,} The stability problem of\\ Maxwellians on a periodic box was established in Guo's seminal work, \cite {Guo02.1}. Since then Guo's nonlinear method has seen considerable success in understanding near Maxwellian regime \cite{Guo02,Guo03,Guo03.2,StGu04,StGu06,Guo12,StZh13}. Remarkably,this regime is well understood in the case of non-cutoff Boltzmann equation as well see, \cite{GrSt11,AMUXY12,AMUXY12.2,AMUXY12.3}.
\item \textbf{Conditional regularity,} Recently there has been some results that concern the regularity of solutions to Landau equation assuming a priori pointwise control of the mass density, energy density and entropy density; see \cite{CaSiSn18,Si17, Sn18, HeSn17, HeSnTa17}.
\item \textbf{Soft Potentials, $\gamma\in [-3,0)$:} The case of soft-potentials is fairly well developed thanks to the recent works \cite{GoImMoVa16,CaSiSn18,HeSn17,HeSnTa17,Si17}. Local well-posedness results are available for the Landau equation due to works \cite{HeSnTa17,HeSnTa19.1}. We also have local existence results for Boltzmann in the soft potential regime thanks to \cite{AMUXY13,HeSnTa19.2}. Other than near Maxwellian regime, global existence is also available near vacuum for $\gamma \in(-2,0)$ due to Luk's recent work \cite{Lu18}.
\item \textbf{Hard potentials and Maxwellian molecule case, $\gamma\in [0,1]$:} The literature is somewhat sparse for these cases. Snelson proved in \cite{Sn18} that the solution to \eref{landau} with hard potentials (under assumptions of upper and lower bounded mass, energy and entropy) satisfies gaussian upper and lower bounds. The main feature of the paper is the appearance of these bounds which is reminiscent of the spatially homogenous case.

$\quad$ The author proved in \cite{Cha19} a local existence result assuming that the initial data is in a weighted tenth-order Sobolev space and has exponential decay in the velocity variable. This local result is a starting point for the present work and the present result relies on a bootstraping argument to prove global existence result from a local one.
\end{itemize}
\item \textbf{Dispersion and stability for collisionless models:} The dispersion properties of the transport operator, which we leverage to prove our global existence result, have been instrumental in proving stability results for close-to-vacuum regime in collisionless models; see \cite{BaDe85,GlSc88,GlSt87} for some early results. Relations between these results and the stability of vacuum for the Boltzmann equation with angular cutoff is discussed in \cite{BaDeGo84}. More recent results can be found in \cite{Bi17, Bi18,Bi19.1, Bi19.2, Bi19.3, FaJoSm17.1,Sm16,Wa18,Wa18.1,Wa18.2,Wo18}. See also \cite{FaJoSm17,LiTa17,Ta17} for proof of the stability of the Minkowski spacetime for the Einstein-Vlasov system.
\end{enumerate}
\subsection{Proof Strategy:} The idea is to start with the local existence result from \cite{Cha19} and use bootstraping to get global existence for small enough initial data.
\begin{enumerate}
\renewcommand{\labelenumi}{(\theenumi)}
\item \textbf{Local existence:} We would need to use a variant of the local existence result etablished in \cite{Cha19} which we explain in \sref{local}. That said we highlight three ingredients from \cite{Cha19} that would still be handy in this set-up.
\begin{itemize}
\item \label{1} Use of a time dependent exponential\footnote{In contrast to \cite{HeSnTa17}, we use an exponential in \cite{Cha19} and this is crucial as we are not able to close the estimates with a Gaussian due to the moment loss issue for the hard potentials.} in $\jap{v}$: As $t$ increases, one only aims for an upper bound by a weaker exponential in the $v-$variables. This lets one to control the $v-$weights in the coefficient $\a$ and in $\cm$.

\quad We thus use a time-dependent exponential weight in $\jap{v}$. Our choice of weights will decrease as $t$ increases, but it needs to decay in a sufficiently slow manner so that it is non-degenerate as $t\to \infty$. More precisely, we define 
\begin{equation}\label{e.g}
g:=e^{d(t)\jap{v}}, \hspace{.5em} d(t):=d_0(1+(1+t)^{-\delta})
\end{equation}
for appropriate $d_0>0$, $\delta>0$ and estimate $g$ instead of $f$, which satisfies the following equation,
\begin{equation} \label{e.rough_eq_for_g}
\begin{split}
\partial_tg+v_i\partial_{x_i}g+\frac{\delta d_0}{(1+t)^{1+\delta}}\jap{v}g&=\a [f]\partial^2_{v_iv_j}g-\cm [f]g+\text{other terms}.
\end{split}
\end{equation}
\item Use of $L^2$ estimates: This lets an integration by parts argument to go through without a loss of derivatives. $L^2$ estimates were already used to establish local existence for Landau equation with soft potentials in \cite{HeSnTa17}.
\item Use of weighted $L^2$ norms and a hierarchy of $\jap{v}$ weights: This facilitates us in overcoming the velocity weight issue in the coefficient $\a$ and is crucial for closing the argument.

\quad As in \cite{Cha19}, a purely $L^2$ approach helps us take care of the velocity weight issue. To see this multiply \eref{landau} by $f$ and integrate in velocity and space. With the help of integration by parts we get a term of the form, $\norm{\part^2_{v_iv_j}\a f^2}_{L^1_xL^1_v}$. Ignoring the time decay issue for now, we have that $$|\part^2_{v_iv_j}\a|\lesssim_t \jap{v}^\gamma$$ (see \pref{pointwise_estimates_a} for more details).

\quad Next note that in a $L^2$ approach, we are guaranteed to have at least two derivatives falling on $\a$ thanks to integration by parts. When at least two of them are $\part_v$ we are fine by above observations, but when there are less than two velocity derivatives, we still have the velocity weight issue. This is where we use the weight heirarchy that was introduced by the author in \cite{Cha19}.
\end{itemize}

For the global problem we still only employ $L^2$ estimates with a weighted norm and a hierarchy of weights but in this setting, we need additional ingredients to handle the large time behavior of $g$.
\item \textbf{Why $L^2$ estimates? Decay and heuristics:} \label{2}In the near vacuum region, one expects that the solutions to the Landau equation approach solutions to the transport equation,
\begin{equation}\label{e.lin_transport}
\part_t f+v_i\part_{x_i}f=0,
\end{equation}
as $t\to \infty$.

We thus expect that the collision term on the rhs of the Landau equation, \eref{landau} to decay as $t\to \infty$. It turns out, though, that the linear decay is insufficient and one has to exploit the null structure of the equation (see point (\ref{3}) for more details). This null structure was first observed and used by Luk in \cite{Lu18}. Luk combined $L^\infty$ bounds with $L^2$ bounds to be able to gain enough decay without loss of derivatives.

This $L^2-L^\infty$ approach would almost work in the hard potential case except there is one wrinkle, the moment loss issue! This issue was already present in the local problem and forced us to use a weighted $L^2$ approach as was outlined in point \ref{1} above. Since we crucially use the fact that we have at least two derivatives hitting $\a$, we cannot use an $L^\infty$ apporach as in \cite{Lu18}. 
%
%
%

Thus we need to be able to squeeze enough time decay to close the estimates, whilst only using an $L^2$ based approach. First note that a sufficiently localized in $x$ and $v$ solution  $f_{\text{free}}$ to the transport equation \eref{lin_transport} satisfies  for all $l\in \N\cup \{0\}$,

\begin{equation}\label{e.lin_L2}
\norm{\jap{v}^l\part^\alpha_x\part^\beta_v f_{\text{free}}}_{L^2_xL^1_v}\lesssim (1+t)^{-\frac{3}{2}+|\beta|},\hspace{1em} \norm{\jap{v}^l\part^\alpha_x\part^\beta_v f_{\text{free}}}_{L^2_xL^2_v}\lesssim (1+t)^{|\beta|}.
\end{equation}

Indeed, since $f_{\text{free}}=f_{\text{data}}(x-tv,v)$ we have by Cauchy-Schwarz, 
\begin{align*}
\left(\int_{\R^3}\left(\int_{\R^3}\jap{v}^lf_{\text{free}}\d v\right)^2\d x\right)^{\frac{1}{2}}&\lesssim\left(\int_{\R^3}\int_{\R^3}\jap{v}^{2l+4}f^2_{\text{free}}\d v\d x\right)^{\frac{1}{2}}\\
&=\left(\int_{\R^3}\int_{\R^3}\jap{v}^{2l+4}f^2_{\text{data}}(x-tv,v)\d v\d x\right)^{\frac{1}{2}}.
\end{align*}
Now if $f_{\text{data}}$ is sufficiently localized in $x$ and $v$, we have that
$$\int_{\R^3}\int_{\R^3}\jap{v}^{2l+4} f^2_{\text{data}}(x-tv,v)\d v\d x\lesssim (1+t)^{-3}.$$

Next, note that each $\part_v$ worsens the decay estimate by $t$ but $\part_x$ leaves the decay estimate untouched. Thus, 
$$\int_{\R^3}\int_{\R^3}\jap{v}^{2l+4} (\part^\alpha_x\part^\beta_vf)^2_{\text{data}}(x-tv,v)\d v\d x\lesssim (1+t)^{-3+2|\beta|},$$
and hence \eref{lin_L2} follows.

We now motivate the need for optimally exploiting null structure in our case. Assume that the solution $f$ to \eref{landau} satisfies the linear decay estimate, \eref{lin_L2}. We will overlook the velocity weight issue in the sequel, having already taken care of it using a weighted hierarchy as in point (\ref{1}) above.\\
First consider the term, $\norm{\part^2_{v_iv_j}\a f^2}_{L^1_xL^1_v}$, which arises due to integration by parts. The linear decay as in \eref{lin_L2} is enough to deal with the term . Indeed, using \eref{lin_L2} and Sobolev embedding, we have that,
$$|\part^2_{v_iv_j}\a|\lesssim (1+t)^{-\frac{3}{2}},$$
see \pref{pointwise_estimates_a} for more details. 

But now consider the term $$\norm{\part^{\alpha}_x \a \part^2_{v_iv_j} f\part^\alpha_x f}_{L^1_xL^1_v},$$
for $|\alpha|\geq 2$. Again by \eref{lin_L2} and Sobolev embedding we have that $$|\part^\alpha_x\a|\lesssim (1+t)^{-\frac{3}{2}}.$$ But \eref{lin_L2} dictates that $$\norm{\part^2_{v_iv_j}f}\lesssim (1+t)^{2}.$$
The above imply that $\norm{\part^{\alpha}_x \a \part^2_{v_iv_j} f\part^\alpha_x f}_{L^1_xL^1_v}$ actually grows like $(1+t)^{\frac{1}{2}}$! In comparison to \cite{Lu18}, where the linear decay estimate without the null structure gave a borderline non-integrable decay of $(1+t)^{-1}$, our linear estimate does not even decay. However, in \cite{Lu18}, both the decay estimate and the exploitation of the null structure rely heavily on $L^\infty$ estimate. In our case, as we explained in point (\ref{1}), we have at least two derivatives hitting $\a$. Our weighted hierarchy crucially uses this to overcome the moment loss issue. Hence, we need to come up with another weighted hierarchy that lets us exploit the null structure which is suited to our $L^2$ approach; see point (\ref{3}) for more details on this hierarchy. 
\item \textbf{Null structure and hierarchy of weights:} \label{3} We review the null structure that was recognized by Luk in \cite{Lu18}. For a sufficiently localized and regular data, the decay estimates \eref{lin_L2} are sharp only when $\frac{x}{t}\sim v$. For $\left|\frac{x}{t}-v\right|\geq t^{-\alpha}$ with $\alpha\in (0,1)$, we have better time decay.

This observation helps us leverage some extra decay due to the structure of $a_{ij}$. There are three scenarios possible:
\begin{itemize}
\item $\frac{x}{t}$ is not too close to $v$.
\item $\frac{x}{t}$ is not too close to $v_*$.
\item $|v-v_*|$ is small.
\end{itemize}  
For the first two cases we get extra decay by the observation above. For the last case we get extra decay for $\a$ as $|v-v_*|$ is small. This null structure is captured by using $\jap{x-(t+1)v}$ weights in the norm as in \cite{Lu18}. Following \cite{Lu18}, we use three vector fields namely $\part_v, \part_x$ and $Y=(t+1)\part_x+\part_v$. $Y$ helps us exploit the null structure in conjunction with a hierarchy of weights and we can write $\part_x=(t+1)^{-1}(Y-\part_v)$ to gain time decay for $\part_x$. 

As we already pointed out in points (\ref{1}) and (\ref{2}), we have two major issues in treating the global problem for hard potentials. One is the moment loss issue, which we overcome by using a hierarchy of $\jap{v}$ weights and the other is time decay issue which we overcome by using a hierarchy of $\jap{x-(t+1)v}$ weights. The idea to use a hierarchy of norms works since we always have at least two derivatives hitting $\a$, thanks to our $L^2$ approach. In addition, when we have two velocity derivatives on $\a$ then by points (\ref{1}) and (\ref{2}) above we not only have the required decay but also we can overcome the moment loss issue. A more important but somewhat trivial observation is that when we have $Y$ or $\part_x$ derivatives hitting $\a$ then we have one less of the corresponding derivative hitting $\part^2_{v_iv_j}g$  by the Leibnitz rule. 

To get decay when $Y$ hits $\a$, we use a hierarchy of $\jap{x-(t+1)v}$ weights which not only helps us get decay but also helps us take care of the moment loss issue. When $\part_x$ hits $\a$, we can get decay by writing $\part_x=(t+1)^{-1}(Y-\part_v)$ but this doesn't take care of the moment loss issue which is problematic in the hard potential case. But we note that this problem was already present in the local existence problem which was treated in \cite{Cha19}. Thus inspired from that work, we also define a hierarchy of $\jap{v}$ weights. Presence of three vector fields complicate our hierarchy but the core idea is the same as in \cite{Lu18} and \cite{Cha19}.
 
We now motivate our choice of weighted norm via the weights on $\jap{v}$. For the sake of discussion here, we again consider \eref{rough_eq_for_g}. First we differentiate \eref{rough_eq_for_g} by $\der$ to get
\begin{equation} \label{e.rough_diff_eq_for_g}
\begin{split}
&\partial_t \der g+v_i\partial_{x_i}\der g+\frac{\delta d_0}{(1+t)^{1+\delta}}\jap{v}\der g\\
&\quad=[\partial_t+v_i\partial_{x_i},\der ]g+\der (\a \part^2_{v_iv_j}g)-\der(\cm g)+\text{other terms}
\end{split}
\end{equation}
Then we integrate in velocity, space and time after multiplying by\break $\jap{v}^{2\nu_{\alpha,\beta,\sigma}}\jap{x-(t+1)v}^{2\omega_{\alpha,\beta,\sigma}}\der g$. There are a few cases that will illuminate the need of our choice of hierarchy:
\begin{itemize}
\item First we look at the case when we have a term of the form 
$$\int_0^T\int_x\int_v \der g\part_x^2\a\part^2_{v_iv_j}\derv{'}{}{}g.$$ Then we need each $\part_x$ to cost at least one $\jap{v}$. So an obvious candidate for $\nu_{\alpha,\beta,\sigma}=20-|\alpha|$.
\item But just as in \cite{Cha19}, we need $\nu_{\alpha,\beta,\sigma}$ to have some non-trivial dependence on $\beta$. This is because of the term $$\int_0^T\int_x\int_v \der g\part_x^9\a\part^2_{v_iv_j}\derv{'}{}{}g.$$ 

Since we need to estimate $\part^2_{v_iv_j}\derv{'}{}{}f$ in $L^\infty_xL^2_v$ and use Sobolev embedding to get to $L^2_xL^2_v$ at the cost of two spatial derivatives. So a more practical candidate is $\nu_{\alpha,\beta,\sigma}=20-\frac{3}{2}|\alpha|-\frac{1}{2}|\beta|$.
\item Finally the term of the form $$\int_0^T\int_x\int_v \der g(\part_x Y\a)\part^2_{v_iv_j}\derv{'}{}{'}g$$ forces us to incorporate $\sigma$ in our weights. Indeed, with $\tilde\nu_{\alpha,\beta,\sigma}=20-\frac{3}{2}|\alpha|-\frac{1}{2}|\beta|$, $$\tilde\nu_{\alpha',\beta',\sigma'}=\tilde\nu_{\alpha,\beta,\sigma}+\frac{1}{2}$$
where $|\beta'|=|\beta|+2$. Thus this hierarchy doesn't let us knock down the velocity growth of $\a$ adequately. But $\nu_{\alpha,\beta,\sigma}=20-\frac{3}{2}|\alpha|-\frac{1}{2}|\beta|-\frac{3}{2}|\sigma|$ does the trick and we claim that this hierarchy works for all other cases (see \sref{errors} for more details).
\end{itemize}

Using the terms $$\int_0^T\int_x\int_v\der g(Y^2\a)\part^2_{v_iv_j}\derv{}{}{'}g$$
and 
$$\int_0^T\int_x\int_v\der g(Y^9\a)\part^3_{v_iv_jv_l}g$$
we deduce in a similar way as above that $\omega_{\alpha,\beta,\sigma}$ must depend on both $\sigma$ and $\beta$.

Finally we remark that although a hierarchy that only depends on $\sigma$ and $\beta$ as $\tilde \omega_{\sigma,\beta}=20-\frac{3}{2}|\sigma|-\frac{1}{2}|\beta|$ is good enough to take care of the moment loss issue in conjuction with our definition of $\nu_{\alpha,\beta,\sigma}$. But it is not good enough to use the null structure optimally and doesn't give us the required time decay.

Indeed, if we use $\tilde \omega_{\sigma,\beta}$, in the term $$\int_0^T\int_x\int_v \der g(\part_x Y\a)\part^2_{v_iv_j}\derv{'}{}{'}g$$ 
we have for $|\beta'|=|\beta|+2$, $$\tilde \omega_{\alpha',\beta',\sigma'}=\tilde \omega_{\alpha,\beta,\sigma}+\frac{1}{2}.$$
Thus we are unable to optimally exploit the null structure. We conclude that our definition of weights must also depend on $\alpha.$ With this in mind, we claim that the following hierarchy works
$$\omega_{\alpha,\beta,\sigma}=20-\frac{3}{2}|\sigma|-\frac{1}{2}(|\alpha|+|\beta|).$$

Note that in this case, the dependence on $\beta$ and $\alpha$ has to be symmetric because of the commutator term arising due to the transport operator. Indeed, the first term on RHS of \eref{rough_diff_eq_for_g} is $$[\part_t+v_i\part_{x_i},\der] g=\derv{'}{'}{}g,$$
where $|\alpha'|=|\alpha|+1$ and $|\beta'|\leq |\beta|-1$. Thus we have a term of the form $$\norm{\jap{x-(t+1)v}^{\omega_{\alpha,\beta,\sigma}}\jap{v}^{\nu_{\alpha,\beta,\sigma}}\der g\cdot \derv{'}{'}{}g}_{L^2([0,T];L^2_xL^2_v)}.$$
So to make $\omega_{\alpha,\beta,\sigma}=\omega_{\alpha',\beta',\sigma}$, we need to make the dependence symmetric.

This commutator term also is the reason why we cannot do away with the vector field $\part_x$ altogether and get a global result by making appropriate changes to the local result in \cite{Cha19}. Indeed, if we used $\part_x=(t+1)^{-1}(Y-\part_v)$ then we would be generating $Y$ derivatives and our proposed hierarchy for $\jap{x-(t+1)v}$ weights would not work anymore.
\item \textbf{Discrepancy in linear time estimate and nonlinear time estimate:}\label{4} As was mentioned in point (\ref{2}), $\part_v$ costs $(1+t)$ at least at the linear estimate level; see \eref{lin_L2}. For soft potentials, Luk is able to propagate this linear estimate but in the hard potentials case we have an additional difficulty that prevents us from closing our argument with this estimate. That is we have to assume that each $\part_v$ costs us $(1+t)^{1+\delta}$ in the bootstrapping assumption instead of the linear estimate which costs $(1+t)$. The culprit is the interaction of our hierarchy of velocity weights and the term arising by commutation of $\part_v$ and the transport operator, i.e. the first term on the RHS of \eref{rough_diff_eq_for_g} (see \eref{T1} for more details).

Since we have a $\part_x$ instead of $\part_v$, the term $\derv{'}{'}{} g$ can handle one less $\jap{v}$ weight. Thus this error term needs to be taken care of by the term that arises by differentiating $e^{d(t)\jap{v}}$ by $\part_t$ (last term on lhs of \eref{rough_diff_eq_for_g}). But then we need an extra decay of $(1+t)^{-1-\delta}$ while the linear estimate only gives us $(1+t)^{-1}$ (since we have one less $\part_v$). Considering these, we bootstrap a time estimate that is weaker than the linear estimate. This implies that we get a slightly slower convergence to transport than we would have expected from the linear estimate\footnote{We don't claim this rate of convergence to be sharp.}.
\end{enumerate}
With these things in mind we define our energy norm for any $T>0$ as follows
\begin{equation*}
\begin{split}
\norm{h}^2_{E_T^m}:=&\sum \limits_{|\alpha|+|\beta|+|\sigma|\leq m} (1+T)^{-|\beta|(1+\delta)}\norm{\jap{x-(t+1)v}^{\omega_{\alpha,\beta,\sigma}}\jap{v}^{\nu_{\alpha,\beta,\sigma}}\der h}^2_{L^\infty([0,T);L^2_xL^2_v)}\\
&+\sum \limits_{|\alpha|+|\beta|+|\sigma|\leq m} (1+T)^{-|\beta|(1+\delta)}\norm{\jap{v}^{\frac{1}{2}}\jap{x-(t+1)v}^{\omega_{\alpha,\beta,\sigma}}\jap{v}^{\nu_{\alpha,\beta,\sigma}}\der h}^2_{L^2([0,T);L^2_xL^2_v)}.
\end{split}
\end{equation*}
\subsection{Paper organization} The remainder of the paper is structured as follows.\\
In \sref{notation}, we introduce some notations that will be in effect throughout the paper. In \sref{local} we review the local existence theorem in \cite{Cha19}. In \sref{bootstrap} we state the bootstrap assumptions and the theorem we wish to establish as a part of out bootstrap scheme. In \sref{coef} we establish bounds on the coefficients and their derivatives. In \sref{set-up}, we set up the energy estimates and figure out the error terms. In \sref{errors}, we estimate the errors and prove the bootstrap theorem from \sref{bootstrap}.
\subsection{Acknowledgements} I would like to thank Jonathan Luk for suggesting this problem and for numerous helpful discussions.
\section{Notations and spaces}\label{s.notation}
We introduce some notations that will be used throughout the paper.

\textbf{Norms}: We will use mixed $L^p$ norms, $1\leq p<\infty$ defined in the standard way:
$$\norm{h}_{L^p_v}:=(\int_{\R^3} |h|^p \d v)^{\frac{1}{p}}.$$
For $p=\infty$, define
$$\norm{h}_{L^\infty_v}:=\text{ess} \sup_{v\in \R^3}|h|(v).$$
For mixed norms, the norm on the right is taken first. For example,
$$\norm{h}_{L^p_xL^q_v}:=(\int_{\R^3}(\int_{\R^3} |h|^q(x,v) \d v)^{\frac{p}{q}} \d x)^{\frac{1}{p}}$$
and
$$\norm{h}_{L^r([0,T];L^p_xL^q_v)}:=(\int_0^T(\int_{\R^3}(\int_{\R^3} |h|^q(x,v) \d v)^{\frac{p}{q}} \d x)^{\frac{r}{p}})^{\frac{1}{r}}$$
with obvious modifications when $p=\infty$, $q=\infty$ or $r=\infty$. All integrals in phase space are over either $\R^3$ or $\R^3\times \R^3$ and the explicit dependence is dropped henceforth.\\

\textbf{Japanese brackets}. Define $$\jap{\cdot}:=\sqrt{1+|\cdot|^2}.$$

\textbf{Multi-indices}. Given a multi-index $\alpha=(\alpha_1,\alpha_2,\alpha_3)\in (\N\cup \{0\})^3$, we define $\part_x^\alpha=\part^{\alpha_1}_{x_1}\part^{\alpha_2}_{x_2}\part^{\alpha_3}_{x_3}$ and similarly for $\part^\beta_v$. Let $|\alpha|=\alpha_1+\alpha_2+\alpha_3$. Multi-indices are added according to the rule that if $\alpha'=(\alpha'_1,\alpha'_2,\alpha'_3)$ and $\alpha''=(\alpha''_1,\alpha''_2,\alpha''_3)$, then $\alpha'+\alpha''=(\alpha'_1+\alpha''_1,\alpha'_2+\alpha''_2,\alpha'_3+\alpha''_3)$.\\

\textbf{Velocity weights} We define the velocity weight function that shows up in the energy estimates. Let $|\alpha|+|\beta|+|\sigma|\leq 10$, then we define $$\nu_{\alpha,\beta,\sigma}=20-\frac{3}{2}(|\alpha|+|\sigma|)-\frac{1}{2}|\beta|.$$

\textbf{Space dependent weights} We now define the space dependent weight function that helps us gain time decay.  Let $|\alpha|+|\beta|+|\sigma|\leq 10$, then we define $$\omega_{\alpha,\beta,\sigma}=20-\frac{3}{2}|\sigma|-\frac{1}{2}(|\beta|+|\alpha|).$$

\textbf{Global energy norms}. For any $T>0$ the energy norm we use in $[0,T)\times\R^3\times\R^3$ is as follows 
\begin{equation}\label{e.Energy_norm}
\begin{split}
\norm{h}^2_{E_T^m}:=&\sum \limits_{|\alpha|+|\beta|+|\sigma|\leq m} (1+T)^{-|\beta|(1+\delta)}\norm{\jap{x-(t+1)v}^{\omega_{\alpha,\beta,\sigma}}\jap{v}^{\nu_{\alpha,\beta,\sigma}}\der h}^2_{L^\infty([0,T);L^2_xL^2_v)}\\
&+\sum \limits_{|\alpha|+|\beta|+|\sigma|\leq m} (1+T)^{-|\beta|(1+\delta)}\norm{\jap{v}^{\frac{1}{2}}\jap{x-(t+1)v}^{\omega_{\alpha,\beta,\sigma}}\jap{v}^{\nu_{\alpha,\beta,\sigma}}\der h}^2_{L^2([0,T);L^2_xL^2_v)}.
\end{split}
\end{equation}
When \underline{$m=10$}, it is dropped from the superscript, i.e. $E_T=E^{10}_T.$

It will also be convenient to define some other energy type norms. Namely,
$$\norm{h}_{Y^m_v}^2(t,x):=\sum \limits_{|\alpha|+|\beta|+|\sigma|\leq m} \norm{\jap{x-(t+1)v}^{\omega_{\alpha,\beta,\sigma}}\jap{v}^{\nu_{\alpha,\beta,\sigma}}\der h}^2_{L^2_v}(t,x),$$
$$\norm{h}^2_{Y^m_{x,v}}(t):=\sum \limits_{|\alpha|+|\beta|+|\sigma|\leq m} \norm{\jap{x-(t+1)v}^{\omega_{\alpha,\beta,\sigma}}\jap{v}^{\nu_{\alpha,\beta,\sigma}}\der h}^2_{L^2_xL^2_v}(t),$$
and 
$$\norm{h}^2_{Y^m_{T}}:=\sum \limits_{|\alpha|+|\beta|+|\sigma|\leq m} \norm{\jap{x-(t+1)v}^{\omega_{\alpha,\beta,\sigma}}\jap{v}^{\nu_{\alpha,\beta,\sigma}}\der h}^2_{L^\infty([0,T];L^2_xL^2_v)}.$$

Finally we define the weaker versions of the energy spaces for $m\leq 4$ that are used in the uniqueness part of \tref{global}.
\begin{equation*}
\begin{split}
\norm{h}^2_{\tilde E_T^m}:=&\sum \limits_{|\alpha|+|\beta|+|\sigma|\leq m} (1+T)^{-|\beta|(1+\delta)}\norm{\jap{x-(t+1)v}^{\tilde \omega_{\alpha,\beta,\sigma}}\jap{v}^{\tilde \nu_{\alpha,\beta,\sigma}}\der h}^2_{L^\infty([0,T);L^2_xL^2_v)}\\
&+\sum \limits_{|\alpha|+|\beta|+|\sigma|\leq m} (1+T)^{-|\beta|(1+\delta)}\norm{\jap{v}^{\frac{1}{2}}\jap{x-(t+1)v}^{\tilde \omega_{\alpha,\beta,\sigma}}\jap{v}^{\tilde \nu_{\alpha,\beta,\sigma}}\der h}^2_{L^2([0,T);L^2_xL^2_v)}
\end{split}
\end{equation*}
and 
$$\norm{h}^2_{\tilde Y^m_{x,v}}(t):=\sum \limits_{|\alpha|+|\beta|+|\sigma|\leq m} \norm{\jap{x-(t+1)v}^{\tilde \omega_{\alpha,\beta,\sigma}}\jap{v}^{\tilde \nu_{\alpha,\beta,\sigma}}\der h}^2_{L^2_xL^2_v}(t),$$
where $\tilde\omega_{\alpha,\beta}=10-\frac{3}{2}|\sigma|-\frac{1}{2}(|\beta|+|\alpha|)$ and $\tilde \nu_{\alpha,\beta,\sigma}=20-\frac{3}{2}(|\alpha|+|\sigma|)-\frac{1}{2}|\beta|.$

For two quantitites, $A$ and $B$ by $A\lesssim B$, we mean $A\leq C(d_0,\gamma)B$, where $C(d_0,\gamma)$ is a positive constant depending only on $d_0$ and $\gamma$.
\section{Local Existence}\label{s.local}
In this section, we recall the local existence result in \cite{Cha19} (and state a small varaint of it).
\begin{definition}
Define the $X^{k,l}_{x,v}$ and the $\tilde X^{k,l}_{x,v}$ norm on $\mathcal{S}(\R^3\times \R^3)$ by 
$$\norm{h}_{X^{k,l}_{x,v}}:=\sum_{|\alpha|+|\beta|\leq k}(\int_{\R^3}\int_{\R^3}\jap{v}^{40-3|\alpha|-|\beta|}\jap{v}^l(\part^\alpha_x\part^\beta_vh)^2\d v\d x)^{\frac{1}{2}}$$
and 
$$\norm{h}_{\tilde X^{k,l}_{x,v}}:=\sum_{|\alpha|+|\beta|\leq k}(\int_{\R^3}\int_{\R^3}\jap{v}^{20-3|\alpha|-|\beta|}\jap{v}^l(\part^\alpha_x\part^\beta_vh)^2\d v\d x)^{\frac{1}{2}}.$$
By abuse of notation we also let $X_{x,v}^{k,l}$ and  $\tilde{X}_{x,v}^{k,l}$ to be the completion of $\mathcal{S}(\R^3\times \R^3)$ under this norm.
\end{definition}
The following theorem is taken from \cite{Cha19}.
\begin{theorem}\label{t.local}
Fix $\gamma \in [0,1)$, $\rho_0>0$, $M_0\in \R$ and $f_{\ini}$ be such that\\
\begin{equation*}
\sum\limits_{|\alpha|+|\beta|\leq 10}\norm{\jap{v}^{20-\frac{3}{2}|\alpha|-\frac{1}{2}|\beta|}\part^\alpha_x\part_v^\beta(e^{d_0\jap{v}}f_{\ini})}_{L^2_xL^2_v}^2\leq M_0
\end{equation*}
Then for a large enough $\kappa\geq C(\rho_0,\gamma)M_0$, there exists $T=T_{\gamma,\rho_0,M_0,\kappa}>0$ such that there exists non-negative solution $f$ to \eref{landau} with initial data $f(0,x,v)=f_{\ini}(x,v)$ and satisfying 
$$\norm{e^{(\rho_0-\kappa t)\jap{v}}f}_{C^0([0,T];X^{10,0}_{x,v})\cap L^2([0,T];X^{10,1}_{x,v})}<\infty.$$

Moreover, $e^{(\rho_0-\kappa t)\jap{v}}f$ is unique in $C^0([0,T];\tilde X^{4,0}_{x,v})\cap L^2([0,T];\tilde X^{4,1}_{x,v}).$
\end{theorem}
For the global result, we will need a slight variant of \tref{local}. It can be proven in a very similar manner as in \cite{Cha19} and so we state it as a corollary and omit the proof.
\begin{corollary}\label{c.local_cor}
Fix $\gamma \in [0,1)$, $\rho_0>0$, $\lambda\in \R-\{0\}$, $M_0>0$, $N_{\alpha,\beta}\in \N$ be a hierarchy of weights with $N_{\alpha,\beta}\geq 2$ and $f_{\ini}$ be such that\\
\begin{equation*}
\sum\limits_{|\alpha|+|\beta|\leq 10}\norm{\jap{x-\lambda v}^{N_{\alpha,\beta}}\jap{v}^{20-\frac{3}{2}|\alpha|-\frac{1}{2}|\beta|}\part^\alpha_x\part_v^\beta(e^{d_0\jap{v}}f_{\ini})}_{L^2_xL^2_v}^2\leq M_0.
\end{equation*}
Then for a large enough $\kappa\geq C(\rho_0,\gamma)M_0$, there exists $T=T_{\gamma,\lambda,\rho_0,M_0,N,\kappa}>0$ such that there exists non-negative solution $f$ to \eref{landau} with initial data $f(0,x,v)=f_{\ini}(x,v)$ and satisfying 
\begin{equation}\label{e.local_est}
\begin{split}
\sum\limits_{|\alpha|+|\beta|\leq 10}&(\norm{\jap{x-(\lambda+t)v}^{N_{\alpha,\beta}}\jap{v}^{20-\frac{3}{2}|\alpha|-\frac{1}{2}|\beta|}\part^\alpha_x\part_v^\beta(e^{(d_0-\kappa t)\jap{v}}f)}_{C^0([0,T];L^2_xL^2_v)}\\
&+\norm{\jap{x-(\lambda+t)v}^{N_{\alpha,\beta}}\jap{v}^{\frac{1}{2}}\jap{v}^{20-\frac{3}{2}|\alpha|-\frac{1}{2}|\beta|}\part^\alpha_x\part_v^\beta(e^{(d_0-\kappa t)\jap{v}}f)}_{L^2([0,T];L^2_xL^2_v)})<\infty
\end{split}
\end{equation}

In addition, $e^{(\rho_0-\kappa t)\jap{v}}f$ is unique in $C^0([0,T];\tilde Y^{4,0}_{x,v})\cap L^2([0,T];\tilde Y^{4,1}_{x,v}).$

Moreover, given fixed $\gamma$, $\rho_0$, $M_0$, $N_{\alpha,\beta}$ and $\kappa$, for any compact interval $K\subset \R$, $T=T_{\gamma,\lambda,\rho_0,M_0,N_{\alpha,\beta},\kappa}>0$ can be chosen uniformly for all $\lambda\in K$
\end{corollary}
\begin{proof}[Comments on proof]
The proof is more or less the same as in \cite{Cha19}. The only difference being the extra weights $\jap{x-(t+\lambda)v}$. This affects the proof minimally as $$(\part_t+v_i\part_{x_i})\jap{x-(\lambda+t)v}=0.$$

Thus all terms other than the ones that arise due to integration by parts when $\part_v$ or $\part^2_{v_iv_j}$ hit $\jap{x-(t+\lambda)v}^N$ are unaffected. The following fact takes care of one of the issues, 
\begin{equation}\label{e.lamb}
|\part_{v_j}\jap{x-(\lambda+t)v}|\lesssim_\lambda 1.
\end{equation}

The other issue that how do we tame the velocity dependence of $\a$ or $\part_{v_i}\a$ when the velocity derivative hits $\jap{x-(t+\lambda)v}^{N_{\alpha,\beta}}$, seems more perilous. But now, we use the fact that $$|v-v_*|\leq (\lambda+t)^{-1}(|x-(\lambda+t)v|+|x-(\lambda+t)v_*|), \hspace{1em} |v-v_*|^2\leq (\lambda+t)^{-2}(|x-(\lambda+t)v|^2+|x-(\lambda+t)v_*|^2).$$

Thus, when $\part_{v_j}$ (or $\part^2_{v_iv_j}$ resp.) hits $\jap{x-(\lambda+t)v}^N$, we use the fact that we have one (or two resp.) less $\jap{x-(t+\lambda)v}$ weights to worry about. More precisely we have $$|\jap{x-(t+\lambda)v}^{-1}\part_{v_i}\a|\lesssim_\lambda \jap{v}^\gamma\int_{\R^3}\jap{v_*} \jap{x-(t+\lambda)v_*}f(v_*)\d v_*$$
and $$|\jap{x-(t+\lambda)v}^{-2}\part_{v_i}\a|\lesssim_\lambda \jap{v}^\gamma\int_{\R^3} \jap{x-(t+\lambda)v_*}^2f(v_*)\d v_*.$$

Since we have our $\lambda$ chosen away from $0$, we can close the energy estimates.

Finally for $\lambda\in K$ and $K\subset \R$ a compact interval, $T$ can be chosen to only depend on $K$ and not the specific value of $\lambda$. This is because the constant in \eref{lamb} can be chosen uniformly for all $\lambda\in K$.
\end{proof}
\section{Bootstrap assumption and the bootstrap theorem}\label{s.bootstrap}
Fix a $\delta>0$ be such that
\begin{equation}\label{e.delta}
\delta<\frac{1}{8}.
\end{equation}
Instead of directly controlling $f$, we define $g:=e^{d(t)\jap{v}}f$.
where $d(t)=d_0(1+(1+t)^{-\delta})$. Substituting $f=e^{-d(t)\jap{v}} g$ in \eref{landau}, we get
\begin{equation} \label{e.eq_for_g}
\begin{split}
\partial_tg+v_i\partial_{x_i}g+\frac{\delta d_0}{(1+t)^{1+\delta}}\jap{v}g&=\a [f]\partial^2_{v_iv_j}g-\cm [f]g-2d(t)\a [f]\frac{v_i}{\jap{v}}\partial_{v_j} g
\\&\quad-d(t)\left(\frac{\delta_{ij}}{\jap{v}}-\left(d(t)+\frac{1}{\jap{v}}\right)\frac{v_iv_j}{\jap{v}^2}\right)\a [f]g.
\end{split}
\end{equation}

Now we introduce the bootstrap assumption for $E$ norm, defined as in \eref{Energy_norm}.
\begin{equation}\label{e.boot_assumption_1}
E_T(g)\leq \eps^{\frac{3}{4}}
\end{equation}
Our goal from now until \sref{errors} would be to improve the bootstrap assumption \eref{boot_assumption_1} with $\eps^{\frac{3}{4}}$ replaced by $C\eps$ for some constant $C$ depending only on $d_0$ and $\gamma$. This is presented as a theorem below,
\begin{theorem}\label{t.boot}
Let $\gamma$, $d_0$ and $f_{\ini}$ be as in \tref{global} and let $\delta>0$ be as in \eref{delta}. There exists $\eps_0=\eps_0(d_0,\gamma)>0$ and $C_0=C_0(d_0,\gamma)>0$ with $C_0\eps_0\leq \frac{1}{2}\eps^{\frac{3}{4}}_0$ such that the following holds for $\eps \in [0,\eps_0]$:\\
Suppose there exists $T_{\boot}>0$ and a solution $f:[0,T_{\boot})\times\R^3\times\R^3$ with $f(t,x,v)\geq 0$ and $f(0,x,v)=f_{\ini}(x,v)$ such that the estimate \eref{boot_assumption_1} holds for $T\in[0,T_{\boot})$, then the estimate in fact holds for all $T\in[0,T_{\boot})$ with $\eps^{\frac{3}{4}}$ replaced by $C\eps.$
\end{theorem}
From now until \sref{errors}, we will prove \tref{boot}. \textbf{In succeeding sections, we always work under the assumptions of \tref{boot}}.
\section{Estimates for the coefficients}\label{s.coef}
\begin{proposition}\label{p.pointwise_estimates_a}
We assume $|\alpha|+|\beta|+|\sigma|\leq 10$.
The coefficient $\a$ and its higher derivatives satisfy the following pointwise bounds:
\begin{equation}\label{e.pw_bound_a}
\max_{i,j}|\derv{}{}{} \a|(t,x,v)\lesssim \int |v-v_*|^{2+\gamma}|\der f|(t,x,v_*)\d v_*,
\end{equation}
\begin{equation}\label{e.pw_bound_a_v}
\max_{j}\left|\der \left(\a \frac{v_i}{\jap{v}}\right)\right|\lesssim  \int |v-v_*|^{1+\gamma}\jap{v_*}|\der f|(t,x,v_*)\d v_*,
\end{equation}
\begin{equation}\label{e.pw_bound_a_2v}
\left|\der\left(\a \frac{v_iv_j}{\jap{v}^2}\right)\right|(t,x,v)\lesssim \jap{v}^\gamma\int\jap{v_*}^4 |\der f|(t,x,v_*)\d v_* .
\end{equation}
The first v-derivatives of $\a$ and the corresponding higher derivatives satisfy:
\begin{equation}\label{e.pw_bound_der_a}
\max_{i,j,k}\left|\der \part_{v_k}\a\right|(t,x,v)\lesssim \int |v-v_*|^{1+\gamma}|\der f|(t,x,v_*)\d v_*,
\end{equation}
\begin{equation}\label{e.pw_bound_der_a_v}
\max_{i,j,k}|\der \part_{v_k}\left(\a(t,x,v)\frac{v_i}{\jap{v}}\right)|\lesssim \jap{v}^\gamma\int \jap{v_*}^{2+\gamma}|\der f|(t,x,v_*)\d v_*.
\end{equation}
Finally, the second derivatives of $\a$ and its higher derivatives follow:
\begin{equation}\label{e.pw_bound_2der_a}
\max_{i,j,k,l}|\der \partial^2_{v_kv_l}\a|(t,x,v)\lesssim \int |v-v_*|^\gamma |\der f|(t,x,v_*)\d v_*.
\end{equation}
\end{proposition}
\begin{proof}
We use the following facts for convolutions with $|\beta'|\leq 2$\\
$$\der \part^{\beta'}_v \a[h]=\int (\part_v^{\beta'}a_{ij}(v-v_*))(\der h)(t,x,v_*) \d v_*$$
and 
$$\der \part^{\beta'}_v \left(\a[h]\frac{v_i}{\jap{v}}\right)=\int (\part_v^{\beta'}\left(a_{ij}\frac{v_i}{\jap{v}}\right)(v-v_*))(\der h)(t,x,v_*) \d v_*.$$
\emph{Proof of \eref{pw_bound_a}:} For this we just notice via the form of matrix $a$ in \eref{matrix_a} that $$|a_{ij}(v-v_*)|\leq |v-v_*|^{2+\gamma}.$$
\emph{Proof of \eref{pw_bound_a_v}:} For \eref{pw_bound_a_v} we use \eref{matrix_a} to get
\begin{equation}\label{e.int_pw_bound_a_v}
\begin{split}
a_{ij}(v-v_*)v_i&=|v-v_*|^{2+\gamma}\left(v_j-\frac{(v\cdot (v-v_*))(v-v_*)_j}{|v-v_*|^2}\right)\\
&=|v-v_*|^\gamma\left(v_j(v-v_*)_l(v-v_*)_l-v_l(v-v_*)_l (v-v_*)_j\right)\\
&=|v-v_*|^\gamma (v-v_*)_l(v_l(v_*)_j-v_j(v_*)_l).
\end{split}
\end{equation}
Thus we have using triangle inequality that $\left|a_{ij}(v-v_*)\frac{v_i}{\jap{v}}\right| \lesssim |v-v_*|^{1+\gamma}\jap{v_*}$.

\emph{Proof of \eref{pw_bound_a_2v}:} For \eref{pw_bound_a_2v}, we again begin with \eref{matrix_a} and see using triangle inequality\\
\begin{equation}
\begin{split}
a_{ij}(v-v_*)v_iv_j&=|v-v_*|^{2+\gamma}\left(|v|^2-\frac{(v\cdot (v-v_*))^2}{|v-v_*|^2}\right) \\
&=|v-v_*|^{\gamma}(|v|^2(|v|^2+|v_*|^2-2(v\cdot v_*))-|v|^4+2|v|^2(v\cdot v_*)-(v\cdot v_*)^2)\\
&=|v-v_*|^\gamma(|v|^2|v_*|^2-(v\cdot v_*)^2)\\
&\lesssim |v-v_*|^\gamma|v|^2|v_*|^2\\
&\lesssim \jap{v}^{2+\gamma}\jap{v_*}^{4+\gamma}.
\end{split}
\end{equation}
Hence we have that $\left|\frac{a_{ij}v_iv_j}{\jap{v}^2}\right|\lesssim \jap{v}^\gamma\jap{v_*}^{4+\gamma}$.

\emph{Proof of \eref{pw_bound_der_a}:} Using homogeneity we have
$$|\part_{v_k}a_{ij}|\lesssim |v-v_*|^{1+\gamma}$$ which implies \eref{pw_bound_der_a}.

\emph{Proof of \eref{pw_bound_der_a_v}:} To prove \eref{pw_bound_der_a_v}, we use \eref{int_pw_bound_a_v} to get
\begin{equation*}
\begin{split}
\part_{v_k}\left(a_{ij}(v-v_*)\frac{v_i}{\jap{v}}\right)&=\part_{v_k}\left(\frac{1}{\jap{v}}|v-v_*|^\gamma (v-v_*)_l(v_l(v_*)_j-v_j(v_*)_l)\right).
\end{split}
\end{equation*}

Using the product rule we get mutiple terms which we treat one by one now,\\
When $\partial_{v_k}$ hit $\frac{1}{\jap{v}}$ we get upto a constant $\frac{v_k}{\jap{v}^3}$, thus we have using triangle inequality $$\frac{v_k}{\jap{v}^3}|v-v_*|^\gamma (v-v_*)_l(v_l(v_*)_j-v_j(v_*)_l)\lesssim \jap{v}^\gamma\jap{v_*}^{2+\gamma}.$$
When $\partial_{v_k}$ hit $|v-v_*|^\gamma$ we get upto a constant $(v-v_*)_k|v-v_*|^{\gamma-2}$, hence using the fact that $\frac{(v-v_*)_k(v-v_*)l}{|v-v_*|^2}\lesssim 1$ and triangle inequality we get\\
$$\frac{1}{\jap{v}}|v-v_*|^\gamma\frac{(v-v_*)_k(v-v_*)l}{|v-v_*|^2}(v_l(v_*)_j-v_j(v_*)_l) \lesssim \jap{v}^\gamma\jap{v_*}^{\gamma+1}.$$
Finally when $\part_{v_k}$ hits $(v_l(v_*)_j-v_j(v_*)_l)$ we get $(\delta_{lk}(v_*)_j-\delta_{jk}(v_*)_l)$. Again, by triangle inequality we get 
$$\frac{1}{\jap{v}}|v-v_*|^\gamma (v-v_*)_l(\delta_{lk}(v_*)_j-\delta_{jk}(v_*)_l)\lesssim \jap{v}^\gamma\jap{v_*}^{2+\gamma}.$$

Hence in total we have $$\part_{v_k}\left(a_{ij}(v-v_*)\frac{v_i}{\jap{v}}\right) \lesssim \jap{v}^{\gamma}\jap{v_*}^{2+\gamma}.$$

\emph{Proof of \eref{pw_bound_2der_a}:} Finally, for the second derivatives of $a_{ij}$ we obtain by homogeneity
$$|\part_{v_l}\part_{v_k}a_{ij}(v-v_*)|\lesssim |v-v_*|^\gamma$$ which implies \eref{pw_bound_2der_a}.
\end{proof}
\begin{lemma}\label{l.L1_to_L2}
Let $h:[0,T_{\boot})\times \R^3\times \R^3$ be a smooth function. Then
$$\norm{h}_{L^1_v}(t,x)\lesssim (1+t)^{-\frac{3}{2}}\norm{\jap{x-(t+1)v}^2h(t,x,v)}_{L^2_v}.$$ 
\end{lemma}
\begin{proof}
We break the region of integration to obtain decay from the $\jap{x-(t+1)v}$ weights.
\begin{align*}
\int_{\R^3}|h|(t,x,v)\d v &\lesssim \int_{\R^3\cap\{|v-\frac{x}{t+1}|\leq (1+t)^{-1}\}}|h| \d v+\int_{\R^3\cap\{|v-\frac{x}{t+1}|\geq (1+t)^{-1}\}}|h| \d v\\
&\lesssim (\int_{\R^3\cap\{|v-\frac{x}{t+1}|\leq (1+t)^{-1}\}}h^2 \d v)^{\frac{1}{2}}(\int_{\R^3\cap\{|v-\frac{x}{t+1}|\leq (1+t)^{-1}\}}1 \d v)^{\frac{1}{2}}\\
&\quad+ (\int_{\R^3\cap\{|v-\frac{x}{t+1}|\geq (1+t)^{-1}\}}|v-\frac{x}{t+1}|^{-4}h^2 \d v)^{\frac{1}{2}}(\int_{\R^3\cap\{|v-\frac{x}{t+1}|\geq (1+t)^{-1}\}}|v-\frac{x}{t+1}|^{-4} \d v)^{\frac{1}{2}}\\
&\lesssim (1+t)^{-\frac{3}{2}}(\int_{\R^3}h^2\d v)^{\frac{1}{2}}+(1+t)^{-2}(\int_{\R^3}\jap{x-(t+1)v}^4h^2\d v)^{\frac{1}{2}}(1+t)^{\frac{1}{2}}\\
&\lesssim (1+t)^{-\frac{3}{2}}(\int_{\R^3}\jap{x-(t+1)v}^4h^2\d v)^{\frac{1}{2}}.
\end{align*}
\end{proof}
We now use \lref{L1_to_L2} to get $L^2_xL^1_v$ and $L^\infty_xL^1_v$ estimates.
\begin{lemma}\label{l.L2_x}
Let $h:[0,T_{\boot})\times \R^3\times \R^3$ be a smooth function. Then
$$\norm{h}_{L^2_xL^1_v}(t,x)\lesssim (1+t)^{-\frac{3}{2}}\norm{\jap{x-(t+1)v}^2h(t,x,v)}_{L^2_xL^2_v}.$$ 
\end{lemma}

\begin{lemma}\label{l.L_infty_x_L2_v}
Let $h:[0,T_{\boot})\times \R^3\times \R^3$ be a smooth function. Then
$$\norm{h}_{L^\infty_xL^2_v}(t,x)\lesssim \sum_{|\alpha|\leq 2}\norm{\part^\alpha_xh(t,x,v)}_{L^2_xL^2_v}$$ 
where $|\alpha|\leq 2$.
\end{lemma}
\begin{proof}
Using Sobolev embedding in $x$, we have that
$$\norm{h}_{L^\infty_x}(v)\lesssim \sum_{|\alpha|\leq 2}\norm{\part^\alpha_x h}_{L^2_x}(v).$$

Since this is true for all $v$, taking the $L^2_v$ norm and using Fubini for the RHS, we get
$$\norm{h}_{L^2_vL^\infty_x}\lesssim \sum_{|\alpha|\leq 2}\norm{\part^\alpha_x h}_{L^2_xL^2_v}.$$

Finally the following trivial inequality gives us the required lemma, $$\norm{h}_{L^\infty_xL^2_v}\lesssim \norm{h}_{L^2_vL^\infty_x}.$$
\end{proof}
\begin{lemma}\label{l.L_infty_x_L1_v}
Let $h:[0,T_{\boot})\times \R^3\times \R^3$ be a smooth function. Then
$$\norm{h}_{L^\infty_xL^1_v}(t,x)\lesssim \sum_{|\alpha|\leq 2}(1+t)^{-\frac{3}{2}}\norm{\jap{x-(t+1)v}^2\part^\alpha_xh(t,x,v)}_{L^2_xL^2_v}.$$ 
\end{lemma}
\begin{proof}
Again, by \lref{L1_to_L2}, we have
\begin{align*}
\norm{h}_{L^\infty_xL^1_v}&\lesssim (1+t)^{-\frac{3}{2}}\norm{\jap{x-(t+1)v}^2h}_{L^\infty_xL^2_v}.
\end{align*}
Now using \lref{L_infty_x_L2_v} we get,
$$\norm{h}_{L^\infty_xL^1_v}\lesssim \sum_{|\alpha|\leq 2}(1+t)^{-\frac{3}{2}}\norm{\part^\alpha_x(\jap{x-(t+1)v}^2h)}_{L^\infty_xL^2_v}.$$

Since $\part_{x_i} \jap{x-(t+1)v}^2=2(x_i-(t+1)v_i)\lesssim \jap{x-(t+1)v}$ and $\part^2_{x_ix_j} \jap{x-(t+1)v}^2=2\delta_{ij}$, we are done.
\end{proof}
\begin{lemma}\label{l.exp_bound}
For every $l\in \N$ and $m\in \N\cup\{0\}$, the following estimate holds with a constant depending on $l,\gamma$ and $d_0$ for any $(t,x,v) \in [0,T_{\boot})\times \R^3\times \R^3$
$$\jap{v}^l\jap{x-(t+1)v}^m|\der f|(t,x,v)\lesssim \sum \limits_{|\beta'|\leq |\beta|}\jap{x-(t+1)v}^m|\derv{}{'}{} g|(t,x,v).$$
\end{lemma}
\begin{proof}
This is immediate from differentiating $g$ and using that $\jap{v}^ne^{-d(t)\jap{v}}\lesssim_n 1$ for all $n\in \N$.
\end{proof}
\begin{lemma}\label{l.L1_v_bound_boot}
If $|\alpha|+|\beta|+|\sigma|\leq 8$, then $$\norm{\jap{v}^3\jap{x-(t+1)v}^{2}\der f}_{L^\infty_xL^1_v}\leq \eps^{\frac{3}{4}}(1+t)^{-\frac{3}{2}+|\beta|(1+\delta)}.$$
If $9\leq |\alpha|+|\beta|+|\sigma|\leq 10$, then $$\norm{\jap{v}^3\jap{x-(t+1)v}^{2}\der f}_{L^2_xL^1_v}\leq \eps^{\frac{3}{4}}(1+t)^{-\frac{3}{2}+|\beta|(1+\delta)}.$$
\end{lemma}
\begin{proof}
For $|\alpha|+|\beta|+|\sigma|\leq 8$ we have by \lref{L_infty_x_L1_v} and \lref{exp_bound}, we obtain
\begin{align*}
\norm{\jap{v}^3\jap{x-(t+1)v}^{2}\der f}_{L^\infty_xL^1_v}&\lesssim (1+t)^{-\frac{3}{2}}\norm{\jap{v}^3\jap{x-(t+1)v}^2\part^\alpha_x (\jap{x-(t+1)v}^{2}\der f)}_{L^2_xL^2_v}\\
&\lesssim  (1+t)^{-\frac{3}{2}}\norm{\jap{v}^3\jap{x-(t+1)v}^4\part^\alpha_x \der f}_{L^2_xL^2_v}\\
&\lesssim  (1+t)^{-\frac{3}{2}}\sum_{\substack{|\alpha'|\leq |\alpha|+2, |\beta'|\leq|\beta|\\ |\sigma'|\leq |\sigma|}}\norm{\jap{x-(t+1)v}^4 \derv{'}{'}{'} g}_{L^2_xL^2_v}.
\end{align*}

For $9\leq |\alpha|+|\beta|+|\sigma|\leq 10$, we use \lref{L2_x} to get $$\norm{\jap{v}^3\jap{x-(t+1)v}^{2}\der f}_{L^2_xL^1_v}\lesssim  (1+t)^{-\frac{3}{2}}\sum_{\substack{|\beta'|\leq|\beta|, |\sigma'|\leq |\sigma|}}\norm{\jap{x-(t+1)v}^4 \derv{'}{'}{'} g}_{L^2_xL^2_v}.$$
The conclusion now follows from \eref{boot_assumption_1}.
\end{proof}
We are finally in a position to be able to bound the various coefficients.
\begin{lemma}\label{l.a_bound}
If $|\alpha|+|\beta|+|\sigma|\leq 8$, then $$\norm{\jap{v}^{-2-\gamma}\der \a}_{L^\infty_xL^\infty_v}\leq \eps^{\frac{3}{4}}(1+t)^{-\frac{3}{2}+|\beta|(1+\delta)}.$$
If $9\leq |\alpha|+|\beta|+|\sigma|\leq 10$, then $$\norm{\jap{v}^{-2-\gamma}\der \a}_{L^2_xL^\infty_v}\leq \eps^{\frac{3}{4}}(1+t)^{-\frac{3}{2}+|\beta|(1+\delta)}.$$
\end{lemma}
\begin{proof}
Using \eref{pw_bound_a}, we have that,
$$|\der \a|(t,x,v)\lesssim \jap{v}^{2+\gamma}\int \jap{v_*}^{2+\gamma}\der f(t,x,v_*)\d v_*.$$
Thus, $$\norm{\jap{v}^{-2-\gamma}\der \a}_{L^\infty_xL^\infty_v}\lesssim \norm{ \jap{v}^{3}\der f}_{L^\infty_xL^1_v}$$
and $$\norm{\jap{v}^{-2-\gamma}\der \a}_{L^2_xL^\infty_v}\lesssim \norm{ \jap{v}^{3}\der f}_{L^2_xL^1_v}.$$
Now \lref{L1_v_bound_boot} gives us the desired result.
\end{proof}
\begin{lemma}\label{l.a_bound_dv}
Let $\underline{|\beta|\geq 1}$. If $|\alpha|+|\beta|+|\sigma|\leq 8$, then $$\norm{\jap{v}^{-1-\gamma}\der \a}_{L^\infty_xL^\infty_v}\leq \eps^{\frac{3}{4}}(1+t)^{-\frac{5}{2}-\delta+|\beta|(1+\delta)}.$$
If $9\leq |\alpha|+|\beta|+|\sigma|\leq 10$, then $$\norm{\jap{v}^{-1-\gamma}\der \a}_{L^2_xL^\infty_v}\leq \eps^{\frac{3}{4}}(1+t)^{-\frac{5}{2}-\delta+|\beta|(1+\delta)}.$$
\end{lemma}
\begin{proof}
Using \eref{pw_bound_der_a}, we have,
$$|\der \a|(t,x,v)\lesssim \jap{v}^{1+\gamma}\int \jap{v_*}^{1+\gamma}\derv{}{'}{} f(t,x,v_*)\d v_*$$
where $|\beta'|= |\beta|-1$.
Now we use \lref{L1_v_bound_boot} to get the desired result.
\end{proof}
\begin{lemma}\label{l.a_bound_d2v}
Let $\underline{|\beta|\geq 2}$. If $|\alpha|+|\beta|+|\sigma|\leq 8$, then $$\norm{\jap{v}^{-\gamma}\der \a}_{L^\infty_xL^\infty_v}\leq \eps^{\frac{3}{4}}(1+t)^{-\frac{7}{2}-2\delta+|\beta|(1+\delta)}.$$
If $9\leq |\alpha|+|\beta|+|\sigma|\leq 10$, then $$\norm{\jap{v}^{-\gamma}\der \a}_{L^2_xL^\infty_v}\leq \eps^{\frac{3}{4}}(1+t)^{-\frac{7}{2}-2\delta+|\beta|(1+\delta)}.$$
\end{lemma}
\begin{proof}
Using \eref{pw_bound_2der_a}, we have,
$$|\der \a|(t,x,v)\lesssim \jap{v}^{\gamma}\int \jap{v_*}^{1+\gamma}\derv{}{''}{} f(t,x,v_*)\d v_*$$
where $|\beta'|= |\beta|-2$.
Using \lref{L1_v_bound_boot} gets us the desired result.
\end{proof}
\begin{lemma}\label{l.a_bound_dx}
Let $\underline{|\alpha|\geq 1}$. If $|\alpha|+|\beta|+|\sigma|\leq 8$, then $$\norm{\jap{v}^{-2-\gamma}\der \a}_{L^\infty_xL^\infty_v}\leq \eps^{\frac{3}{4}}(1+t)^{-\frac{5}{2}+|\beta|(1+\delta)}.$$
If $9\leq |\alpha|+|\beta|+|\sigma|\leq 10$, then $$\norm{\jap{v}^{-2-\gamma}\a}_{L^2_xL^\infty_v}\leq \eps^{\frac{3}{4}}(1+t)^{-\frac{5}{2}+|\beta|(1+\delta)}.$$
\end{lemma}
\begin{proof}
We write $\part_x=(1+t)^{-1}(Y-\part_v)$. Now using \lref{a_bound} for the first term and \lref{a_bound_dv} for the second term, gives us the required lemma.
\end{proof}
\begin{lemma}\label{l.a_bound_d2x}
Let $\underline{|\alpha|\geq 2}$. If $|\alpha|+|\beta|+|\sigma|\leq 8$, then $$\norm{\jap{v}^{-2-\gamma}\der \a}_{L^\infty_xL^\infty_v}\leq \eps^{\frac{3}{4}}(1+t)^{-\frac{7}{2}+|\beta|(1+\delta)}.$$
If $9\leq |\alpha|+|\beta|+|\sigma|\leq 10$, then $$\norm{\jap{v}^{-2-\gamma}\der \a}_{L^2_xL^\infty_v}\leq \eps^{\frac{3}{4}}(1+t)^{-\frac{7}{2}+|\beta|(1+\delta)}.$$
\end{lemma}
\begin{proof}
\begin{align*}
|\der \a|&\lesssim (1+t)^{-2}\sum_{\substack{|\alpha'|=|\alpha|-2\\ |\sigma'|=|\sigma|+2}}|\derv{'}{}{'}\a|+(1+t)^{-2}\sum_{\substack{|\alpha'|=|\alpha|-2\\|\beta'|=|\beta|+1\\ |\sigma'|=|\sigma|+1}}|\derv{'}{}{'}\a|\\
&\quad+(1+t)^{-2}\sum_{\substack{|\alpha'|=|\alpha|-2\\ |\beta'|=|\beta|+2}}|\derv{'}{'}{}\a|.
\end{align*}
 Now using \lref{a_bound} for the first term, \lref{a_bound_dv} for the second term and \lref{a_bound_d2v} for the third term, gives us the desired result.
\end{proof}
\begin{lemma}\label{l.a_bound_Y}
If $|\alpha|+|\beta|+|\sigma|\leq 8$, then $$\norm{\jap{x-(1+t)v}^{-1}\jap{v}^{-1-\gamma}\der \a}_{L^\infty_xL^\infty_v}\leq \eps^{\frac{3}{4}}(1+t)^{-\frac{5}{2}+|\beta|(1+\delta)}.$$
If $9\leq |\alpha|+|\beta|+|\sigma|\leq 10$, then $$\norm{\jap{x-(1+t)v}^{-1}\jap{v}^{-1-\gamma}\der \a}_{L^2_xL^\infty_v}\leq \eps^{\frac{3}{4}}(1+t)^{-\frac{5}{2}+|\beta|(1+\delta)}.$$
\end{lemma}
\begin{proof}
Using \eref{pw_bound_a} in \pref{pointwise_estimates_a} and the triangle inequality, we get
\begin{align*}
|\der \a|(t,x,v)&\lesssim \int |v-v_*|^{2+\gamma}|\der f|(t,x,v_*)\d v_*\\
&\lesssim \int|v-v_*|^{1+\gamma}\left(\left|v-\frac{x}{t+1}\right|+\left|v_*-\frac{x}{t+1}\right|\right)|\der f|\d v_*.
\end{align*}
Thus we have,
\begin{align*}
&\max_{i,j}\norm{\jap{x-(t+1)v}^{-1}\jap{v}^{-1-\gamma}\der \a}_{L^p_xL^\infty_v}\\
&\lesssim (1+t)^{-1}\norm{\int\jap{v_*}^{2}\jap{x-(t+1)v_*}|\der f|(t,x,v_*)\d v_*}_{L^p_xL^\infty_v}.
\end{align*}
where $p=2$ or $p=\infty$, depending on the number of derivatives falling on $\a$.

We finish off by using \lref{L1_v_bound_boot}.
\end{proof}
\begin{lemma}\label{l.a_bound_2Y}
If $|\alpha|+|\beta|+|\sigma|\leq 8$, then $$\norm{\jap{x-(1+t)v}^{-2}\jap{v}^{-\gamma}\der \a}_{L^\infty_xL^\infty_v}\leq \eps^{\frac{3}{4}}(1+t)^{-\frac{7}{2}+|\beta|(1+\delta)}.$$
If $9\leq |\alpha|+|\beta|+|\sigma|\leq 10$, then $$\norm{\jap{x-(1+t)v}^{-2}\jap{v}^{-\gamma}\der \a}_{L^2_xL^\infty_v}\leq \eps^{\frac{3}{4}}(1+t)^{-\frac{7}{2}+|\beta|(1+\delta)}.$$
\end{lemma}
\begin{proof}
Using \eref{pw_bound_a} in \pref{pointwise_estimates_a} and the triangle inequality we have,
\begin{align*}
|\der \a|(t,x,v)&\lesssim \int |v-v_*|^{2+\gamma}|\der f|(t,x,v_*)\d v_*\\
&\lesssim \int|v-v_*|^{\gamma}\left(\left|v-\frac{x}{t+1}\right|^2+\left|v_*-\frac{x}{t+1}\right|^2\right)|\der f|\d v_*.
\end{align*}
Thus we have,
\begin{align*}
&\max_{i,j}\norm{\jap{x-(t+1)v}^{-2}\jap{v}^{-\gamma}\der \a}_{L^p_xL^\infty_v}\\
&\lesssim (1+t)^{-2}\norm{\int\jap{v_*}^{1}\jap{x-(t+1)v_*}^2|\der f|(t,x,v_*)\d v_*}_{L^p_xL^\infty_v}.
\end{align*}
where $p=2$ or $p=\infty$, depending on the number of derivatives falling on $\a$.
We finish off by using \lref{L1_v_bound_boot}.
\end{proof}
Now we collect some lemmas which bound the matrix $\der \a$ when we gain decay by mixing two different mechanisms. The proofs follow by appropriate combinations of \lref{a_bound_dv}, \lref{a_bound_dx} and \lref{a_bound_Y}.\\
\begin{lemma}\label{l.a_bound_dv_dx}
Let $\underline{|\alpha|\geq 1}$ and $\underline{|\beta|\geq 1}$. If $|\alpha|+|\beta|+|\sigma|\leq 8$, then $$\norm{\jap{v}^{-1-\gamma}\der \a}_{L^\infty_xL^\infty_v}\leq \eps^{\frac{3}{4}}(1+t)^{-\frac{7}{2}-\delta+|\beta|(1+\delta)}.$$
If $9\leq |\alpha|+|\beta|+|\sigma|\leq 10$, then $$\norm{\jap{v}^{-1-\gamma}\der \a}_{L^2_xL^\infty_v}\leq \eps^{\frac{3}{4}}(1+t)^{-\frac{7}{2}-\delta+|\beta|(1+\delta)}.$$
\end{lemma}
\begin{lemma}\label{l.a_bound_dv_Y}
Let $\underline{|\beta|\geq 1}$. If $|\alpha|+|\beta|+|\sigma|\leq 8$, then $$\norm{\jap{x-(t+1)v}^{-1}\jap{v}^{-\gamma}\der \a}_{L^\infty_xL^\infty_v}\leq \eps^{\frac{3}{4}}(1+t)^{-\frac{7}{2}-\delta+|\beta|(1+\delta)}.$$
If $9\leq |\alpha|+|\beta|+|\sigma|\leq 10$, then $$\norm{\jap{x-(t+1)v}^{-1}\jap{v}^{-\gamma}\der \a}_{L^2_xL^\infty_v}\leq \eps^{\frac{3}{4}}(1+t)^{-\frac{7}{2}-\delta+|\beta|(1+\delta)}.$$
\end{lemma}
\begin{lemma}\label{l.a_bound_dx_Y}
Let $\underline{|\alpha|\geq 1}$. If $|\alpha|+|\beta|+|\sigma|\leq 8$, then $$\norm{\jap{x-(t+1)v}^{-1}\jap{v}^{-1-\gamma}\der \a}_{L^\infty_xL^\infty_v}\leq \eps^{\frac{3}{4}}(1+t)^{-\frac{7}{2}+|\beta|(1+\delta)}.$$
If $9\leq |\alpha|+|\beta|+|\sigma|\leq 10$, then $$\norm{\jap{x-(t+1)v}^{-1}\jap{v}^{-1-\gamma}\der \a}_{L^2_xL^\infty_v}\leq \eps^{\frac{3}{4}}(1+t)^{-\frac{7}{2}+|\beta|(1+\delta)}.$$
\end{lemma}
\begin{lemma}\label{l.av_bound}
If $|\alpha|+|\beta|+|\sigma|\leq 8$, then $$\norm{\jap{v}^{-1-\gamma}\der \left(\a \frac{v_i}{\jap{v}}\right)}_{L^\infty_xL^\infty_v}\leq \eps^{\frac{3}{4}}(1+t)^{-\frac{3}{2}+|\beta|(1+\delta)}.$$
If $9\leq |\alpha|+|\beta|+|\sigma|\leq 10$, then $$\norm{\jap{v}^{-1-\gamma}\der \left(\a \frac{v_i}{\jap{v}}\right)}_{L^2_xL^\infty_v}\leq \eps^{\frac{3}{4}}(1+t)^{-\frac{3}{2}+|\beta|(1+\delta)}.$$
\end{lemma}
\begin{proof}
Using \eref{pw_bound_a_v} from \pref{pointwise_estimates_a} we get,
$$\max_{j}\left|\der \left(\a \frac{v_i}{\jap{v}}\right)\right|\lesssim  \int \jap{v}^{1+\gamma}\jap{v_*}^{2+\gamma}|\der f|(t,x,v_*)\d v_*.$$
Thus we get,
$$\norm{\jap{v}^{-1-\gamma}\der \left(\a \frac{v_i}{\jap{v}}\right)}_{L^p_xL^\infty_v}\lesssim \norm{\jap{v}^3\der f}_{L^p_xL^1_v}.$$
where, again, $p=2$ or $p=\infty$.\\
The required lemma now follows from \lref{L1_v_bound_boot}.
\end{proof}
\begin{lemma}\label{l.av_bound_dv}
Let $|\beta|\geq 1$. If $|\alpha|+|\beta|+|\sigma|\leq 8$, then $$\norm{\jap{v}^{-\gamma}\der \left(\a \frac{v_i}{\jap{v}}\right)}_{L^\infty_xL^\infty_v}\leq \eps^{\frac{3}{4}}(1+t)^{-\frac{5}{2}-\delta+|\beta|(1+\delta)}.$$
If $9\leq |\alpha|+|\beta|+|\sigma|\leq 10$, then $$\norm{\jap{v}^{-\gamma}\der \left(\a \frac{v_i}{\jap{v}}\right)}_{L^2_xL^\infty_v}\leq \eps^{\frac{3}{4}}(1+t)^{-\frac{5}{2}-\delta+|\beta|(1+\delta)}.$$
\end{lemma}
\begin{proof}
The proof is the same as \lref{a_bound_dv}, with the change that we use \eref{pw_bound_der_a_v}.
\end{proof}
\begin{lemma}\label{l.av_bound_dx}
Let $|\alpha|\geq 1$. If $|\alpha|+|\beta|+|\sigma|\leq 8$, then $$\norm{\jap{v}^{-1-\gamma}\der \left(\a \frac{v_i}{\jap{v}}\right)}_{L^\infty_xL^\infty_v}\leq \eps^{\frac{3}{4}}(1+t)^{-\frac{5}{2}+|\beta|(1+\delta)}.$$
If $9\leq |\alpha|+|\beta|+|\sigma|\leq 10$, then $$\norm{\jap{v}^{-1-\gamma}\der \left(\a \frac{v_i}{\jap{v}}\right)}_{L^2_xL^\infty_v}\leq \eps^{\frac{3}{4}}(1+t)^{-\frac{5}{2}+|\beta|(1+\delta)}.$$
\end{lemma}
\begin{proof}
The proof is the same as \lref{a_bound_dx} but we use \eref{pw_bound_a_v} instead of \eref{pw_bound_a}.
\end{proof}
\begin{lemma}\label{l.av_bound_Y}
If $|\alpha|+|\beta|+|\sigma|\leq 8$, then $$\norm{\jap{x-(t+1)v}^{-1}\jap{v}^{-\gamma}\der \left(\a \frac{v_i}{\jap{v}}\right)}_{L^\infty_xL^\infty_v}\leq \eps^{\frac{3}{4}}(1+t)^{-\frac{5}{2}+|\beta|(1+\delta)}.$$
If $9\leq |\alpha|+|\beta|+|\sigma|\leq 10$, then $$\norm{\jap{x-(t+1)v}^{-1}\jap{v}^{-\gamma}\der \left(\a \frac{v_i}{\jap{v}}\right)}_{L^2_xL^\infty_v}\leq \eps^{\frac{3}{4}}(1+t)^{-\frac{5}{2}+|\beta|(1+\delta)}.$$
\end{lemma}
\begin{proof}
The proof is the same as \lref{a_bound_Y}, with the change that we use \eref{pw_bound_a_v}.
\end{proof}
\begin{lemma}\label{l.a2v_bound}
If $|\alpha|+|\beta|+|\sigma|\leq 8$, then $$\norm{\jap{v}^{-\gamma}\der\left(\a \frac{v_iv_j}{\jap{v}^2}\right)}_{L^\infty_xL^\infty_v}\leq \eps^{\frac{3}{4}}(1+t)^{-\frac{3}{2}+|\beta|(1+\delta)}.$$
If $9\leq |\alpha|+|\beta|+|\sigma|\leq 10$, then $$\norm{\jap{v}^{-\gamma}\der \left(\a \frac{v_iv_j}{\jap{v}^2}\right)}_{L^2_xL^\infty_v}\leq \eps^{\frac{3}{4}}(1+t)^{-\frac{3}{2}+|\beta|(1+\delta)}.$$
\end{lemma}
\begin{proof}
The proof proceeds in the same way as \lref{a_bound} but instead of \eref{pw_bound_a}, we use \eref{pw_bound_a_2v}.
\end{proof}
\begin{lemma}\label{l.c_bound}
If $|\alpha|+|\beta|+|\sigma|\leq 8$, then $$\norm{\jap{v}^{-\gamma}\der \cm}_{L^\infty_xL^\infty_v}\leq \eps^{\frac{3}{4}}(1+t)^{-\frac{3}{2}+|\beta|(1+\delta)}.$$
If $9\leq |\alpha|+|\beta|+|\sigma|\leq 10$, then $$\norm{\jap{v}^{-\gamma}\der \cm}_{L^2_xL^\infty_v}\leq \eps^{\frac{3}{4}}(1+t)^{-\frac{7}{2}-2\delta+|\beta|(1+\delta)}.$$
\end{lemma}
\begin{proof}
Since we have $$|\der \cm|\lesssim \int |v-v_*|^\gamma f(t,x,v_*)\d v_*,$$
we get the desired conclusion using \lref{L1_v_bound_boot}.
\end{proof}
\section{Set-up for the energy estimates}\label{s.set-up}
Since we need to deal with higher derivatives as well, we apply $\der$ to \eref{eq_for_g} to get\\
\begin{equation} \label{e.diff_eq_for_g}
\begin{aligned}
&\partial_t \der g+v_i\partial_{x_i}\der g+\frac{\delta d_0}{(1+t)^{1+\delta}}\jap{v}\der g-\a\der g\\
&=\underbrace{[\partial_t+v_i\partial_{x_i},\der ]g}_{\text{Term } 1}+\underbrace{\frac{\delta d_0}{(1+t)^{1+\delta}}(\der(\jap{v}g)-\jap{v}\der  g)}_{\text{Term } 2}\\
&\quad+\underbrace{\der (\a \part^2_{v_iv_j}g)-\a\part^2_{v_iv_j}\der g}_{\text{Term } 3}-\underbrace{\der(\cm g)}_{\text{Term } 4}-\underbrace{2(d(t))\der \left(\a \frac{v_i}{\jap{v}}g\right)}_{\text{Term } 5}\\
&\quad-\underbrace{d(t)\der \left(\left(\frac{\delta_{ij}}{\jap{v}}-\left(d(t)+\frac{1}{\jap{v}}\right)\frac{v_iv_j}{\jap{v}^2}\right)\a g\right)}_{\text{Term } 6}.
\end{aligned}
\end{equation}
\begin{lemma}\label{l.energy_est_set_up}
Let $g:[0,T_{\boot})\times \R^3\times \R^3\to \R$ be a solution to \eref{eq_for_g}, then we have the following energy estimate\\
\begin{align*}
&\norm{\jap{v}^\nu\jap{x-(t+1)v}^\omega\der g}_{L^\infty([0,T];L^2_xL^2_v)}^2+\norm{(1+t)^{-\frac{1}{2}-\frac{\delta}{2}}\jap{v}^{\frac{1}{2}}\jap{v}^\nu\jap{x-(t+1)v}^\omega\der g}_{L^2([0,T];L^2_xL^2_v)}^2\\
&\lesssim \eps^2+\int_0^T\int \int \jap{v}^{2\nu}\jap{x-(t+1)v}^{2\omega}(\der g)J(t,x,v) \d v \d x\d t \\
&\quad+\mathcal{A}^{\alpha,\beta,\sigma}_1+\mathcal{A}^{\alpha,\beta,\sigma}_2+\mathcal{A}^{\alpha,\beta,\sigma}_3+\mathcal{A}^{\alpha,\beta,\sigma}_4+\mathcal{A}^{\alpha,\beta,\sigma}_5+\mathcal{A}^{\alpha,\beta,\sigma}_6
\end{align*}
where the second term is estimated in \lref{J_term} and is a combination of Terms 1-Terms 6 in \eref{diff_eq_for_g}.
\newpage
\begin{equation}\label{e.a1}
\mathcal{A}^{\alpha,\beta,\sigma}_1:=\norm{\jap{v}^{2\nu_{\alpha,\beta,\sigma}}\jap{x-(t+1)v}^{2\omega_{\alpha,\beta,\sigma}}\part^2_{v_iv_j}\a (\der g)^2}_{L^1([0,T];L^1_xL^1_v)},
\end{equation}
\begin{equation}\label{e.a2}
\mathcal{A}^{\alpha,\beta,\sigma}_2:=\norm{\jap{v}^{2\nu_{\alpha,\beta,\sigma}-1}\jap{x-(t+1)v}^{2\omega_{\alpha,\beta,\sigma}}\part_{v_i}\a (\der g)^2}_{L^1([0,T];L^1_xL^1_v)},
\end{equation}
\begin{equation}\label{e.a3}
\mathcal{A}^{\alpha,\beta,\sigma}_3:=\norm{(1+t)\jap{v}^{2\nu_{\alpha,\beta,\sigma}}\jap{x-(t+1)v}^{2\omega_{\alpha,\beta,\sigma}-1}\part_{v_i}\a (\der g)^2}_{L^1([0,T];L^1_xL^1_v)},
\end{equation}
\begin{equation}\label{e.a4}
\mathcal{A}^{\alpha,\beta,\sigma}_4:=\norm{\jap{v}^{2\nu_{\alpha,\beta,\sigma}-2}\jap{x-(t+1)v}^{2\omega_{\alpha,\beta,\sigma}}\a (\der g)^2}_{L^1([0,T];L^1_xL^1_v)},
\end{equation}
\begin{equation}\label{e.a5}
\mathcal{A}^{\alpha,\beta,\sigma}_5:=\norm{(1+t)\jap{v}^{2\nu_{\alpha,\beta,\sigma}-1}\jap{x-(t+1)v}^{2\omega_{\alpha,\beta,\sigma}-1}\a (\der g)^2}_{L^1([0,T];L^1_xL^1_v)},
\end{equation}
\begin{equation}\label{e.a6}
\mathcal{A}^{\alpha,\beta,\sigma}_6:=\norm{(1+t)^2\jap{v}^{2\nu_{\alpha,\beta,\sigma}}\jap{x-(t+1)v}^{2\omega_{\alpha,\beta,\sigma}-2}\a (\der g)^2}_{L^1([0,T];L^1_xL^1_v)}.
\end{equation}
\end{lemma}

\begin{proof}
Let $T$ be as in the statement of the lemma and take $T_*\in (0,T]$ to be arbitrary. The idea is to multiply \eref{diff_eq_for_g} by $\jap{x-(t+1)v}^{2\omega_{\alpha,\beta,\sigma}}\jap{v}^{2\nu_{\alpha,\beta,\sigma}}\der g$, integrate in $[0,T_*]\times \R^3\times \R^3$ and intgrate by parts. First note that we have 
$$\frac{\part}{\part t}+v_i\frac{\part}{\part x_i}(\jap{x-(t+1)v}^{2\omega})=0.$$
(Henceforth, we drop the dependence of the weight function on $\alpha,\beta$ and $\sigma$ for the remainder of the proof).

Hence performing the integration by parts in $t$ and $x$, we get
\begin{align}
&\frac{1}{2}\int \int \jap{v}^{2\nu}\jap{x-(T_*+1)v}^{2\omega}(\der g)^2(T_*,x,v)\d v\d x\nonumber\\
&\qquad-\frac{1}{2}\int \int  \jap{v}^{2\nu}\jap{x-v}^{2\omega}(\der g)^2(0,x,v) \d v \d x\\
&\quad+\int_0^T \int \int  \jap{v}^{2\nu}\jap{x-(t+1)v}^{2\omega}\frac{\delta d_0}{(1+t)^{1+\delta}} \jap{v} (\der g)^2(t,x,v)\d v \d x \d t\\\label{e.a_matrix_term}
 &\quad-\int_0^T \int \int \jap{v}^{2\nu}\jap{x-(t+1)v}^{2\omega}\der g\a\part^2_{v_i v_j}\der g\d v\d x\d t\\\label{e.H_term}
&=\int_0^T \int \int \jap{v}^{2\nu}\jap{x-(t+1)v}^{2\omega}(\der g)J(t,x,v)\d v \d x\d t.
\end{align}

Integrating by parts twice in $v$ for \eref{a_matrix_term} we get (for brevity we drop the integration signs)
\begin{equation*}
\begin{split}
\eref{a_matrix_term}&\equiv \jap{v}^{2\omega_{\alpha,\beta}}\jap{x-(t+1)v}^{2\omega} \part_{v_i}\der g(\a)\part_{v_j}\der g+\frac{1}{2}\jap{v}^{2\nu}\jap{x-(t+1)v}^{2\omega} \part_{v_i}\a\part_{v_j}((\der g)^2)\\
&\quad+\nu v_i \jap{v}^{2\nu-2}\jap{x-(t+1)v}^{2\omega}\a\part_{v_j}((\der g)^2)\\
&\quad-\omega (t+1)(x_i-(t+1)v_i)\jap{v}^{2\nu}\jap{x-(t+1)v}^{2\omega-2}\a\part_{v_j}((\der G)^2)\\
&\equiv \underbrace{\jap{v}^{2\nu}\jap{x-(t+1)v}^{2\omega} \part_{v_i}\der g(\a)\part_{v_j}\der g}_{A_1}- \underbrace{\frac{1}{2}\jap{v}^{2\nu}\jap{x-(t+1)v}^{2\omega}\part_{v_iv_j}^2\a(\der g)^2}_{A_2}\\
&\quad- \underbrace{\nu v_j\jap{v}^{2\nu-2}\jap{x-(t+1)v}^{2\omega}\part_{v_i}\a(\der g)^2}_{A_3}\\
&\quad+ \underbrace{\omega (t+1)(x_j-(t+1)v_j)\jap{v}^{2\nu}\jap{x-(t+1)v}^{2\omega-2}\part_{v_i}\a(\der g)^2}_{A_4}\\
&\quad- \underbrace{\nu\delta_{i,j} \jap{v}^{2\nu-2}\jap{x-(t+1)v}^{2\omega}\a(\der g)^2}_{A_5}- \underbrace{2\nu(\nu-1)v_i v_j\jap{v}^{2\nu-4}\jap{x-(t+1)v}^{2\omega}\a(\der g)^2}_{A_6}\\
&\quad+ \underbrace{2\nu\omega v_i (t+1)(x_j-(t+1)v_j) \jap{v}^{2\nu-2}\jap{x-(t+1)v}^{2\omega-2}(\a)(\der )g^2}_{A_7}\\
&\quad- \underbrace{\nu v_i \jap{v}^{2\nu-2}\jap{x-(t+1)v}^{2\omega}\part_{v_j}\a(\der g)^2}_{A_8}- \underbrace{\omega (t+1)^2 \delta_{ij}\jap{v}^{2\nu}\jap{x-(t+1)v}^{2\omega-2}(\a)(\der g)^2}_{A_9}\\
&\quad+ \underbrace{2\nu\omega (t+1)(x_i-(t+1)v_i)v_j\jap{v}^{2\nu-2}\jap{x-(t+1)v}^{2\omega-2}(\a)(\der g)^2}_{A_{10}}\\
&\quad- \underbrace{2\omega(\omega-1) (t+1)^2(x_i-(t+1)v_i)(x_j-(t+1)v_j)\jap{v}^{2\nu}\jap{x-(t+1)v}^{2\omega-4}(\a)(\der g)^2}_{A_{11}}\\
&\quad+ \underbrace{\omega (t+1) (x_i-tv_i)\jap{v}^{2\nu}\jap{x-(t+1)v}^{2\omega-2}\part_{v_j}\a(\der g)^2}_{A_{12}}.
\end{split}
\end{equation*}
Now we start bounding the various terms above.
First note that due to non-negativity of $f$, $A_1$ is non-negative and hence can be dropped from our analysis.
\begin{itemize}
\item $|A_2|\lesssim \mathcal{A}^{\alpha,\beta,\sigma}_1$.\\
\item $|A_3|+|A_8|\lesssim  \mathcal{A}^{\alpha,\beta,\sigma}_2$.\\
\item $|A_4|+|A_{12}|\lesssim  \mathcal{A}^{\alpha,\beta,\sigma}_3$.\\
\item $|A_5|+|A_6|\lesssim \mathcal{A}^{\alpha,\beta,\sigma}_4$.\\
\item $|A_7|+|A_{10}|\lesssim \mathcal{A}^{\alpha,\beta,\sigma}_5$.\\
\item $|A_{11}+|A_9|\lesssim \mathcal{A}^{\alpha,\beta,\sigma}_6$.
\end{itemize}
\end{proof}
\begin{lemma}\label{l.J_term}
For $|\alpha|+|\beta|\leq 10$, we have that\\
\begin{align*}
 \int_0^T \int \int \jap{v}^{2\nu}\jap{x-(t+1)v}^{2\omega}(\der g)J(t,x,v)\d v \d x\d t &\lesssim T_1^{\alpha,\beta,\sigma}+T_2^{\alpha,\beta,\sigma}+T^{\alpha,\beta}_{3,1}+T^{\alpha,\beta,\sigma}_{3,2}+T^{\alpha,\beta,\sigma}_{3,3}\\
&\quad+T^{\alpha,\beta,\sigma}_{3,4}+T^{\alpha,\beta,\sigma}_{4}+T^{\alpha,\beta,\sigma}_{5,1}+T^{\alpha,\beta,
\sigma}_{5,2}\\
&\quad+T^{\alpha,\beta,\sigma}_{5,3}+T^{\alpha,\beta,\sigma}_{6,1}+T^{\alpha,\beta,\sigma}_{6,2}.
\end{align*}
Such that,
\begin{equation}\label{e.T1}
T^{\alpha,\beta,\sigma}_1:=\sum_{\substack{|\alpha'|\leq |\alpha|+1\\ |\beta'|\leq |\beta|-1}}\norm{\jap{v}^{2\nu_{\alpha,\beta,\sigma}}\jap{x-(t+1)v}^{2\omega_{\alpha,\beta,\sigma}}|\der g|\cdot|\derv{'}{'}{}g|}_{L^1([0,T];L^1_xL^1_v)},
\end{equation}
\begin{equation}\label{e.T2}
T^{\alpha,\beta,\sigma}_2:=\sum_{\substack{|\beta'|\leq |\beta|, |\sigma'|\leq |\sigma|\\ |\beta'|+|\sigma'|\leq |\beta|+|\sigma|-1}}\norm{\frac{1}{(1+t)^{1+\delta}}\jap{v}^{2\nu_{\alpha,\beta,\sigma}}\jap{x-(t+1)v}^{2\omega_{\alpha,\beta,\sigma}}|\der g|\cdot|\derv{}{'}{'}g|}_{L^1([0,T];L^1_xL^1_v)},
\end{equation}
\begin{equation}\label{e.T3_1}
\begin{split}
T_{3,1}^{\alpha,\beta,\sigma}:=\sum_{\substack{|\alpha'|+|\alpha''|+|\alpha'''|\leq 2|\alpha|\\ |\beta'|+|\beta''|+|\beta'''|\leq 2|\beta|+2\\|\sigma'|+|\sigma''|+|\sigma'''|\leq 2|\sigma|\\ |\alpha'''|+|\beta'''|+|\sigma'''|=|\alpha|+|\beta|+|\sigma|\\  2\leq |\alpha'|+|\beta'|+|\sigma'|\leq \min\{8,|\alpha|+|\beta|+|\sigma|\}\\ |\alpha''|+|\beta''|+|\sigma''|\leq |\alpha|+|\beta|+|\sigma|}}&\norm{\jap{v}^{2\nu_{\alpha,\beta,\sigma}}\jap{x-(t+1)v}^{2\omega_{\alpha,\beta,\sigma}}|\derv{'''}{'''}{'''} g|\\
&\times|\derv{'}{'}{'} \a||\derv{''}{''}{''} g|}_{L^1([0,T];L^1_xL^1_v)},
\end{split}
\end{equation}
\begin{equation}\label{e.T3_2}
\begin{split}
T_{3,2}^{\alpha,\beta,\sigma}:=\sum_{\substack{|\alpha'|+|\alpha''|=|\alpha|\\ |\beta'|+|\beta''|=|\beta|+2\\|\sigma'|+|\sigma''|=|\sigma|\\ |\alpha'|+|\beta'|+|\sigma'|\geq 9}}&\norm{\jap{v}^{2\nu_{\alpha,\beta,\sigma}}\jap{x-(t+1)v}^{2\omega_{\alpha,\beta,\sigma}}|\der g|\\
&\times|\derv{'}{'}{'} \a||\derv{''}{''}{''} g|}_{L^1([0,T];L^1_xL^1_v)},
\end{split}
\end{equation}
\begin{equation}\label{e.T3_3}
\begin{split}
T_{3,3}^{\alpha,\beta,\sigma}:=\sum_{\substack{|\alpha'|+|\alpha''|+|\alpha'''|\leq 2|\alpha|\\ |\beta'|+|\beta''|+|\beta'''|\leq 2|\beta|+1\\|\sigma'|+|\sigma''|+|\sigma'''|\leq 2|\sigma|\\ |\alpha'''|+|\beta'''|+|\sigma'''|=|\alpha|+|\beta|+|\sigma|\\  |\alpha'|+|\beta'|+|\sigma'|=1\\ |\alpha''|+|\beta''|+|\sigma''|\leq |\alpha|+|\beta|+|\sigma|}}&\norm{\jap{v}^{2\nu_{\alpha,\beta,\sigma}-1}\jap{x-(t+1)v}^{2\omega_{\alpha,\beta,\sigma}}|\derv{'''}{'''}{'''} g|\\
&\times|\derv{'}{'}{'} \a||\derv{''}{''}{''} g|}_{L^1([0,T];L^1_xL^1_v)},
\end{split}
\end{equation}
\begin{equation}\label{e.T3_4}
\begin{split}
T_{3,4}^{\alpha,\beta,\sigma}:=\sum_{\substack{|\alpha'|+|\alpha''|+|\alpha'''|\leq 2|\alpha|\\ |\beta'|+|\beta''|+|\beta'''|\leq 2|\beta|+1\\|\sigma'|+|\sigma''|+|\sigma'''|\leq 2|\sigma|\\ |\alpha'''|+|\beta'''|+|\sigma'''|=|\alpha|+|\beta|+|\sigma|\\  |\alpha'|+|\beta'|+|\sigma'|=1\\ |\alpha''|+|\beta''|+|\sigma''|\leq |\alpha|+|\beta|+|\sigma|}}&\norm{(t+1)\jap{v}^{2\nu_{\alpha,\beta,\sigma}}\jap{x-(t+1)v}^{2\omega_{\alpha,\beta,\sigma}-1}|\derv{'''}{'''}{'''} g|\\
&\times|\derv{'}{'}{'} \a||\derv{''}{''}{''} g|}_{L^1([0,T];L^1_xL^1_v)},
\end{split}
\end{equation}
\begin{equation}\label{e.T4}
\begin{split}
T_4^{\alpha,\beta,\sigma}:=\sum_{\substack{|\alpha'|+|\alpha''|=|\alpha|\\ |\beta'|+|\beta''|=|\beta|\\|\sigma'|+|\sigma''|=|\sigma|}} \norm{\jap{v}^{2\nu_{\alpha,\beta,\sigma}}\jap{x-(t+1)v}^{2\omega_{\alpha,\beta,\sigma}}|\der g|\times|\derv{'}{'}{'} \cm||\derv{''}{''}{''} g|}_{L^1([0,T];L^1_xL^1_v)},
\end{split}
\end{equation}
\begin{equation}\label{e.T5_1}
\begin{split}
T_{5,1}^{\alpha,\beta,\sigma}:=\sum_{\substack{|\alpha'|+|\alpha''|=|\alpha|\\ |\beta'|+|\beta''|=|\beta|+1\\ 1\leq |\alpha'|+|\beta'|+|\sigma'|\leq |\alpha|+|\beta|+|\sigma|\\|\sigma'|+|\sigma''|=|\sigma|}}&\norm{\jap{v}^{2\nu_{\alpha,\beta,\sigma}}\jap{x-(t+1)v}^{2\omega_{\alpha,\beta,\sigma}}|\der g|\left|\derv{'}{'}{'}\left(\a \frac{v_i}{\jap{v}}\right)\right|\\
&\times|\derv{''}{''}{''}g|}_{L^1([0,T];L^1_xL^1_v)},
\end{split}
\end{equation}
\begin{equation}\label{e.T5_2}
T_{5,2}^{\alpha,\beta,\sigma}:=\max_j\norm{\jap{v}^{2\nu_{\alpha,\beta,\sigma}-1}\jap{x-(t+1)v}^{2\omega_{\alpha,\beta,\sigma}}|\der g|^2\left|\left(\a \frac{v_i}{\jap{v}}\right)\right|}_{L^1([0,T];L^1_xL^1_v)},
\end{equation}
\begin{equation}\label{e.T5_3}
T_{5,3}^{\alpha,\beta,\sigma}:=\max_j\norm{(t+1)\jap{v}^{2\nu_{\alpha,\beta,\sigma}}\jap{x-(t+1)v}^{2\omega_{\alpha,\beta,\sigma}-1}|\der g|^2\left|\left(\a \frac{v_i}{\jap{v}}\right)\right|}_{L^1([0,T];L^1_xL^1_v)},
\end{equation}
\begin{equation}\label{e.T6_1}
T_{6,1}^{\alpha,\beta,\sigma}:=\sum_{\substack{|\alpha'|+|\alpha''|\leq |\alpha|\\ |\beta'|+|\beta''|\leq |\beta|\\|\sigma'|+|\sigma''|\leq|\sigma|\\ |\alpha'|+|\beta'|+|\sigma'|\geq 1}} \norm{\jap{v}^{2\nu_{\alpha,\beta,\sigma}}\jap{x-(t+1)v}^{2\omega_{\alpha,\beta,\sigma}}|\der g|\left|\derv{'}{'}{'}\left(\frac{\bar{a}_{ii}}{\jap{v}}\right)\right||\derv{''}{''}{''} g|}_{L^1([0,T];L^1_xL^1_v)},
\end{equation}
\begin{equation}\label{e.T6_2}
T_{6,2}^{\alpha,\beta,\sigma}:=\sum_{\substack{|\alpha'|+|\alpha''|\leq |\alpha|\\ |\beta'|+|\beta''|\leq |\beta|\\|\sigma'|+|\sigma''|\leq|\sigma|}} \norm{\jap{v}^{2\nu_{\alpha,\beta,\sigma}}\jap{x-(t+1)v}^{2\omega_{\alpha,\beta,\sigma}}|\der g|\left|\derv{'}{'}{'}\left(\frac{\a v_iv_j}{\jap{v}^2} \right)\right||\derv{''}{''}{''} g|}_{L^1([0,T];L^1_xL^1_v)}.
\end{equation}
\end{lemma}
\begin{proof}
\emph{Term 1:} We have that $$[\part_t+v_i\part_{x_i},\der]g=\sum_{\substack{|\beta'|+|\beta''|=|\beta|\\ |\beta'|=1}}\part^{\beta'}_{x}\part^{\beta''}_{v}\part^\alpha_x Y^\sigma g.$$
Thus it can be bounded by $T^{\alpha,\beta,\sigma}_1$.

\emph{Term 2:} The commutator term arises from $\part_v$ or $t\part_x+\part_v$ acting on $\jap{v}$. Thus,
$$|\text{Term }2|\lesssim \sum_{\substack{|\beta'|\leq|\beta|,|\sigma'|\leq |\sigma|\\ |\beta'|+|\sigma'|\leq |\beta|+|\sigma|-1}}\frac{\delta d_0}{(1+t)^{1+\delta}}|\derv{}{'}{'} g|.$$
Hence the contribution can be bounded by $T^{\alpha,\beta,\sigma}_2$.

\emph{Term 3:} We need to perform integration by parts to handle this term.
We first consider the case $(|\alpha'|,|\beta'|,|\sigma'|)=(1,0,0)$. This this case, $\part^{\alpha'}_x=\part_{x_l}$ for some $l$.
\begin{align}
&\jap{v}^{2\nu}\jap{x-(t+1)v}^{2\omega}\der g(\part_{x_l}\a)\part^2_{v_iv_j}\derv{''}{}{} g \nonumber\\
&\equiv -\jap{v}^{2\nu}\jap{x-(t+1)v}^{2\omega}\part_{v_i}\part_{x_l}\derv{''}{}{} g(\part_{x_l}\a)\part_{v_j}\derv{''}{}{} g \nonumber\\
&\quad-2\nu v_i\jap{v}^{2\nu-2}\jap{x-(t+1)v}^{2\omega}\der g(\part_{x_l}\a)\part_{v_j}\derv{''}{}{} g \nonumber\\
&\quad+2\omega(t+1)(x_i-(t+1)v_i)\jap{v}^{2\nu}\jap{x-(t+1)v}^{2\omega-2}\der g(\part_{x_l}\a)\part_{v_j}\derv{''}{}{}g \nonumber \\
&\quad -\jap{v}^{2\nu}\jap{x-(t+1)v}^{2\omega}\der g(\part_{v_i}\part_{x_l}\a)\part_{v_j}\derv{''}{}{} g \nonumber\\
&\equiv \frac{1}{2}\jap{v}^{2\nu}\jap{x-(t+1)v}^{2\omega}\part_{v_i}\derv{''}{}{} g(\part^2_{x_l}\a)\part_{v_j}\derv{''}{}{} g \label{e.a_dx_1}\\
&\quad+2\omega(x_l-(t+1)v_l)\jap{v}^{2\nu}\jap{x-(t+1)v}^{2\omega-2}\part_{v_i}\derv{''}{}{} g(\part_{x_l}\a)\part_{v_j}\derv{''}{}{} g\label{e.a_dx_2}\\
&\quad-2\nu v_i\jap{v}^{2\nu-2}\jap{x-(t+1)v}^{2\omega}\der g(\part_{x_l}\a)\part_{v_j}\derv{''}{}{} g\label{e.a_dx_3}\\
&\quad+2\omega(t+1)(x_i-(t+1)v_i)\jap{v}^{2\nu}\jap{x-(t+1)v}^{2\omega-2}\der g(\part_{x_l}\a)\part_{v_j}\derv{''}{}{} g\label{e.a_dx_4}\\
& \quad-\jap{v}^{2\nu}\jap{x-(t+1)v}^{2\omega}\der g(\part_{v_i}\part_{x_l}\a)\part_{v_j}\derv{''}{}{} g\label{e.a_dx_5}.
\end{align}
We have, by definition, $|\eref{a_dx_1}|+|\eref{a_dx_5}|\lesssim T_{3,1}^{\alpha,\beta,\sigma}$ or $|\eref{a_dx_1}|+|\eref{a_dx_5}|\lesssim T_{3,2}^{\alpha,\beta,\sigma}$, $|\eref{a_dx_2}|+|\eref{a_dx_4}|\lesssim T_{3,4}^{\alpha,\beta,\sigma}$ and $|\eref{a_dx_3}|+|\eref{a_dx_5}|\lesssim T_{3,3}^{\alpha,\beta,\sigma}$.

Now we consider the case, $(|\alpha'|,|\beta'|,|\sigma'|)=(0,1,0)$. This this case, $\part^{\beta'}_v=\part_{v_l}$ for some $l$.
\begin{align}
&\jap{v}^{2\nu}\jap{x-(t+1)v}^{2\omega}\der g(\part_{v_l}\a)\part^2_{v_iv_j}\derv{}{''}{} g\nonumber \\
&\equiv -\jap{v}^{2\nu}\jap{x-(t+1)v}^{2\omega}\part_{v_i}\part_{v_l}\derv{}{''}{} g(\part_{v_l}\a)\part_{v_j}\derv{}{''}{} g\nonumber\\
&\quad-2\nu v_i\jap{v}^{2\nu-2}\jap{x-(t+1)v}^{2\omega}\der g(\part_{v_l}\a)\part_{v_j}\derv{}{''}{} g\nonumber\\
&\quad+2\omega(t+1)(x_i-(t+1)v_i)\jap{v}^{2\nu}\jap{x-(t+1)v}^{2\omega-2}\der g(\part_{v_l}\a)\part_{v_j}\derv{}{''}{} g\nonumber\\
&\quad -\jap{v}^{2\nu}\jap{x-(t+1)v}^{2\omega}\der g(\part^2_{v_iv_l}\a)\part_{v_j}\derv{}{''}{} g\nonumber\\
&\equiv \frac{1}{2}\jap{v}^{2\nu}\jap{x-(t+1)v}^{2\omega}\part_{v_i}\derv{}{''}{} g(\part^2_{v_l}\a)\part_{v_j}\derv{}{''}{} g \label{e.a_dv_1}\\
&\quad-(t+1)2\omega(x_l-(t+1)v_l)\jap{v}^{2\nu}\jap{x-(t+1)v}^{2\omega-2}\part_{v_i}\derv{}{''}{} g(\part_{v_l}\a)\part_{v_j}\derv{}{''}{} g\label{e.a_dv_2}\\
&\quad+2\nu v_l\jap{v}^{2\nu-2}\jap{x-(t+1)v}^{2\omega}\part_{v_i}\derv{}{''}{} g(\part_{v_l}\a)\part_{v_j}\derv{}{''}{} g\label{e.a_dv_3}\\
&\quad-2\nu v_i\jap{v}^{2\nu-2}\jap{x-(t+1)v}^{2\omega}\der g(\part_{v_l}\a)\part_{v_j}\derv{}{''}{} g\label{e.a_dv_4}\\
&\quad+2\omega(t+1)(x_i-(t+1)v_i)\jap{v}^{2\nu}\jap{x-(t+1)v}^{2\omega-2}\der g(\part_{v_l}\a)\part_{v_j}\derv{}{''}{} g\label{e.a_dv_5}\\
&\quad -\jap{v}^{2\nu}\jap{x-(t+1)v}^{2\omega}\der g(\part^2_{v_iv_l}\a)\part_{v_j}\derv{}{''}{} g\label{e.a_dv_6}.
\end{align}
Again, by definition, $|\eref{a_dv_1}|+|\eref{a_dv_6}|\lesssim T_{3,1}^{\alpha,\beta,\sigma}$ or $|\eref{a_dv_1}|+|\eref{a_dv_6}|\lesssim T_{3,2}^{\alpha,\beta,\sigma}$, $|\eref{a_dv_2}|+|\eref{a_dv_5}|\lesssim T_{3,4}^{\alpha,\beta,\sigma}$ and $|\eref{a_dv_3}|+|\eref{a_dv_4}|\lesssim T_{3,3}^{\alpha,\beta,\sigma}$.

Finally, we consider the case,  $(|\alpha'|,|\beta'|,|\sigma'|)=(0,0,1)$. This this case, $Y^{\sigma'}=Y_l=(t+1)\part_{x_l}+\part_{v_l}$ for some $l$.
\begin{align}
&\jap{v}^{2\nu}\jap{x-(t+1)v}^{2\omega}\der g(Y_l\a)\part^2_{v_iv_j}\derv{}{}{''} g\nonumber\\
&\equiv -\jap{v}^{2\nu}\jap{x-(t+1)v}^{2\omega}\part_{v_i}Y_l\derv{}{}{''} g(Y_l\a)\part_{v_j}\derv{}{}{''} g\nonumber\\
&\quad-2\nu v_i\jap{v}^{2\nu-2}\jap{x-(t+1)v}^{2\omega}\der g(Y_l\a)\part_{v_j}\derv{}{}{''} g\nonumber\\
&\quad+2\omega(t+1)(x_i-(t+1)v_i)\jap{v}^{2\nu}\jap{x-(t+1)v}^{2\omega-2}\der g(Y_l\a)\part_{v_j}\derv{}{}{''} g\nonumber\\
&\quad -\jap{v}^{2\nu}\jap{x-(t+1)v}^{2\omega}\der g(\part_{v_i}Y_l\a)\part_{v_j}\derv{}{}{''} g\nonumber\\
&\equiv \frac{1}{2}\jap{v}^{2\nu}\jap{x-(t+1)v}^{2\omega}\part_{v_i}\derv{}{}{''} g(Y^2_l\a)\part_{v_j}\derv{}{}{''} g \label{e.a_Y_1}\\
&\quad+2\nu v_l\jap{v}^{2\nu-2}\jap{x-(t+1)v}^{2\omega}\part_{v_i}\derv{}{}{''} g(Y_l\a)\part_{v_j}\derv{}{}{''} g\label{e.a_Y_2}\\
&\quad-2\nu v_i\jap{v}^{2\nu-2}\jap{x-(t+1)v}^{2\omega}\der g(Y_l\a)\part_{v_j}\derv{}{}{''} g\label{e.a_Y_3}\\
&\quad+2\omega(t+1)(x_i-(t+1)v_i)\jap{v}^{2\nu}\jap{x-(t+1)v}^{2\omega-2}\der g(Y_l\a)\part_{v_j}\derv{}{}{''} g\label{e.a_Y_4}\\
&\quad -\jap{v}^{2\nu}\jap{x-(t+1)v}^{2\omega}\der g(\part_{v_i}Y_l\a)\part_{v_j}\derv{}{}{''} g\label{e.a_Y_5}.
\end{align}
 $|\eref{a_Y_1}|+|\eref{a_Y_5}|\lesssim T_{3,1}^{\alpha,\beta,\sigma}$ or $|\eref{a_Y_1}|+|\eref{a_Y_5}|\lesssim T_{3,2}^{\alpha,\beta,\sigma}$, $|\eref{a_Y_4}|\lesssim T_{3,4}^{\alpha,\beta,\sigma}$ and $|\eref{a_dx_2}|+|\eref{a_dx_3}|\lesssim T_{3,3}^{\alpha,\beta,\sigma}$.

\emph{Term 4:} This term can be bounded by $T_4^{\alpha,\beta,\sigma}$.

\emph{Term 5:} We need to perform integration by parts when no derivatives fall on $\a\frac{v_i}{\jap{v}}$.
\begin{equation}\label{e.Term_5}
\begin{split}
&2(d(t))\der  \left(\a \frac{v_i}{\jap{v}}\part_{v_j}g\right)\\
&=2(d(t))\sum_{\substack{|\alpha'|+|\alpha''|=|\alpha|\\ |\beta'|+|\beta''|=|\beta|\\ |\alpha'|+|\beta'|+|\sigma'|\geq 1\\ |\sigma'|+|\sigma''|=|\sigma|}}\left( \derv{'}{'}{'}\left(\a \frac{v_i}{\jap{v}}\right)\right)(\part_{v_j}\derv{''}{''}{''}g)+2(d(t))\a \frac{v_i}{\jap{v}}\part_{v_j}\der g.
\end{split}
\end{equation}
For the second term we do integration by parts to get
\begin{align}
&2(d(t))\jap{v}^{2\nu}\jap{x-(t+1)v}^{2\omega}\der g\left(\a\frac{v_i}{\jap{v}}\right)\part_{v_j}\der g \nonumber\\
&\equiv -(d(t)) \jap{v}^{2\nu}\jap{x-(t+1)v}^{2\omega}\left(\part_{v_j}\left(\a \frac{v_i}{\jap{v}}\right)\right)(\der g)^2 \label{e.a_cont_v_1}\\
&\quad-2d(t)\nu v_j\jap{v}^{2\nu-2}\jap{x-(t+1)v}^{2\omega}\left(\a \frac{v_i}{\jap{v}}\right)(\der g)^2 \label{e.a_cont_v_2}\\
&\quad+2d(t)(t+1)(x_j-(t+1)v_j)\jap{v}^{2\nu}\jap{x-(t+1)v}^{2\omega-2}\left(\a[h] \frac{v_i}{\jap{v}}\right)(\der g)^2 \label{e.a_cont_v_3}.
\end{align}
\emph{Term 6:} For first part of Term 6 note that $\delta_{ij} \a=\bar{a}_{ii}$ and when no derivatives hit $\frac{\bar{a}_{ii}}{\jap{v}}$ then we have the term $$-d(t)\int_0^T\int \int \jap{v}^{2\nu}\jap{x-(t+1)v}^{2\omega}(\der g)^2 \frac{\bar{a}_{ii}}{\jap{v}}\d v \d x\d t.$$ Note, crucially that since $f\geq 0$, we have that $\frac{\bar{a}_{ii}}{\jap{v}}\geq 0$ implying that the whole integral is negative and thus can be dropped.

When at least one derivative fall on $\frac{\bar{a}_{ii}[h]}{\jap{v}}$ then we can bound it by $T^{\alpha,\beta}_{6,1}$.

Similarly the second part of Term 6 can be bounded by $T^{\alpha,\beta}_{6,2}$.
\end{proof}
\begin{remark}
Note that $\mathcal{A}^{\alpha,\beta,\sigma}_1\lesssim T^{\alpha,\beta,\sigma}_{3,1}$, $\mathcal{A}^{\alpha,\beta,\sigma}_2\lesssim T^{\alpha,\beta,\sigma}_{3,3}$ and $\mathcal{A}^{\alpha,\beta,\sigma}_3\lesssim T^{\alpha,\beta,\sigma}_{3,4}$. Thus it suffices to bound the $T^{\alpha,\beta,\sigma}_i$ terms, $\mathcal{A}^{\alpha,\beta,\sigma}_4$, $\mathcal{A}^{\alpha,\beta,\sigma}_5$ and $\mathcal{A}^{\alpha,\beta,\sigma}_6$.
\end{remark}
\section{Error estimates}\label{s.errors}
\begin{lemma}\label{l.L2_time_in_norm} 
For $T\in (0,T_{\boot}]$ and $|\alpha|+|\beta|+|\sigma|\leq 10$,
$$\norm{(1+t)^{-\frac{1}{2}-\delta-|\beta|(1+\delta)}\jap{v}^{\frac{1}{2}}\jap{v}^{\nu_{\alpha,\beta,\sigma}}\jap{x-(t+1)v}^{\omega_{\alpha,\beta,\sigma}}\der g}_{L^2([0,T];L^2_xL^2_v)}\lesssim \eps^{\frac{3}{4}}.$$
\end{lemma}
\begin{proof}
We can assume that $T>1$, otherwise the inequality is an immediate consequence of the bootstrap assumtion \eref{boot_assumption_1}.\\
We split the integration in time into dyadic intervals. More precisely, let $k=\lceil\log_2 T\rceil+1$. Define $\{T_i\}_{i=0}^k$ with $T_0<T_1<\dots<T_k$, where $T_0=0$ and $T_i=2^{i-1}$ where $i=1,\dots, k-1$ and $T_k=T$. Now by the bootstrap assumption \eref{boot_assumption_1}, for any $i=1,\dots,k$, 
$$\norm{(1+t)^{-\frac{1}{2}-\frac{\delta}{2}}\jap{v}^{\frac{1}{2}}\jap{v}^{\nu_{\alpha,\beta,\sigma}}\jap{x-(t+1)v}^{\omega_{\alpha,\beta,\sigma}}\der g}_{L^2([T_{i-1},T_i];L^2_xL^2_v)}\lesssim \eps^{\frac{3}{4}}2^{|\beta|(1+\delta)i}.$$
Thus,
\begin{align*}
&\norm{(1+t)^{-\frac{1}{2}-\delta-|\beta|(1+\delta)}\jap{v}^{\frac{1}{2}}\jap{v}^{\nu_{\alpha,\beta,\sigma}}\jap{x-(t+1)v}^{\omega_{\alpha,\beta,\sigma}}\der g}_{L^2([0,T];L^2_xL^2_v)}\\
&\lesssim (\sum_{i=1}^k\norm{(1+t)^{-\frac{1}{2}-\delta-|\beta|(1+\delta)}\jap{v}^{\frac{1}{2}}\jap{v}^{\nu_{\alpha,\beta,\sigma}}\jap{x-(t+1)v}^{\omega_{\alpha,\beta,\sigma}}\der g}^2_{L^2([T_{i-1},T_I];L^2_xL^2_v)})^{\frac{1}{2}}\\
&\lesssim (\sum_{i=1}^k2^{-2|\beta|(1+\delta)i-\delta i}\norm{(1+t)^{-\frac{1}{2}-\frac{\delta}{2}}\jap{v}^{\frac{1}{2}}\jap{v}^{\nu_{\alpha,\beta,\sigma}}\jap{x-(t+1)v}^{\omega_{\alpha,\beta,\sigma}}\der g}^2_{L^2([T_{i-1},T_I];L^2_xL^2_v)})^{\frac{1}{2}}\\
&\lesssim \eps^{\frac{3}{4}}(\sum_{i=1}^k 2^{-2|\beta|(1+\delta)i-\delta i}\cdot 2^{2|\beta|(1+\delta)i})^{\frac{1}{2}}=\eps^{\frac{3}{4}}(\sum_{i=1}^k 2^{-\delta i})^{\frac{1}{2}}\lesssim \eps^{\frac{3}{4}}.
\end{align*}
\end{proof}
\begin{proposition}\label{p.T1}
Let $|\alpha|+|\beta|+|\sigma|\leq 10$. Then for every $\eta>0$, there exists a constant $C_\eta>0$(depending on $\eta$ in addition to $d_0$ and $\gamma$) such that the term $T_1^{\alpha,\beta,\sigma}$ in \eref{T1} is bounded as follows for every $T\in[0,T_{\boot})$.\\
\begin{align*}
T^{\alpha,\beta,\sigma}_1&\leq \eta\norm{(1+t)^{-\frac{1+\delta}{2}}\jap{v}^{\frac{1}{2}}\jap{v}^{\nu_{\alpha,\beta,\sigma}}\jap{x-(t+1)v}^{\omega_{\alpha,\beta,\sigma}}\der g}_{L^2([0,T];L^2_xL^2_v)}^2\\
&\quad+C_\eta (1+T)^{2(1+\delta)}\sum_{\substack{|\alpha'|\leq|\alpha|+1\\ |\beta'|\leq |\beta|-1}} \norm{(1+t)^{-\frac{1+\delta}{2}}\jap{v}^{\frac{1}{2}}\jap{v}^{\nu_{\alpha',\beta',\sigma}}\jap{x-(t+1)v}^{\omega_{\alpha',\beta',\sigma}}\derv{'}{'}{} g}_{L^2([0,T];L^2_xL^2_v)}^2.
\end{align*}
\end{proposition}
\begin{proof}
Since we have that $|\alpha'|\leq |\alpha|+1$ and $|\beta'|\leq |\beta|-1$ we have that $\omega_{\alpha',\beta',\sigma}\geq \omega_{\alpha,\beta,\sigma}$ and $\nu_{\alpha'\beta'\sigma}\geq \nu_{\alpha,\beta,\sigma}-1$.\\
Thus we have\\
\begin{align*}
&\sum_{\substack{|\alpha'|\leq|\alpha|+1\\ |\beta'|\leq |\beta|-1}}\norm{\jap{v}^{2\nu_{\alpha,\beta,\sigma}}\jap{x-(t+1)v}^{\omega_{2\alpha,\beta,\sigma}}|\der g||\derv{'}{'}{} g|}_{L^1([0,T];L^1_xL^1_v)}\\
&\lesssim (1+T)^{1+\delta}\sum_{\substack{|\alpha'|\leq|\alpha|+1\\ |\beta'|\leq |\beta|-1}}\norm{(1+t)^{-\frac{1+\delta}{2}}\jap{v}^{\frac{1}{2}}\jap{v}^{\nu_{\alpha,\beta,\sigma}}\jap{x-(t+1)v}^{\omega_{\alpha,\beta,\sigma}}\der g}_{L^2([0,T];L^2_xL^2_v)}\\
&\quad\times \norm{(1+t)^{-\frac{1+\delta}{2}}\jap{v}^{\frac{1}{2}}\jap{v}^{\nu_{\alpha',\beta',\sigma}}\jap{x-(t+1)v}^{\omega_{\alpha',\beta',\sigma}}\derv{'}{'}{} g}_{L^2([0,T];L^2_xL^2_v)}.
\end{align*}
The conclusion now follows by Young's inequality.
\end{proof}
\begin{proposition}\label{p.T2}
Let $|\alpha|+|\beta|+|\sigma|\leq 10$. Then for every $\eta>0$, there exists a constant $C_\eta>0$(depending on $\eta$ in addition to $d_0$ and $\gamma$) such that the term $T_2^{\alpha,\beta,\sigma}$ in \eref{T2} is bounded as follows for every $T\in[0,T_{\boot})$
\begin{align*}
T^{\alpha,\beta,\sigma}_2\leq &\eta \norm{(1+t)^{-\frac{1}{2}-\frac{\delta}{2}}\jap{v}^{\nu_{\alpha,\beta,\sigma}}\jap{x-(t+1)v}^{\omega_{\alpha,\beta,\sigma}}\der g}^2_{L^2([0,T];L^2_xL^2_v)}\\
&+C_\eta\sum_{\substack{|\beta'|\leq|\beta|, |\sigma'|\leq |\sigma|\\ |\beta'|+|\sigma'|\leq |\beta|+|\sigma|-1}} \norm{(1+t)^{-\frac{1}{2}-\frac{\delta}{2}}\jap{v}^{\nu_{\alpha,\beta',\sigma'}}\jap{x-(t+1)v}^{\omega_{\alpha,\beta',\sigma'}}\derv{}{'}{'} g}^2_{L^2([0,T];L^2_xL^2_v)}.
\end{align*}
\end{proposition}
\begin{proof}
Under the conditions for $\beta'$ and $\sigma'$, we have that $\nu_{\alpha,\beta',\sigma'}\geq \nu_{\alpha,\beta,\sigma}$ and $\omega_{\alpha,\beta',\sigma'}\geq \omega_{\alpha,\beta,\sigma}$. Thus the required lemma is an easy consequence of the Cauchy-Schwartz inequality and Young's inequality
\end{proof}
\begin{proposition}\label{p.T3_1}
Let $|\alpha|+|\beta|+|\sigma|\leq 10$. Then the term $T^{\alpha,\beta,\sigma}_{3,1}$ is bounded as follows for all $T\in [0,T_{\boot})$,
$$T^{\alpha,\beta,\sigma}_{3,1}\lesssim \eps^2(1+T)^{2|\beta|(1+\delta)}.$$
\end{proposition}
\begin{proof}
We fix a typical term with $\alpha',\alpha'',\alpha''',\beta',\beta'',\beta''',\sigma',\sigma'',\sigma'''$ satisfying the required conditions. Note that there are less than $8$ deriviatives hitting $\a$ thus we estimate it in $L^\infty_x$. Further note that since there are atleast two derivatives hitting $\a$, there are $6$ different scenarios to be considered. Before we begin our case-by-case analysis we note some relations for the weights.
\begin{align*}
\nu_{\alpha'',\beta'',\sigma''}+\nu_{\alpha''',\beta''',\sigma'''}&=40-\frac{3}{2}(|\alpha''|+|\alpha'''|)-\frac{1}{2}(|\beta''|+|\beta'''|)-\frac{3}{2}(|\sigma''|+|\sigma'''|)\\
&\geq 40-\frac{3}{2}(2|\alpha|-|\alpha'|)-\frac{1}{2}(2|\beta|-|\beta'|+2)-\frac{3}{2}(2|\sigma|-|\sigma'|)\\
&=2\nu_{\alpha,\beta,\sigma}+\frac{3|\alpha'|+|\beta'|+3|\sigma'|}{2}-1.
\end{align*}
Similarly, $$\omega_{\alpha'',\beta'',\sigma''}+\omega_{\alpha''',\beta''',\sigma'''}\geq 2\omega_{\alpha,\beta,\sigma}+\frac{|\alpha'|+|\beta'|+3|\sigma'|}{2}-1.$$
\emph{Case 1:} $|\beta'|\geq 2$.\\
In this case we have that $\nu_{\alpha'',\beta'',\sigma''}+\nu_{\alpha''',\beta''',\sigma'''}\geq2\nu_{\alpha,\beta,\sigma}$ and  $\omega_{\alpha'',\beta'',\sigma''}+\omega_{\alpha''',\beta''',\sigma'''}\geq2\omega_{\alpha,\beta,\sigma}$.\\
We apply \lref{a_bound_d2v} and \lref{L2_time_in_norm} and H\"{o}lder's inequality to get
\begin{align*}
&\norm{\jap{v}^{2\nu_{\alpha,\beta,\sigma}}\jap{x-(t+1)v}^{2\omega_{\alpha,\beta,\sigma}}|\derv{'''}{'''}{'''} g||\derv{'}{'}{'} \a||\derv{''}{''}{''} g|}_{L^1([0,T];L^1_xL^1_v)}\\
&\lesssim (1+T)^{2|\beta|(1+\delta)}\norm{(1+t)^{-\frac{1}{2}-\delta-|\beta'''|(1+\delta)}\jap{v}^{\frac{1}{2}}\jap{v}^{\nu_{\alpha''',\beta''',\sigma'''}}\jap{x-(t+1)v}^{\omega_{\alpha''',\beta''',\sigma'''}}\derv{'''}{'''}{'''}g}_{L^2([0,T];L^2_xL^2_v)}\\
&\quad\times \norm{(1+t)^{1+2\delta-(|\beta'|-2)(1+\delta)}\jap{v}^{-1}\derv{'}{'}{'}\a}_{L^\infty([0,T];L^\infty_xL^\infty_v)}\\
&\quad\times \norm{(1+t)^{-\frac{1}{2}-\delta-|\beta''|(1+\delta)}\jap{v}^{\frac{1}{2}}\jap{v}^{\nu_{\alpha'',\beta'',\sigma''}}\jap{x-(t+1)v}^{\omega_{\alpha'',\beta'',\sigma''}}\derv{''}{''}{''}g}_{L^2([0,T];L^2_xL^2_v)}\\
&\lesssim (1+T)^{2|\beta|(1+\delta)}\eps^{\frac{3}{4}}\times(\sup_{t\in [0,T]}\eps^{\frac{3}{4}}(1+t)^{1+2\delta-(|\beta'|-2)(1+\delta)}(1+t)^{-\frac{7}{2}+|\beta'|(1+\delta)})\eps^{\frac{3}{4}}\\
&=\eps^{\frac{9}{4}}(1+T)^{2|\beta|(1+\delta)},
\end{align*}
where we used $\underline{4\delta<\frac{1}{2}}$ in the last line.\\
\emph{Case 2:} $|\beta'|=1$.\\
We break this into two furhther subcases.\\
\emph{Subcase 2a):} $|\alpha'|\geq 1$.

In this case we have  $\nu_{\alpha'',\beta'',\sigma''}+\nu_{\alpha''',\beta''',\sigma'''}\geq2\nu_{\alpha,\beta,\sigma}+1$ and  $\omega_{\alpha'',\beta'',\sigma''}+\omega_{\alpha''',\beta''',\sigma'''}\geq2\omega_{\alpha,\beta,\sigma}$.

We apply \lref{a_bound_dv_dx} and \lref{L2_time_in_norm}, to obtain
\begin{align*}
&\norm{\jap{v}^{2\nu_{\alpha,\beta,\sigma}}\jap{x-(t+1)v}^{2\omega_{\alpha,\beta,\sigma}}|\derv{'''}{'''}{'''} g||\derv{'}{'}{'} \a||\derv{''}{''}{''} g|}_{L^1([0,T];L^1_xL^1_v)}\\
&\lesssim \norm{\jap{v}^{\nu_{\alpha'',\beta'',\sigma''}+\nu_{\alpha''',\beta''',\sigma'''}-1}\jap{x-(t+1)v}^{2\omega_{\alpha,\beta,\sigma}}|\derv{'''}{'''}{'''} g||\derv{'}{'}{'} \a||\derv{''}{''}{''} g|}_{L^1([0,T];L^1_xL^1_v)}\\
&\lesssim (1+T)^{2|\beta|(1+\delta)}\norm{(1+t)^{-\frac{1}{2}-\delta-|\beta'''|(1+\delta)}\jap{v}^{\frac{1}{2}}\jap{v}^{\nu_{\alpha''',\beta''',\sigma'''}}\jap{x-(t+1)v}^{\omega_{\alpha''',\beta''',\sigma'''}}\derv{'''}{'''}{'''}g}_{L^2([0,T];L^2_xL^2_v)}\\
&\quad\times \norm{(1+t)^{1+2\delta-(|\beta'|-2)(1+\delta)}\jap{v}^{-2}\derv{'}{'}{'}\a}_{L^\infty([0,T];L^\infty_xL^\infty_v)}\\
&\quad\times \norm{(1+t)^{-\frac{1}{2}-\delta-|\beta''|(1+\delta)}\jap{v}^{\frac{1}{2}}\jap{v}^{\nu_{\alpha'',\beta'',\sigma''}}\jap{x-(t+1)v}^{\omega_{\alpha'',\beta'',\sigma''}}\derv{''}{''}{''}g}_{L^2([0,T];L^2_xL^2_v)}\\
&\lesssim (1+T)^{2|\beta|(1+\delta)}\eps^{\frac{3}{4}}\times(\sup_{t\in [0,T]}\eps^{\frac{3}{4}}(1+t)^{1+2\delta-(|\beta'|-2)(1+\delta)}(1+t)^{-\frac{7}{2}+|\beta'|(1+\delta)})\eps^{\frac{3}{4}}\\
&=\eps^{\frac{9}{4}}(1+T)^{2|\beta|(1+\delta)}.
\end{align*}
\emph{Subcase 2b):} $|\sigma'|\geq 1$.

In this case we have  $\nu_{\alpha'',\beta'',\sigma''}+\nu_{\alpha''',\beta''',\sigma'''}\geq2\nu_{\alpha,\beta,\sigma}+1$ and  $\omega_{\alpha'',\beta'',\sigma''}+\omega_{\alpha''',\beta''',\sigma'''}\geq2\omega_{\alpha,\beta,\sigma}+1$.\\
We apply \lref{a_bound_dv_Y} and \lref{L2_time_in_norm} to the following as in the above cases
\begin{align*}
&\norm{\jap{v}^{2\nu_{\alpha,\beta,\sigma}}\jap{x-(t+1)v}^{2\omega_{\alpha,\beta,\sigma}}|\derv{'''}{'''}{'''} g||\derv{'}{'}{'} \a||\derv{''}{''}{''} g|}_{L^1([0,T];L^1_xL^1_v)}\\
&\lesssim \norm{\jap{v}^{2\nu_{\alpha,\beta,\sigma}}\jap{x-(t+1)v}^{\omega_{\alpha'',\beta'',\sigma''}+\omega_{\alpha''',\beta''',\sigma'''}-1}|\derv{'''}{'''}{'''} g||\derv{'}{'}{'} \a||\derv{''}{''}{''} g|}_{L^1([0,T];L^1_xL^1_v)}\\
&\lesssim (1+T)^{2|\beta|(1+\delta)}\norm{(1+t)^{-\frac{1}{2}-\delta-|\beta'''|(1+\delta)}\jap{v}^{\frac{1}{2}}\jap{v}^{\nu_{\alpha''',\beta''',\sigma'''}}\jap{x-(t+1)v}^{\omega_{\alpha''',\beta''',\sigma'''}}\derv{'''}{'''}{'''}g}_{L^2([0,T];L^2_xL^2_v)}\\
&\quad\times \norm{(1+t)^{1+2\delta-(|\beta'|-2)(1+\delta)}\jap{x-(t+1)v}^{-1}\jap{v}^{-1}\derv{'}{'}{'}\a}_{L^\infty([0,T];L^\infty_xL^\infty_v)}\\
&\quad\times \norm{(1+t)^{-\frac{1}{2}-\delta-|\beta''|(1+\delta)}\jap{v}^{\frac{1}{2}}\jap{v}^{\nu_{\alpha'',\beta'',\sigma''}}\jap{x-(t+1)v}^{\omega_{\alpha'',\beta'',\sigma''}}\derv{''}{''}{''}g}_{L^2([0,T];L^2_xL^2_v)}\\
&\lesssim\eps^{\frac{9}{4}}(1+T)^{2|\beta|(1+\delta)}.
\end{align*}
\emph{Case 3:} $|\beta=0|$.\\
We have three subcases:\\
\emph{Subcase 3a):} $|\alpha'|\geq 2$.

In this case we have $\nu_{\alpha'',\beta'',\sigma''}+\nu_{\alpha''',\beta''',\sigma'''}\geq2\nu_{\alpha,\beta,\sigma}+2$ and  $\omega_{\alpha'',\beta'',\sigma''}+\omega_{\alpha''',\beta''',\sigma'''}\geq2\omega_{\alpha,\beta,\sigma}$.\\
We apply \lref{a_bound_d2x} and \lref{L2_time_in_norm} to the following
\begin{align*}
&\norm{\jap{v}^{2\nu_{\alpha,\beta,\sigma}}\jap{x-(t+1)v}^{2\omega_{\alpha,\beta,\sigma}}|\derv{'''}{'''}{'''} g||\derv{'}{'}{'} \a||\derv{''}{''}{''} g|}_{L^1([0,T];L^1_xL^1_v)}\\
&\lesssim \norm{\jap{v}^{\nu_{\alpha'',\beta'',\sigma''}+\nu_{\alpha''',\beta''',\sigma'''}-2}\jap{x-(t+1)v}^{2\omega_{\alpha,\beta,\sigma}}|\derv{'''}{'''}{'''} g||\derv{'}{'}{'} \a||\derv{''}{''}{''} g|}_{L^1([0,T];L^1_xL^1_v)}\\
&\lesssim (1+T)^{2|\beta|(1+\delta)}\norm{(1+t)^{-\frac{1}{2}-\delta-|\beta'''|(1+\delta)}\jap{v}^{\frac{1}{2}}\jap{v}^{\nu_{\alpha''',\beta''',\sigma'''}}\jap{x-(t+1)v}^{\omega_{\alpha''',\beta''',\sigma'''}}\derv{'''}{'''}{'''}g}_{L^2([0,T];L^2_xL^2_v)}\\
&\quad\times \norm{(1+t)^{1+2\delta-(|\beta'|-2)(1+\delta)}\jap{v}^{-3}\derv{'}{'}{'}\a}_{L^\infty([0,T];L^\infty_xL^\infty_v)}\\
&\quad\times \norm{(1+t)^{-\frac{1}{2}-\delta-|\beta''|(1+\delta)}\jap{v}^{\frac{1}{2}}\jap{v}^{\nu_{\alpha'',\beta'',\sigma''}}\jap{x-(t+1)v}^{\omega_{\alpha'',\beta'',\sigma''}}\derv{''}{''}{''}g}_{L^2([0,T];L^2_xL^2_v)}\\
&\lesssim\eps^{\frac{9}{4}}(1+T)^{2|\beta|(1+\delta)}.
\end{align*}
\emph{Subcase 3b):} $|\alpha'|\geq 1$ and $|\sigma'|\geq 1$.

In this case we have $\nu_{\alpha'',\beta'',\sigma''}+\nu_{\alpha''',\beta''',\sigma'''}\geq2\nu_{\alpha,\beta,\sigma}+2$ and  $\omega_{\alpha'',\beta'',\sigma''}+\omega_{\alpha''',\beta''',\sigma'''}\geq2\omega_{\alpha,\beta,\sigma}+1$.\\
We apply \lref{a_bound_dx_Y} and \lref{L2_time_in_norm} to the following
\begin{align*}
&\norm{\jap{v}^{2\nu_{\alpha,\beta,\sigma}}\jap{x-(t+1)v}^{2\omega_{\alpha,\beta,\sigma}}|\derv{'''}{'''}{'''} g||\derv{'}{'}{'} \a||\derv{''}{''}{''} g|}_{L^1([0,T];L^1_xL^1_v)}\\
&\lesssim \norm{\jap{v}^{\nu_{\alpha'',\beta'',\sigma''}+\nu_{\alpha''',\beta''',\sigma'''}-1}\jap{x-(t+1)v}^{\omega_{\alpha'',\beta'',\sigma''}+\omega_{\alpha''',\beta''',\sigma'''}-1}|\derv{'''}{'''}{'''} g|\\
&\times|\derv{'}{'}{'} \a||\derv{''}{''}{''} g|}_{L^1([0,T];L^1_xL^1_v)}\\
&\lesssim (1+T)^{2|\beta|(1+\delta)}\norm{(1+t)^{-\frac{1}{2}-\delta-|\beta'''|(1+\delta)}\jap{v}^{\frac{1}{2}}\jap{v}^{\nu_{\alpha''',\beta''',\sigma'''}}\jap{x-(t+1)v}^{\omega_{\alpha''',\beta''',\sigma'''}}\derv{'''}{'''}{'''}g}_{L^2([0,T];L^2_xL^2_v)}\\
&\quad\times \norm{(1+t)^{1+2\delta-(|\beta'|-2)(1+\delta)}\jap{v}^{-2}\jap{x-(t+1)v}^{-1}\derv{'}{'}{'}\a}_{L^\infty([0,T];L^\infty_xL^\infty_v)}\\
&\quad\times \norm{(1+t)^{-\frac{1}{2}-\delta-|\beta''|(1+\delta)}\jap{v}^{\frac{1}{2}}\jap{v}^{\nu_{\alpha'',\beta'',\sigma''}}\jap{x-(t+1)v}^{\omega_{\alpha'',\beta'',\sigma''}}\derv{''}{''}{''}g}_{L^2([0,T];L^2_xL^2_v)}\\
&\lesssim\eps^{\frac{9}{4}}(1+T)^{2|\beta|(1+\delta)}.
\end{align*}
\emph{Subcase 3c):}$|\sigma'|\geq 2$.

In this case we have $\nu_{\alpha'',\beta'',\sigma''}+\nu_{\alpha''',\beta''',\sigma'''}\geq2\nu_{\alpha,\beta,\sigma}$ and  $\omega_{\alpha'',\beta'',\sigma''}+\omega_{\alpha''',\beta''',\sigma'''}\geq2\omega_{\alpha,\beta,\sigma}+2$.\\
We apply \lref{a_bound_2Y} and \lref{L2_time_in_norm} to the following
\begin{align*}
&\norm{\jap{v}^{2\nu_{\alpha,\beta,\sigma}}\jap{x-(t+1)v}^{2\omega_{\alpha,\beta,\sigma}}|\derv{'''}{'''}{'''} g||\derv{'}{'}{'} \a||\derv{''}{''}{''} g|}_{L^1([0,T];L^1_xL^1_v)}\\
&\lesssim \norm{\jap{v}^{2\nu_{\alpha,\beta,\sigma}}\jap{x-(t+1)v}^{\omega_{\alpha'',\beta'',\sigma''}+\omega_{\alpha''',\beta''',\sigma'''}-2}|\derv{'''}{'''}{'''} g||\derv{'}{'}{'} \a||\derv{''}{''}{''} g|}_{L^1([0,T];L^1_xL^1_v)}\\
&\lesssim (1+T)^{2|\beta|(1+\delta)}\norm{(1+t)^{-\frac{1}{2}-\delta-|\beta'''|(1+\delta)}\jap{v}^{\frac{1}{2}}\jap{v}^{\nu_{\alpha''',\beta''',\sigma'''}}\jap{x-(t+1)v}^{\omega_{\alpha''',\beta''',\sigma'''}}\derv{'''}{'''}{'''}g}_{L^2([0,T];L^2_xL^2_v)}\\
&\quad\times \norm{(1+t)^{1+2\delta-(|\beta'|-2)(1+\delta)}\jap{v}^{-1}\jap{x-(t+1)v}^{-2}\derv{'}{'}{'}\a}_{L^\infty([0,T];L^\infty_xL^\infty_v)}\\
&\quad\times \norm{(1+t)^{-\frac{1}{2}-\delta-|\beta''|(1+\delta)}\jap{v}^{\frac{1}{2}}\jap{v}^{\nu_{\alpha'',\beta'',\sigma''}}\jap{x-(t+1)v}^{\omega_{\alpha'',\beta'',\sigma''}}\derv{''}{''}{''}g}_{L^2([0,T];L^2_xL^2_v)}\\
&\lesssim\eps^{\frac{9}{4}}(1+T)^{2|\beta|(1+\delta)}.
\end{align*}
\end{proof}
\begin{proposition}\label{p.T3_2}
Let $|\alpha|+|\beta|+|\sigma|\leq 10$. Then the term $T^{\alpha,\beta,\sigma}_{3,2}$ is bounded as follows for all $T\in [0,T_{\boot})$,
$$T^{\alpha,\beta,\sigma}_{3,2}\lesssim \eps^2(1+T)^{2|\beta|(1+\delta)}.$$
\end{proposition}
\begin{proof}
We fix a typical term with $\alpha',\alpha'',\beta',\beta'',\sigma',\sigma''$ satisfying the required conditions. Note that there are more than $8$ deriviatives hitting $\a$ thus we have to estimate it in $L^2_x$. Further note that since there are atleast two derivatives hitting $\a$, there are again $6$ different scenarios to be considered. By our conditions we have the following relations between the weights
$$\nu_{\alpha'',\beta'',\sigma''}=\nu_{\alpha,\beta,\sigma}+3\frac{|\alpha'|+|\beta'|+|\sigma'|}{2}-|\beta'|-1$$
and
$$\omega_{\alpha'',\beta'',\sigma''}=\omega_{\alpha,\beta,\sigma}+3\frac{|\alpha'|+|\beta'|+|\sigma'|}{2}-(|\beta'|+|\alpha'|)-1.$$
\emph{Case 1:} $|\beta'|\geq 2$.

Since the total number of derivatives are atleast $9$, in this case we have that $\nu_{\alpha'',\beta'',\sigma''}\geq\nu_{\alpha,\beta,\sigma}+3$ and  $\omega_{\alpha'',\beta'',\sigma''}\geq\omega_{\alpha,\beta,\sigma}+3$.\\
We apply \lref{L_infty_x_L2_v}, \lref{a_bound_d2v} and \lref{L2_time_in_norm} and H\"{o}lder's inequality to get
\begin{align*}
&\norm{\jap{v}^{2\nu_{\alpha,\beta,\sigma}}\jap{x-(t+1)v}^{2\omega_{\alpha,\beta,\sigma}}|\der g||\derv{'}{'}{'} \a||\derv{''}{''}{''} g|}_{L^1([0,T];L^1_xL^1_v)}\\
&\lesssim (1+T)^{2|\beta|(1+\delta)}\norm{(1+t)^{-\frac{1}{2}-\delta-|\beta|(1+\delta)}\jap{v}^{\frac{1}{2}}\jap{v}^{\nu_{\alpha,\beta,\sigma}}\jap{x-(t+1)v}^{\omega_{\alpha,\beta,\sigma}}\der g}_{L^2([0,T];L^2_xL^2_v)}\\
&\quad\times \norm{(1+t)^{1+2\delta-(|\beta'|-2)(1+\delta)}\jap{v}^{-1}\derv{'}{'}{'}\a}_{L^\infty([0,T];L^2_xL^\infty_v)}\\
&\quad\times \norm{(1+t)^{-\frac{1}{2}-\delta-|\beta''|(1+\delta)}\jap{v}^{\frac{1}{2}}\jap{v}^{\nu_{\alpha'',\beta'',\sigma''}-3}\jap{x-(t+1)v}^{\omega_{\alpha'',\beta'',\sigma''}-3}\derv{''}{''}{''}}_{L^2([0,T];L^\infty_xL^2_v)}\\
&\lesssim (1+T)^{2|\beta|(1+\delta)}\eps^{\frac{3}{4}}\times(\sup_{t\in [0,T]}\eps^{\frac{3}{4}}(1+t)^{1+2\delta-(|\beta'|-2)(1+\delta)}(1+t)^{-\frac{7}{2}+|\beta'|(1+\delta)})\\
&\quad\times  \norm{(1+t)^{-\frac{1}{2}-\delta-|\beta''|(1+\delta)}\jap{v}^{\frac{1}{2}}\jap{v}^{\nu_{\alpha'',\beta'',\sigma''}-3}\jap{x-(t+1)v}^{\omega_{\alpha'',\beta'',\sigma''}-3}\derv{'''}{''}{''}}_{L^2([0,T];L^2_xL^2_v)}\\
&\lesssim (1+T)^{2|\beta|(1+\delta)}\eps^{\frac{3}{4}}\times \eps^{\frac{3}{4}}\times \eps^{\frac{3}{4}}=\eps^{\frac{9}{4}}(1+T)^{2|\beta|(1+\delta)},
\end{align*}
where $|\alpha'''|\leq |\alpha''|+2$, thus $\nu_{\alpha''',\beta'',\sigma''}\geq\nu_{\alpha'',\beta'',\sigma''}-3$ and $\omega_{\alpha''',\beta'',\sigma''}\geq\omega_{\alpha'',\beta'',\sigma''}-1$.\\
\emph{Case 2:} $|\beta'|=1$.\\
We break this into two further subcases.\\
\emph{Subcase 2a):} $|\alpha'|\geq 1$.

For both $|\alpha'|+|\beta'|+|\sigma|=\{9,10\}$, we have in this case $\nu_{\alpha'',\beta'',\sigma''}\geq \nu_{\alpha,\beta,\sigma}+4$ and $\omega_{\alpha'',\beta'',\sigma''}\geq\omega_{\alpha,\beta,\sigma}+3$.\\
We apply \lref{a_bound_dv_dx} and \lref{L2_time_in_norm}, to obtain
\begin{align*}
&\norm{\jap{v}^{2\nu_{\alpha,\beta,\sigma}}\jap{x-(t+1)v}^{2\omega_{\alpha,\beta,\sigma}}|\der g||\derv{'}{'}{'} \a||\derv{''}{''}{''} g|}_{L^1([0,T];L^1_xL^1_v)}\\
&\lesssim (1+T)^{2|\beta|(1+\delta)}\norm{(1+t)^{-\frac{1}{2}-\delta-|\beta|(1+\delta)}\jap{v}^{\frac{1}{2}}\jap{v}^{\nu_{\alpha,\beta,\sigma}}\jap{x-(t+1)v}^{\omega_{\alpha,\beta,\sigma}}\der g}_{L^2([0,T];L^2_xL^2_v)}\\
&\quad\times \norm{(1+t)^{1+2\delta-(|\beta'|-2)(1+\delta)}\jap{v}^{-2}\derv{'}{'}{'}\a}_{L^\infty([0,T];L^2_xL^\infty_v)}\\
&\quad\times \norm{(1+t)^{-\frac{1}{2}-\delta-|\beta''|(1+\delta)}\jap{v}^{\frac{1}{2}}\jap{v}^{\nu_{\alpha'',\beta'',\sigma''}-3}\jap{x-(t+1)v}^{\omega_{\alpha'',\beta'',\sigma''}-3}\derv{''}{''}{''}}_{L^2([0,T];L^\infty_xL^2_v)}.
\end{align*}
Now we proceed in the same way as in Case 1.\\
\emph{Subcase 2b):} $|\sigma'|\geq 1$.

In this case we have  $\nu_{\alpha'',\beta'',\sigma''}\geq\nu_{\alpha,\beta,\sigma}+3$ and  $\omega_{\alpha'',\beta'',\sigma''}\geq\omega_{\alpha,\beta,\sigma}+3$.\\
We apply \lref{a_bound_dv_Y} and \lref{L2_time_in_norm} to the following as in the above cases
\begin{align*}
&\norm{\jap{v}^{2\nu_{\alpha,\beta,\sigma}}\jap{x-(t+1)v}^{2\omega_{\alpha,\beta,\sigma}}|\der g||\derv{'}{'}{'} \a||\derv{''}{''}{''} g|}_{L^1([0,T];L^1_xL^1_v)}\\
&\lesssim (1+T)^{2|\beta|(1+\delta)}\norm{(1+t)^{-\frac{1}{2}-\delta-|\beta|(1+\delta)}\jap{v}^{\frac{1}{2}}\jap{v}^{\nu_{\alpha,\beta,\sigma}}\jap{x-(t+1)v}^{\omega_{\alpha,\beta,\sigma}}\der g}_{L^2([0,T];L^2_xL^2_v)}\\
&\quad\times \norm{(1+t)^{1+2\delta-(|\beta'|-2)(1+\delta)}\jap{x-(t+1)v}^{-1}\jap{v}^{-1}\derv{'}{'}{'}\a}_{L^\infty([0,T];L^2_xL^\infty_v)}\\
&\quad\times \norm{(1+t)^{-\frac{1}{2}-\delta-|\beta''|(1+\delta)}\jap{v}^{\frac{1}{2}}\jap{v}^{\nu_{\alpha'',\beta'',\sigma''}-3}\jap{x-(t+1)v}^{\omega_{\alpha'',\beta'',\sigma''}-2}\derv{''}{''}{''}}_{L^2([0,T];L^\infty_xL^2_v)}.
\end{align*}
We again proceed as in Case 1.\\
\emph{Case 3:} $|\beta'|=0$.\\
We have three subcases:\\
\emph{Subcase 3a):} $|\alpha'|\geq 2$.

For both $|\alpha'|+|\beta'|+|\sigma|=\{9,10\}$, we have in this case $\nu_{\alpha'',\beta'',\sigma''}\geq \nu_{\alpha,\beta,\sigma}+5$ and $\omega_{\alpha'',\beta'',\sigma''}\geq\omega_{\alpha,\beta,\sigma}+3$.\\
We apply \lref{a_bound_d2x} and \lref{L2_time_in_norm} to the following
\begin{align*}
&\norm{\jap{v}^{2\nu_{\alpha,\beta,\sigma}}\jap{x-(t+1)v}^{2\omega_{\alpha,\beta,\sigma}}|\der g||\derv{'}{'}{'} \a||\derv{''}{''}{''} g|}_{L^1([0,T];L^1_xL^1_v)}\\
&\lesssim (1+T)^{2|\beta|(1+\delta)}\norm{(1+t)^{-\frac{1}{2}-\delta-|\beta|(1+\delta)}\jap{v}^{\frac{1}{2}}\jap{v}^{\nu_{\alpha,\beta,\sigma}}\jap{x-(t+1)v}^{\omega_{\alpha,\beta,\sigma}}\der g}_{L^2([0,T];L^2_xL^2_v)}\\
&\quad\times \norm{(1+t)^{1+2\delta-(|\beta'|-2)(1+\delta)}\jap{v}^{-3}\derv{'}{'}{'}\a}_{L^\infty([0,T];L^2_xL^\infty_v)}\\
&\quad\times \norm{(1+t)^{-\frac{1}{2}-\delta-|\beta''|(1+\delta)}\jap{v}^{\frac{1}{2}}\jap{v}^{\nu_{\alpha'',\beta'',\sigma''}-3}\jap{x-(t+1)v}^{\omega_{\alpha'',\beta'',\sigma''}-3}\derv{''}{''}{''}}_{L^2([0,T];L^\infty_xL^2_v)}.
\end{align*}
\emph{Subcase 3b):} $|\alpha'|\geq 1$ and $|\sigma'|\geq 1$.

For both $|\alpha'|+|\beta'|+|\sigma|=\{9,10\}$, we have in this case $\nu_{\alpha'',\beta'',\sigma''}\geq \nu_{\alpha,\beta,\sigma}+4$ and\\ $\omega_{\alpha'',\beta'',\sigma''}\geq\omega_{\alpha,\beta,\sigma}+3$.\\
We apply \lref{a_bound_d2x} and \lref{L2_time_in_norm} to the following
\begin{align*}
&\norm{\jap{v}^{2\nu_{\alpha,\beta,\sigma}}\jap{x-(t+1)v}^{2\omega_{\alpha,\beta,\sigma}}|\der g||\derv{'}{'}{'} \a||\derv{''}{''}{''} g|}_{L^1([0,T];L^1_xL^1_v)}\\
&\lesssim (1+T)^{2|\beta|(1+\delta)}\norm{(1+t)^{-\frac{1}{2}-\delta-|\beta|(1+\delta)}\jap{v}^{\frac{1}{2}}\jap{v}^{\nu_{\alpha,\beta,\sigma}}\jap{x-(t+1)v}^{\omega_{\alpha,\beta,\sigma}}\der g}_{L^2([0,T];L^2_xL^2_v)}\\
&\quad\times \norm{(1+t)^{1+2\delta-(|\beta'|-2)(1+\delta)}\jap{x-(t+1)v}^{-1}\jap{v}^{-2}\derv{'}{'}{'}\a}_{L^\infty([0,T];L^2_xL^\infty_v)}\\
&\quad\times \norm{(1+t)^{-\frac{1}{2}-\delta-|\beta''|(1+\delta)}\jap{v}^{\frac{1}{2}}\jap{v}^{\nu_{\alpha'',\beta'',\sigma''}-3}\jap{x-(t+1)v}^{\omega_{\alpha'',\beta'',\sigma''}-2}\derv{''}{''}{''}}_{L^2([0,T];L^\infty_xL^2_v)}.
\end{align*}
We finish off in the same way as in Case 1.\\
\emph{Subcase 3c):}$|\sigma'|\geq 2$.

In this case we have  $\nu_{\alpha'',\beta'',\sigma''}\geq\nu_{\alpha,\beta,\sigma}+3$ and  $\omega_{\alpha'',\beta'',\sigma''}\geq\omega_{\alpha,\beta,\sigma}+3$.\\
We apply \lref{a_bound_dv_Y} and \lref{L2_time_in_norm} to the following as in the above cases
\begin{align*}
&\norm{\jap{v}^{2\nu_{\alpha,\beta,\sigma}}\jap{x-(t+1)v}^{2\omega_{\alpha,\beta,\sigma}}|\der g||\derv{'}{'}{'} \a||\derv{''}{''}{''} g|}_{L^1([0,T];L^1_xL^1_v)}\\
&\lesssim (1+T)^{2|\beta|(1+\delta)}\norm{(1+t)^{-\frac{1}{2}-\delta-|\beta|(1+\delta)}\jap{v}^{\frac{1}{2}}\jap{v}^{\nu_{\alpha,\beta,\sigma}}\jap{x-(t+1)v}^{\omega_{\alpha,\beta,\sigma}}\der g}_{L^2([0,T];L^2_xL^2_v)}\\
&\quad\times \norm{(1+t)^{1+2\delta-(|\beta'|-2)(1+\delta)}\jap{x-(t+1)v}^{-2}\jap{v}^{-1}\derv{'}{'}{'}\a}_{L^\infty([0,T];L^2_xL^\infty_v)}\\
&\quad\times \norm{(1+t)^{-\frac{1}{2}-\delta-|\beta''|(1+\delta)}\jap{v}^{\frac{1}{2}}\jap{v}^{\nu_{\alpha'',\beta'',\sigma''}-2}\jap{x-(t+1)v}^{\omega_{\alpha'',\beta'',\sigma''}-1}\derv{''}{''}{''}}_{L^2([0,T];L^\infty_xL^2_v)}.
\end{align*}
We again proceed as in Case 1.\\
\end{proof}
\begin{proposition}\label{p.T3_3}
Let $|\alpha|+|\beta|+|\sigma|\leq 10$. Then the term $T^{\alpha,\beta,\sigma}_{3,3}$ is bounded as follows for all $T\in [0,T_{\boot})$,
$$T^{\alpha,\beta,\sigma}_{3,3}\lesssim \eps^2(1+T)^{2|\beta|(1+\delta)}$$
\end{proposition}
\begin{proof}
We fix a typical term with $\alpha',\alpha'',\alpha''',\beta',\beta'',\beta''',\sigma',\sigma'',\sigma'''$ satisfying the required conditions. Note that there are less than $8$ deriviatives hitting $\a$ thus we estimate it in $L^\infty_x$. Further note that although we just have one derivative hitting $\a$, there is one less $\jap{v}$ weight to worry about. We split the proof into three cases. We have the following relations for the weights
$$\nu_{\alpha'',\beta'',\sigma''}+\nu_{\alpha''',\beta''',\sigma'''}=2\nu_{\alpha,\beta,\sigma}+\frac{3|\alpha'|+|\beta'|+3|\sigma'|}{2}-\frac{1}{2}$$
and 
$$\omega_{\alpha'',\beta'',\sigma''}+\omega_{\alpha''',\beta''',\sigma'''}=2\omega_{\alpha,\beta,\sigma}+\frac{|\alpha'|+|\beta'|+3|\sigma'|}{2}-\frac{1}{2}.$$
\emph{Case 1:} $|\beta'|=1$.

In this case we have $\nu_{\alpha'',\beta'',\sigma''}+\nu_{\alpha''',\beta''',\sigma'''}=2\nu_{\alpha,\beta,\sigma}$ and $\omega_{\alpha'',\beta'',\sigma''}+\omega_{\alpha''',\beta''',\sigma'''}=2\omega_{\alpha,\beta,\sigma}$.\\
We apply \lref{a_bound_dv} and \lref{L2_time_in_norm} and H\"{o}lder's inequality to get
\begin{align*}
&\norm{\jap{v}^{2\nu_{\alpha,\beta,\sigma}-1}\jap{x-(t+1)v}^{2\omega_{\alpha,\beta,\sigma}}|\derv{'''}{'''}{'''} g||\derv{'}{'}{'} \a||\derv{''}{''}{''} g|}_{L^1([0,T];L^1_xL^1_v)}\\
&\lesssim (1+T)^{2|\beta|(1+\delta)}\norm{(1+t)^{-\frac{1}{2}-\delta-|\beta'''|(1+\delta)}\jap{v}^{\frac{1}{2}}\jap{v}^{\nu_{\alpha''',\beta''',\sigma'''}}\jap{x-(t+1)v}^{\omega_{\alpha''',\beta''',\sigma'''}}\derv{'''}{'''}{'''}g}_{L^2([0,T];L^2_xL^2_v)}\\
&\quad\times \norm{(1+t)^{1+2\delta-(|\beta'|-1)(1+\delta)}\jap{v}^{-2}\derv{'}{'}{'}\a}_{L^\infty([0,T];L^\infty_xL^\infty_v)}\\
&\quad\times \norm{(1+t)^{-\frac{1}{2}-\delta-|\beta''|(1+\delta)}\jap{v}^{\frac{1}{2}}\jap{v}^{\nu_{\alpha'',\beta'',\sigma''}}\jap{x-(t+1)v}^{\omega_{\alpha'',\beta'',\sigma''}}\derv{''}{''}{''}g}_{L^2([0,T];L^2_xL^2_v)}\\
&\lesssim (1+T)^{2|\beta|(1+\delta)}\eps^{\frac{3}{4}}\times(\sup_{t\in [0,T]}\eps^{\frac{3}{4}}(1+t)^{1+2\delta-(|\beta'|-1)(1+\delta)}(1+t)^{-\frac{5}{2}+|\beta'|(1+\delta)})\eps^{\frac{3}{4}}\\
&=\eps^{\frac{9}{4}}(1+T)^{2|\beta|(1+\delta)}.
\end{align*}
\emph{Case 2:} $|\alpha'|=1$.

In this case we have $\nu_{\alpha'',\beta'',\sigma''}+\nu_{\alpha''',\beta''',\sigma'''}=2\nu_{\alpha,\beta,\sigma}+1$ and $\omega_{\alpha'',\beta'',\sigma''}+\omega_{\alpha''',\beta''',\sigma'''}=2\omega_{\alpha,\beta,\sigma}$.\\
We apply \lref{a_bound_dv} and \lref{L2_time_in_norm} and H\"{o}lder's inequality to get
\begin{align*}
&\norm{\jap{v}^{2\nu_{\alpha,\beta,\sigma}-1}\jap{x-(t+1)v}^{2\omega_{\alpha,\beta,\sigma}}|\derv{'''}{'''}{'''} g||\derv{'}{'}{'} \a||\derv{''}{''}{''} g|}_{L^1([0,T];L^1_xL^1_v)}\\
&\lesssim (1+T)^{2|\beta|(1+\delta)}\norm{(1+t)^{-\frac{1}{2}-\delta-|\beta'''|(1+\delta)}\jap{v}^{\frac{1}{2}}\jap{v}^{\nu_{\alpha''',\beta''',\sigma'''}}\jap{x-(t+1)v}^{\omega_{\alpha''',\beta''',\sigma'''}}\derv{'''}{'''}{'''}g}_{L^2([0,T];L^2_xL^2_v)}\\
&\quad\times \norm{(1+t)^{1+2\delta-(|\beta'|-1)(1+\delta)}\jap{v}^{-3}\derv{'}{'}{'}\a}_{L^\infty([0,T];L^\infty_xL^\infty_v)}\\
&\quad\times \norm{(1+t)^{-\frac{1}{2}-\delta-|\beta''|(1+\delta)}\jap{v}^{\frac{1}{2}}\jap{v}^{\nu_{\alpha'',\beta'',\sigma''}}\jap{x-(t+1)v}^{\omega_{\alpha'',\beta'',\sigma''}}\derv{''}{''}{''}g}_{L^2([0,T];L^2_xL^2_v)}\\
&\lesssim (1+T)^{2|\beta|(1+\delta)}\eps^{\frac{3}{4}}\times(\sup_{t\in [0,T]}\eps^{\frac{3}{4}}(1+t)^{1+2\delta-(|\beta'|-1)(1+\delta)}(1+t)^{-\frac{5}{2}+|\beta'|(1+\delta)})\eps^{\frac{3}{4}}\\
&=\eps^{\frac{9}{4}}(1+T)^{2|\beta|(1+\delta)}.
\end{align*}
\emph{Case 3:} $|\sigma'|=1$.

In this case we have $\nu_{\alpha'',\beta'',\sigma''}+\nu_{\alpha''',\beta''',\sigma'''}=2\nu_{\alpha,\beta,\sigma}+1$ and $\omega_{\alpha'',\beta'',\sigma''}+\omega_{\alpha''',\beta''',\sigma'''}=2\omega_{\alpha,\beta,\sigma}+1$.\\
We apply \lref{a_bound_Y} and \lref{L2_time_in_norm} and H\"{o}lder's inequality to get
\begin{align*}
&\norm{\jap{v}^{2\nu_{\alpha,\beta,\sigma}-1}\jap{x-(t+1)v}^{2\omega_{\alpha,\beta,\sigma}}|\derv{'''}{'''}{'''} g||\derv{'}{'}{'} \a||\derv{''}{''}{''} g|}_{L^1([0,T];L^1_xL^1_v)}\\
&\lesssim (1+T)^{2|\beta|(1+\delta)}\norm{(1+t)^{-\frac{1}{2}-\delta-|\beta'''|(1+\delta)}\jap{v}^{\frac{1}{2}}\jap{v}^{\nu_{\alpha''',\beta''',\sigma'''}}\jap{x-(t+1)v}^{\omega_{\alpha''',\beta''',\sigma'''}}\derv{'''}{'''}{'''}g}_{L^2([0,T];L^2_xL^2_v)}\\
&\quad\times \norm{(1+t)^{1+2\delta-(|\beta'|-1)(1+\delta)}\jap{v}^{-2}\jap{x-(t+1)v}^{-1}\derv{'}{'}{'}\a}_{L^\infty([0,T];L^\infty_xL^\infty_v)}\\
&\quad\times \norm{(1+t)^{-\frac{1}{2}-\delta-|\beta''|(1+\delta)}\jap{v}^{\frac{1}{2}}\jap{v}^{\nu_{\alpha'',\beta'',\sigma''}}\jap{x-(t+1)v}^{\omega_{\alpha'',\beta'',\sigma''}}\derv{''}{''}{''}g}_{L^2([0,T];L^2_xL^2_v)}\\
&\lesssim (1+T)^{2|\beta|(1+\delta)}\eps^{\frac{3}{4}}\times(\sup_{t\in [0,T]}\eps^{\frac{3}{4}}(1+t)^{1+2\delta-(|\beta'|-1)(1+\delta)}(1+t)^{-\frac{5}{2}+|\beta'|(1+\delta)})\eps^{\frac{3}{4}}\\
&=\eps^{\frac{9}{4}}(1+T)^{2|\beta|(1+\delta)}.
\end{align*}
\end{proof}
\begin{proposition}\label{p.T3_4}
Let $|\alpha|+|\beta|+|\sigma|\leq 10$. Then the term $T^{\alpha,\beta,\sigma}_{3,4}$ is bounded as follows for all $T\in [0,T_{\boot})$,
$$T^{\alpha,\beta,\sigma}_{3,4}\lesssim \eps^2(1+T)^{2|\beta|(1+\delta)}.$$
\end{proposition}
\begin{proof}
We fix a typical term with $\alpha',\alpha'',\alpha''',\beta',\beta'',\beta''',\sigma',\sigma'',\sigma'''$ satisfying the required conditions. Note that there are less than $8$ deriviatives hitting $\a$ thus we estimate it in $L^\infty_x$. Further note that although we just have one derivative hitting $\a$, there is one less $\jap{x-(t+1)v}$ weight to worry about. We split the proof into three cases. We have the following relations for the weights
$$\nu_{\alpha'',\beta'',\sigma''}+\nu_{\alpha''',\beta''',\sigma'''}=2\nu_{\alpha,\beta,\sigma}+\frac{3|\alpha'|+|\beta'|+3|\sigma'|}{2}-\frac{1}{2}$$
and 
$$\omega_{\alpha'',\beta'',\sigma''}+\omega_{\alpha''',\beta''',\sigma'''}=2\omega_{\alpha,\beta,\sigma}+\frac{|\alpha'|+|\beta'|+3|\sigma'|}{2}-\frac{1}{2}.$$
\emph{Case 1:} $|\beta'|=1$.

In this case we have $\nu_{\alpha'',\beta'',\sigma''}+\nu_{\alpha''',\beta''',\sigma'''}=2\nu_{\alpha,\beta,\sigma}$ and $\omega_{\alpha'',\beta'',\sigma''}+\omega_{\alpha''',\beta''',\sigma'''}=2\omega_{\alpha,\beta,\sigma}$.\\
We apply \lref{a_bound_dv_Y} and \lref{L2_time_in_norm} and H\"{o}lder's inequality to get
\begin{align*}
&\norm{(1+t)\jap{v}^{2\nu_{\alpha,\beta,\sigma}}\jap{x-(t+1)v}^{2\omega_{\alpha,\beta,\sigma}-1}|\derv{'''}{'''}{'''} g||\derv{'}{'}{'} \a||\derv{''}{''}{''} g|}_{L^1([0,T];L^1_xL^1_v)}\\
&\lesssim (1+T)^{2|\beta|(1+\delta)}\norm{(1+t)^{-\frac{1}{2}-\delta-|\beta'''|(1+\delta)}\jap{v}^{\frac{1}{2}}\jap{v}^{\nu_{\alpha''',\beta''',\sigma'''}}\jap{x-(t+1)v}^{\omega_{\alpha''',\beta''',\sigma'''}}\derv{'''}{'''}{'''}g}_{L^2([0,T];L^2_xL^2_v)}\\
&\quad\times \norm{(1+t)^{1+2\delta-(|\beta'|-2)(1+\delta)}\jap{v}^{-1}\jap{x-(t+1)v}^{-1}\derv{'}{'}{'}\a}_{L^\infty([0,T];L^\infty_xL^\infty_v)}\\
&\quad\times \norm{(1+t)^{-\frac{1}{2}-\delta-|\beta''|(1+\delta)}\jap{v}^{\frac{1}{2}}\jap{v}^{\nu_{\alpha'',\beta'',\sigma''}}\jap{x-(t+1)v}^{\omega_{\alpha'',\beta'',\sigma''}}\derv{''}{''}{''}g}_{L^2([0,T];L^2_xL^2_v)}\\
&\lesssim (1+T)^{2|\beta|(1+\delta)}\eps^{\frac{3}{4}}\times(\sup_{t\in [0,T]}\eps^{\frac{3}{4}}(1+t)^{1+2\delta-(|\beta'|-2)(1+\delta)}(1+t)^{-\frac{7}{2}+|\beta'|(1+\delta)})\eps^{\frac{3}{4}}\\
&=\eps^{\frac{9}{4}}(1+T)^{2|\beta|(1+\delta)}.
\end{align*}
\emph{Case 2:} $|\alpha'|=1$.

In this case we have $\nu_{\alpha'',\beta'',\sigma''}+\nu_{\alpha''',\beta''',\sigma'''}=2\nu_{\alpha,\beta,\sigma}+1$ and $\omega_{\alpha'',\beta'',\sigma''}+\omega_{\alpha''',\beta''',\sigma'''}=2\omega_{\alpha,\beta,\sigma}$.\\
We apply \lref{a_bound_dx_Y} and \lref{L2_time_in_norm} and H\"{o}lder's inequality to get
\begin{align*}
&\norm{\jap{v}^{2\nu_{\alpha,\beta,\sigma}}\jap{x-(t+1)v}^{2\omega_{\alpha,\beta,\sigma}-1}|\derv{'''}{'''}{'''} g||\derv{'}{'}{'} \a||\derv{''}{''}{''} g|}_{L^1([0,T];L^1_xL^1_v)}\\
&\lesssim (1+T)^{2|\beta|(1+\delta)}\norm{(1+t)^{-\frac{1}{2}-\delta-|\beta'''|(1+\delta)}\jap{v}^{\frac{1}{2}}\jap{v}^{\nu_{\alpha''',\beta''',\sigma'''}}\jap{x-(t+1)v}^{\omega_{\alpha''',\beta''',\sigma'''}}\derv{'''}{'''}{'''}g}_{L^2([0,T];L^2_xL^2_v)}\\
&\quad\times \norm{(1+t)^{1+2\delta-(|\beta'|-2)(1+\delta)}\jap{v}^{-2}\jap{x-(t+1)v}^{-1}\derv{'}{'}{'}\a}_{L^\infty([0,T];L^\infty_xL^\infty_v)}\\
&\quad\times \norm{(1+t)^{-\frac{1}{2}-\delta-|\beta''|(1+\delta)}\jap{v}^{\frac{1}{2}}\jap{v}^{\nu_{\alpha'',\beta'',\sigma''}}\jap{x-(t+1)v}^{\omega_{\alpha'',\beta'',\sigma''}}\derv{''}{''}{''}g}_{L^2([0,T];L^2_xL^2_v)}\\
&\lesssim (1+T)^{2|\beta|(1+\delta)}\eps^{\frac{3}{4}}\times(\sup_{t\in [0,T]}\eps^{\frac{3}{4}}(1+t)^{1+2\delta-(|\beta'|-2)(1+\delta)}(1+t)^{-\frac{7}{2}+|\beta'|(1+\delta)})\eps^{\frac{3}{4}}\\
&=\eps^{\frac{9}{4}}(1+T)^{2|\beta|(1+\delta)}.
\end{align*}
\emph{Case 3:} $|\sigma'|=1$.

In this case we have $\nu_{\alpha'',\beta'',\sigma''}+\nu_{\alpha''',\beta''',\sigma'''}=2\nu_{\alpha,\beta,\sigma}+1$ and $\omega_{\alpha'',\beta'',\sigma''}+\omega_{\alpha''',\beta''',\sigma'''}=2\omega_{\alpha,\beta,\sigma}+1$.\\
We apply \lref{a_bound_2Y} and \lref{L2_time_in_norm} and H\"{o}lder's inequality to get
\begin{align*}
&\norm{\jap{v}^{2\nu_{\alpha,\beta,\sigma}}\jap{x-(t+1)v}^{2\omega_{\alpha,\beta,\sigma}-1}|\derv{'''}{'''}{'''} g||\derv{'}{'}{'} \a||\derv{''}{''}{''} g|}_{L^1([0,T];L^1_xL^1_v)}\\
&\lesssim (1+T)^{2|\beta|(1+\delta)}\norm{(1+t)^{-\frac{1}{2}-\delta-|\beta'''|(1+\delta)}\jap{v}^{\frac{1}{2}}\jap{v}^{\nu_{\alpha''',\beta''',\sigma'''}}\jap{x-(t+1)v}^{\omega_{\alpha''',\beta''',\sigma'''}}\derv{'''}{'''}{'''}g}_{L^2([0,T];L^2_xL^2_v)}\\
&\quad\times \norm{(1+t)^{1+2\delta-(|\beta'|-2)(1+\delta)}\jap{v}^{-1}\jap{x-(t+1)v}^{-2}\derv{'}{'}{'}\a}_{L^\infty([0,T];L^\infty_xL^\infty_v)}\\
&\quad\times \norm{(1+t)^{-\frac{1}{2}-\delta-|\beta''|(1+\delta)}\jap{v}^{\frac{1}{2}}\jap{v}^{\nu_{\alpha'',\beta'',\sigma''}}\jap{x-(t+1)v}^{\omega_{\alpha'',\beta'',\sigma''}}\derv{''}{''}{''}g}_{L^2([0,T];L^2_xL^2_v)}\\
&\lesssim (1+T)^{2|\beta|(1+\delta)}\eps^{\frac{3}{4}}\times(\sup_{t\in [0,T]}\eps^{\frac{3}{4}}(1+t)^{1+2\delta-(|\beta'|-2)(1+\delta)}(1+t)^{-\frac{7}{2}+|\beta'|(1+\delta)})\eps^{\frac{3}{4}}\\
&=\eps^{\frac{9}{4}}(1+T)^{2|\beta|(1+\delta)}.
\end{align*}
\end{proof}
\begin{proposition}\label{p.T4}
Let $|\alpha|+|\beta|+|\sigma|\leq 10$. Then the term $T^{\alpha,\beta,\sigma}_{4}$ is bounded as follows for all $T\in [0,T_{\boot})$,
$$T^{\alpha,\beta,\sigma}_{4}\lesssim \eps^2(1+T)^{2|\beta|(1+\delta)}.$$
\end{proposition}
\begin{proof}
The proof proceeds by analyzing two cases- when less than 9 derivatives fall on $\cm$ and when more than $8$ fall on it.

The weights satify the relations 
$$\nu_{\alpha'',\beta'',\sigma''}=\nu_{\alpha,\beta,\sigma}+\frac{3|\alpha'|+|\beta'|+3|\sigma'|}{2}$$
and
$$\omega_{\alpha'',\beta'',\sigma''}=2\omega_{\alpha,\beta,\sigma}+\frac{|\alpha'|+|\beta'|+3|\sigma'|}{2}.$$
 We use \lref{c_bound} and \lref{L2_time_in_norm} in both the following cases.\\
\emph{Case 1:} $|\alpha'|+|\beta'|+|\sigma'|\leq 8$.

In this case $\nu_{\alpha'',\beta'',\sigma''}\geq\nu_{\alpha,\beta,\sigma}$ and $\omega_{\alpha'',\beta'',\sigma''}\geq\omega_{\alpha,\beta,\sigma}$.
\begin{align*}
&\norm{\jap{v}^{2\nu_{\alpha,\beta,\sigma}}\jap{x-(t+1)v}^{2\omega_{\alpha,\beta,\sigma}}|\der g||\derv{'}{'}{'} \cm||\derv{''}{''}{''} g|}_{L^1([0,T];L^1_xL^1_v)}\\
&\lesssim (1+T)^{2|\beta|(1+\delta)}\norm{(1+t)^{-\frac{1}{2}-\delta-|\beta|(1+\delta)}\jap{v}^{\frac{1}{2}}\jap{v}^{\nu_{\alpha,\beta,\sigma}}\jap{x-(t+1)v}^{\omega_{\alpha,\beta,\sigma}}\der g}_{L^2([0,T];L^2_xL^2_v)}\\
&\quad\times \norm{(1+t)^{1+2\delta-|\beta'|(1+\delta)}\jap{v}^{-1}\derv{'}{'}{'}\cm}_{L^\infty([0,T];L^\infty_xL^\infty_v)}\\
&\quad\times \norm{(1+t)^{-\frac{1}{2}-\delta-|\beta''|(1+\delta)}\jap{v}^{\frac{1}{2}}\jap{v}^{\nu_{\alpha'',\beta'',\sigma''}}\jap{x-(t+1)v}^{\omega_{\alpha'',\beta'',\sigma''}}\derv{''}{''}{''}g}_{L^2([0,T];L^2_xL^2_v)}\\
&\lesssim (1+T)^{2|\beta|(1+\delta)}\eps^{\frac{3}{4}}\times(\sup_{t\in [0,T]}\eps^{\frac{3}{4}}(1+t)^{1+2\delta-|\beta'|(1+\delta)}(1+t)^{-\frac{3}{2}+|\beta'|(1+\delta)})\eps^{\frac{3}{4}}\\
&=\eps^{\frac{9}{4}}(1+T)^{2|\beta|(1+\delta)}.
\end{align*}
\emph{Case 2:} $|\alpha'|+|\beta'|+|\sigma'|\geq 9$.

In this case $\nu_{\alpha'',\beta'',\sigma''}\geq\nu_{\alpha,\beta,\sigma}+3$ and $\omega_{\alpha'',\beta'',\sigma''}\geq\omega_{\alpha,\beta,\sigma}+3$.\\
We use \lref{L_infty_x_L2_v} to go from $L^\infty_x$ to $L^2_x$ for the third term.\\
\begin{align*}
&\norm{\jap{v}^{2\nu_{\alpha,\beta,\sigma}}\jap{x-(t+1)v}^{2\omega_{\alpha,\beta,\sigma}}|\der g||\derv{'}{'}{'} \cm||\derv{''}{''}{''} g|}_{L^1([0,T];L^1_xL^1_v)}\\
&\lesssim (1+T)^{2|\beta|(1+\delta)}\norm{(1+t)^{-\frac{1}{2}-\delta-|\beta|(1+\delta)}\jap{v}^{\frac{1}{2}}\jap{v}^{\nu_{\alpha,\beta,\sigma}}\jap{x-(t+1)v}^{\omega_{\alpha,\beta,\sigma}}\der g}_{L^2([0,T];L^2_xL^2_v)}\\
&\quad\times \norm{(1+t)^{1+2\delta-|\beta'|(1+\delta)}\jap{v}^{-1}\derv{'}{'}{'}\cm}_{L^\infty([0,T];L^2_xL^\infty_v)}\\
&\quad\times \norm{(1+t)^{-\frac{1}{2}-\delta-|\beta''|(1+\delta)}\jap{v}^{\frac{1}{2}}\jap{v}^{\nu_{\alpha'',\beta'',\sigma''}-3}\jap{x-(t+1)v}^{\omega_{\alpha'',\beta'',\sigma''}-3}\derv{''}{''}{''}g}_{L^2([0,T];L^2_xL^2_v)}\\
&\lesssim (1+T)^{2|\beta|(1+\delta)}\eps^{\frac{3}{4}}\times(\sup_{t\in [0,T]}\eps^{\frac{3}{4}}(1+t)^{1+2\delta-|\beta'|(1+\delta)}(1+t)^{-\frac{3}{2}+|\beta'|(1+\delta)})\eps^{\frac{3}{4}}\\
&=\eps^{\frac{9}{4}}(1+T)^{2|\beta|(1+\delta)}.
\end{align*}
We proceeded in the same way Case 1 of \pref{T3_2} to deduce the penultimate inequality.
\end{proof}
\begin{proposition}\label{p.T5_1}
Let $|\alpha|+|\beta|+|\sigma|\leq 10$. Then the term $T^{\alpha,\beta,\sigma}_{5,1}$ is bounded as follows for all $T\in [0,T_{\boot})$,
$$T^{\alpha,\beta,\sigma}_{5,1}\lesssim \eps^2(1+T)^{2|\beta|(1+\delta)}.$$
\end{proposition}
\begin{proof}
We need to break into two cases- when less than 9 derivatives fall on the coefficient and when more than 8 fall on it. Each of these cases is further subdivided into three cases depending on which derivative falls on the coefficient.\\
The weights satify the relations 
$$\nu_{\alpha'',\beta'',\sigma''}=\nu_{\alpha,\beta,\sigma}+\frac{3|\alpha'|+|\beta'|+3|\sigma'|}{2}-\frac{1}{2}$$
and
$$\omega_{\alpha'',\beta'',\sigma''}=2\omega_{\alpha,\beta,\sigma}+\frac{|\alpha'|+|\beta'|+3|\sigma'|}{2}-\frac{1}{2}.$$
\emph{Case 1}: $|\alpha'|+|\beta'|+|\sigma'|\leq 8$.\\
\emph{Subcase 1a):} $|\beta'|\geq 1$.

In this case we have $\nu_{\alpha'',\beta'',\sigma''}\geq\nu_{\alpha,\beta,\sigma}$ and $\omega_{\alpha'',\beta'',\sigma''}\geq\omega_{\alpha,\beta,\sigma}$.\\
We use \lref{av_bound_dv} and \lref{L2_time_in_norm} to obtain
\begin{align*}
&\norm{\jap{v}^{2\nu_{\alpha,\beta,\sigma}}\jap{x-(t+1)v}^{2\omega_{\alpha,\beta,\sigma}}|\der g||\derv{'}{'}{'} \left(\a \frac{v_i}{\jap{v}}\right)||\derv{''}{''}{''} g|}_{L^1([0,T];L^1_xL^1_v)}\\
&\lesssim (1+T)^{2|\beta|(1+\delta)}\norm{(1+t)^{-\frac{1}{2}-\delta-|\beta|(1+\delta)}\jap{v}^{\frac{1}{2}}\jap{v}^{\nu_{\alpha,\beta,\sigma}}\jap{x-(t+1)v}^{\omega_{\alpha,\beta,\sigma}}\der g}_{L^2([0,T];L^2_xL^2_v)}\\
&\quad\times \norm{(1+t)^{1+2\delta-(|\beta'|-1)(1+\delta)}\jap{v}^{-1}\derv{'}{'}{'}\left(\a \frac{v_i}{\jap{v}}\right)}_{L^\infty([0,T];L^\infty_xL^\infty_v)}\\
&\quad\times \norm{(1+t)^{-\frac{1}{2}-\delta-|\beta''|(1+\delta)}\jap{v}^{\frac{1}{2}}\jap{v}^{\nu_{\alpha'',\beta'',\sigma''}}\jap{x-(t+1)v}^{\omega_{\alpha'',\beta'',\sigma''}}\derv{''}{''}{''}g}_{L^2([0,T];L^2_xL^2_v)}\\
&\lesssim (1+T)^{2|\beta|(1+\delta)}\eps^{\frac{3}{4}}\times(\sup_{t\in [0,T]}\eps^{\frac{3}{4}}(1+t)^{1+2\delta-(|\beta'|-1)(1+\delta)}(1+t)^{-\frac{5}{2}+|\beta'|(1+\delta)})\eps^{\frac{3}{4}}\\
&=\eps^{\frac{9}{4}}(1+T)^{2|\beta|(1+\delta)}.
\end{align*}
\emph{Subcase 1b):} $|\alpha'|\geq 1$.

In this case we have $\nu_{\alpha'',\beta'',\sigma''}\geq\nu_{\alpha,\beta,\sigma}+1$ and $\omega_{\alpha'',\beta'',\sigma''}\geq\omega_{\alpha,\beta,\sigma}$.\\
We use \lref{av_bound_dx} to obtain
\begin{align*}
&\norm{\jap{v}^{2\nu_{\alpha,\beta,\sigma}}\jap{x-(t+1)v}^{2\omega_{\alpha,\beta,\sigma}}|\der g||\derv{'}{'}{'} \left(\a \frac{v_i}{\jap{v}}\right)||\derv{''}{''}{''} g|}_{L^1([0,T];L^1_xL^1_v)}\\
&\lesssim (1+T)^{2|\beta|(1+\delta)}\norm{(1+t)^{-\frac{1}{2}-\delta-|\beta|(1+\delta)}\jap{v}^{\frac{1}{2}}\jap{v}^{\nu_{\alpha,\beta,\sigma}}\jap{x-(t+1)v}^{\omega_{\alpha,\beta,\sigma}}\der g}_{L^2([0,T];L^2_xL^2_v)}\\
&\quad\times \norm{(1+t)^{1+2\delta-(|\beta'|-1)(1+\delta)}\jap{v}^{-2}\derv{'}{'}{'}\left(\a \frac{v_i}{\jap{v}}\right)}_{L^\infty([0,T];L^\infty_xL^\infty_v)}\\
&\quad\times \norm{(1+t)^{-\frac{1}{2}-\delta-|\beta''|(1+\delta)}\jap{v}^{\frac{1}{2}}\jap{v}^{\nu_{\alpha'',\beta'',\sigma''}}\jap{x-(t+1)v}^{\omega_{\alpha'',\beta'',\sigma''}}\derv{''}{''}{''}g}_{L^2([0,T];L^2_xL^2_v)}\\
&\lesssim (1+T)^{2|\beta|(1+\delta)}\eps^{\frac{3}{4}}\times(\sup_{t\in [0,T]}\eps^{\frac{3}{4}}(1+t)^{1+2\delta-(|\beta'|-1)(1+\delta)}(1+t)^{-\frac{5}{2}+|\beta'|(1+\delta)})\eps^{\frac{3}{4}}\\
&=\eps^{\frac{9}{4}}(1+T)^{2|\beta|(1+\delta)}.
\end{align*}
\emph{Subcase 1c):} $|\sigma'|\geq 1$.

In this case we have $\nu_{\alpha'',\beta'',\sigma''}\geq\nu_{\alpha,\beta,\sigma}+1$ and $\omega_{\alpha'',\beta'',\sigma''}\geq\omega_{\alpha,\beta,\sigma}+1$.\\
We use \lref{av_bound_Y} to obtain
\begin{align*}
&\norm{\jap{v}^{2\nu_{\alpha,\beta,\sigma}}\jap{x-(t+1)v}^{2\omega_{\alpha,\beta,\sigma}}|\der g||\derv{'}{'}{'} \left(\a \frac{v_i}{\jap{v}}\right)||\derv{''}{''}{''} g|}_{L^1([0,T];L^1_xL^1_v)}\\
&\lesssim (1+T)^{2|\beta|(1+\delta)}\norm{(1+t)^{-\frac{1}{2}-\delta-|\beta|(1+\delta)}\jap{v}^{\frac{1}{2}}\jap{v}^{\nu_{\alpha,\beta,\sigma}}\jap{x-(t+1)v}^{\omega_{\alpha,\beta,\sigma}}\der g}_{L^2([0,T];L^2_xL^2_v)}\\
&\quad\times \norm{(1+t)^{1+2\delta-(|\beta'|-1)(1+\delta)}\jap{x-(t+1)v}^{-1}\jap{v}^{-1}\derv{'}{'}{'}\left(\a \frac{v_i}{\jap{v}}\right)}_{L^\infty([0,T];L^\infty_xL^\infty_v)}\\
&\quad\times \norm{(1+t)^{-\frac{1}{2}-\delta-|\beta''|(1+\delta)}\jap{v}^{\frac{1}{2}}\jap{v}^{\nu_{\alpha'',\beta'',\sigma''}}\jap{x-(t+1)v}^{\omega_{\alpha'',\beta'',\sigma''}}\derv{''}{''}{''}g}_{L^2([0,T];L^2_xL^2_v)}\\
&\lesssim (1+T)^{2|\beta|(1+\delta)}\eps^{\frac{3}{4}}\times(\sup_{t\in [0,T]}\eps^{\frac{3}{4}}(1+t)^{1+2\delta-(|\beta'|-1)(1+\delta)}(1+t)^{-\frac{5}{2}+|\beta'|(1+\delta)})\eps^{\frac{3}{4}}\\
&=\eps^{\frac{9}{4}}(1+T)^{2|\beta|(1+\delta)}.
\end{align*}
\emph{Case 2:} $|\alpha'|+|\beta'|+|\sigma'|\geq 9$.

In each case we have $\nu_{\alpha'',\beta'',\sigma''}\geq\nu_{\alpha,\beta,\sigma}+4$ and $\omega_{\alpha'',\beta'',\sigma''}\geq\omega_{\alpha,\beta,\sigma}+4$.\\
\emph{Subcase 2a):} $|\beta'|\geq 1$.\\
We use \lref{av_bound_dv}, \lref{L2_time_in_norm} and \lref{L_infty_x_L2_v} to obtain
\begin{align*}
&\norm{\jap{v}^{2\nu_{\alpha,\beta,\sigma}}\jap{x-(t+1)v}^{2\omega_{\alpha,\beta,\sigma}}|\der g||\derv{'}{'}{'} \left(\a \frac{v_i}{\jap{v}}\right)||\derv{''}{''}{''} g|}_{L^1([0,T];L^1_xL^1_v)}\\
&\lesssim (1+T)^{2|\beta|(1+\delta)}\norm{(1+t)^{-\frac{1}{2}-\delta-|\beta|(1+\delta)}\jap{v}^{\frac{1}{2}}\jap{v}^{\nu_{\alpha,\beta,\sigma}}\jap{x-(t+1)v}^{\omega_{\alpha,\beta,\sigma}}\der g}_{L^2([0,T];L^\infty_xL^2_v)}\\
&\quad\times \norm{(1+t)^{1+2\delta-(|\beta'|-1)(1+\delta)}\jap{v}^{-1}\derv{'}{'}{'}\left(\a \frac{v_i}{\jap{v}}\right)}_{L^\infty([0,T];L^2_xL^\infty_v)}\\
&\quad\times \norm{(1+t)^{-\frac{1}{2}-\delta-|\beta''|(1+\delta)}\jap{v}^{\frac{1}{2}}\jap{v}^{\nu_{\alpha'',\beta'',\sigma''}-4}\jap{x-(t+1)v}^{\omega_{\alpha'',\beta'',\sigma''}-4}\derv{''}{''}{''}g}_{L^\infty([0,T];L^\infty_xL^2_v)}\\
&\lesssim (1+T)^{2|\beta|(1+\delta)}\eps^{\frac{3}{4}}\times(\sup_{t\in [0,T]}\eps^{\frac{3}{4}}(1+t)^{1+2\delta-(|\beta'|-1)(1+\delta)}(1+t)^{-\frac{5}{2}+|\beta'|(1+\delta)})\\
&\times \norm{(1+t)^{-\frac{1}{2}-\delta-|\beta''|(1+\delta)}\jap{v}^{\frac{1}{2}}\jap{v}^{\nu_{\alpha''',\beta'',\sigma''}}\jap{x-(t+1)v}^{\omega_{\alpha''',\beta'',\sigma''}}\derv{'''}{''}{''}g}_{L^2([0,T];L^2_xL^2_v)}\\
&\lesssim \eps^{\frac{9}{4}}(1+T)^{2|\beta|(1+\delta)}.
\end{align*}
Here $|\alpha'''|\leq |\alpha''|+2$. Thus $\nu_{\alpha''',\beta'',\sigma''}\geq \nu_{\alpha'',\beta'',\sigma''}-3$ and $\omega_{\alpha''',\beta'',\sigma''}\geq \omega_{\alpha'',\beta'',\sigma''}-1$.\\
\emph{Subcase 2b):} $|\alpha'|\geq 1$.\\
We use \lref{av_bound_dx}, \lref{L2_time_in_norm} and \lref{L_infty_x_L2_v} to obtain
\begin{align*}
&\norm{\jap{v}^{2\nu_{\alpha,\beta,\sigma}}\jap{x-(t+1)v}^{2\omega_{\alpha,\beta,\sigma}}|\der g||\derv{'}{'}{'} \left(\a \frac{v_i}{\jap{v}}\right)||\derv{''}{''}{''} g|}_{L^1([0,T];L^1_xL^1_v)}\\
&\lesssim (1+T)^{2|\beta|(1+\delta)}\norm{(1+t)^{-\frac{1}{2}-\delta-|\beta|(1+\delta)}\jap{v}^{\frac{1}{2}}\jap{v}^{\nu_{\alpha,\beta,\sigma}}\jap{x-(t+1)v}^{\omega_{\alpha,\beta,\sigma}}\der g}_{L^2([0,T];L^\infty_xL^2_v)}\\
&\quad\times \norm{(1+t)^{1+2\delta-(|\beta'|-1)(1+\delta)}\jap{v}^{-2}\derv{'}{'}{'}\left(\a \frac{v_i}{\jap{v}}\right)}_{L^\infty([0,T];L^2_xL^\infty_v)}\\
&\quad\times \norm{(1+t)^{-\frac{1}{2}-\delta-|\beta''|(1+\delta)}\jap{v}^{\frac{1}{2}}\jap{v}^{\nu_{\alpha'',\beta'',\sigma''}-3}\jap{x-(t+1)v}^{\omega_{\alpha'',\beta'',\sigma''}-4}\derv{''}{''}{''}g}_{L^\infty([0,T];L^\infty_xL^2_v)}.
\end{align*}
Now we proceed as above.\\
\emph{Subcase 2c):} $|\sigma'|\geq 1$.\\
We use \lref{av_bound_Y}, \lref{L2_time_in_norm} and \lref{L_infty_x_L2_v} to obtain
\begin{align*}
&\norm{\jap{v}^{2\nu_{\alpha,\beta,\sigma}}\jap{x-(t+1)v}^{2\omega_{\alpha,\beta,\sigma}}|\der g||\derv{'}{'}{'} \left(\a \frac{v_i}{\jap{v}}\right)||\derv{''}{''}{''} g|}_{L^1([0,T];L^1_xL^1_v)}\\
&\lesssim (1+T)^{2|\beta|(1+\delta)}\norm{(1+t)^{-\frac{1}{2}-\delta-|\beta|(1+\delta)}\jap{v}^{\frac{1}{2}}\jap{v}^{\nu_{\alpha,\beta,\sigma}}\jap{x-(t+1)v}^{\omega_{\alpha,\beta,\sigma}}\der g}_{L^2([0,T];L^\infty_xL^2_v)}\\
&\quad\times \norm{(1+t)^{1+2\delta-(|\beta'|-1)(1+\delta)}\jap{x-(t+1)v}^{-1}\jap{v}^{-1}\derv{'}{'}{'}\left(\a \frac{v_i}{\jap{v}}\right)}_{L^\infty([0,T];L^2_xL^\infty_v)}\\
&\quad\times \norm{(1+t)^{-\frac{1}{2}-\delta-|\beta''|(1+\delta)}\jap{v}^{\frac{1}{2}}\jap{v}^{\nu_{\alpha'',\beta'',\sigma''}-3}\jap{x-(t+1)v}^{\omega_{\alpha'',\beta'',\sigma''}-4}\derv{''}{''}{''}g}_{L^\infty([0,T];L^\infty_xL^2_v)}.
\end{align*}
\end{proof}
\begin{proposition}\label{p.T5_2}
Let $|\alpha|+|\beta|+|\sigma|\leq 10$. Then the term $T^{\alpha,\beta,\sigma}_{5,2}$ is bounded as follows for all $T\in [0,T_{\boot})$,
$$T^{\alpha,\beta,\sigma}_{5,2}\lesssim \eps^2(1+T)^{2|\beta|(1+\delta)}.$$
\end{proposition}
\begin{proof}
Although we don't have any derivative hitting the coefficient term, we have one less $\jap{v}$ term. Thus,
\begin{align*}
&\max_j\norm{\jap{v}^{2\nu_{\alpha,\beta,\sigma}-1}\jap{x-(t+1)v}^{2\omega_{\alpha,\beta,\sigma}}|\der g|^2\left|\left(\a \frac{v_i}{\jap{v}}\right)\right|}_{L^1([0,T];L^1_xL^1_v)}\\
&\lesssim\max_{j}(1+T)^{2|\beta|(1+\delta)}\norm{(1+t)^{-\frac{1}{2}-\delta-|\beta|(1+\delta)}\jap{v}^{\nu_{\alpha,\beta,\sigma}}\jap{x-(t+1)v}^{\omega_{\alpha,\beta,\sigma}}\jap{v}^{\frac{1}{2}}\der g}_{L^2([0,T];L^2_xL^2_v)}^2\\
&\quad\times\norm{(1+t)^{1+2\delta}\jap{v}^{-2}\left(\a \frac{v_i}{\jap{v}}\right)}_{L^\infty([0,T];L^\infty_xL^\infty_v)}\\
&\lesssim (1+T)^{2|\beta|(1+\delta)}\eps^{\frac{3}{2}}\times (\sup_{t\in [0,T]}\eps^{\frac{3}{4}}(1+t)^{1+2\delta}(1+t)^{-\frac{3}{2}})=\eps^{\frac{9}{4}}(1+T)^{2|\beta|(1+\delta)}.
\end{align*}
\end{proof}
\begin{proposition}\label{p.T5_3}
Let $|\alpha|+|\beta|+|\sigma|\leq 10$. Then the term $T^{\alpha,\beta,\sigma}_{5,3}$ is bounded as follows for all $T\in [0,T_{\boot})$,
$$T^{\alpha,\beta,\sigma}_{5,3}\lesssim \eps^2(1+T)^{2|\beta|(1+\delta)}.$$
\end{proposition}
\begin{proof}
Since we have one less $\jap{x-(t+1)v}$ weight, we get
\begin{align*}
&\max_j\norm{\jap{v}^{(1+t)2\nu_{\alpha,\beta,\sigma}}\jap{x-(t+1)v}^{2\omega_{\alpha,\beta,\sigma}-1}|\der g|^2\left|\left(\a \frac{v_i}{\jap{v}}\right)\right|}_{L^1([0,T];L^1_xL^1_v)}\\
&\lesssim\max_{j}(1+T)^{2|\beta|(1+\delta)}\norm{(1+t)^{-\frac{1}{2}-\delta-|\beta|(1+\delta)}\jap{v}^{\nu_{\alpha,\beta,\sigma}}\jap{x-(t+1)v}^{\omega_{\alpha,\beta,\sigma}}\jap{v}^{\frac{1}{2}}\der g}_{L^2([0,T];L^2_xL^2_v)}^2\\
&\quad\times\norm{(1+t)^{1+2\delta}\jap{v}^{-1}\jap{x-(t+1)v}^{-1}\left(\a \frac{v_i}{\jap{v}}\right)}_{L^\infty([0,T];L^\infty_xL^\infty_v)}\\
&\lesssim (1+T)^{2|\beta|(1+\delta)}\eps^{\frac{3}{2}}\times (\sup_{t\in [0,T]}\eps^{\frac{3}{4}}(1+t)^{1+2\delta}(1+t)^{-\frac{3}{2}})=\eps^{\frac{9}{4}}(1+T)^{2|\beta|(1+\delta)}.
\end{align*}
\end{proof}
\begin{proposition}\label{p.T6_1}
Let $|\alpha|+|\beta|+|\sigma|\leq 10$. Then the term $T^{\alpha,\beta,\sigma}_{6,1}$ is bounded as follows for all $T\in [0,T_{\boot})$,
$$T^{\alpha,\beta,\sigma}_{6,1}\lesssim \eps^2(1+T)^{2|\beta|(1+\delta)}.$$
\end{proposition}
\begin{proof}
We need to break into two cases- when less than 9 derivatives fall on the coefficient and when more than 8 fall on it. Each of these cases is further subdivided into three cases depending on which derivative falls on the coefficient.\\
The weights satify the relations 
$$\nu_{\alpha'',\beta'',\sigma''}=\nu_{\alpha,\beta,\sigma}+\frac{3|\alpha'|+|\beta'|+3|\sigma'|}{2}$$
and
$$\omega_{\alpha'',\beta'',\sigma''}=2\omega_{\alpha,\beta,\sigma}+\frac{|\alpha'|+|\beta'|+3|\sigma'|}{2}.$$
\emph{Case 1:} $|\alpha'|+|\beta'|+|\sigma'|\leq 8$\\
\emph{Subcase 1a):} $|\beta'|\geq 1$.

In this case we have $\nu_{\alpha'',\beta'',\sigma''}\geq\nu_{\alpha,\beta,\sigma}$ and $\omega_{\alpha'',\beta'',\sigma''}\geq\omega_{\alpha,\beta,\sigma}$.\\
We use \lref{a_bound_dv} (with the appropriate change now that we have a $\jap{v}$ weight in the denominator) and \lref{L2_time_in_norm} to obtain
\begin{align*}
 &\norm{\jap{v}^{2\nu_{\alpha,\beta,\sigma}}\jap{x-(t+1)v}^{2\omega_{\alpha,\beta,\sigma}}|\der g|\left|\derv{'}{'}{'}\left(\frac{\bar{a}_{ii}}{\jap{v}}\right)\right||\derv{''}{''}{''} g|}_{L^1([0,T];L^1_xL^1_v)}\\
&\lesssim (1+T)^{2|\beta|(1+\delta)}\norm{(1+t)^{-\frac{1}{2}-\delta-|\beta|(1+\delta)}\jap{v}^{\frac{1}{2}}\jap{v}^{\nu_{\alpha,\beta,\sigma}}\jap{x-(t+1)v}^{\omega_{\alpha,\beta,\sigma}}\der g}_{L^2([0,T];L^2_xL^2_v)}\\
&\quad\times \norm{(1+t)^{1+2\delta-|\beta'|(1+\delta)}\jap{v}^{-1}\derv{'}{'}{'}\left(\frac{\bar{a}_{ii}}{\jap{v}}\right)}_{L^\infty([0,T];L^\infty_xL^\infty_v)}\\
&\quad\times \norm{(1+t)^{-\frac{1}{2}-\delta-|\beta''|(1+\delta)}\jap{v}^{\frac{1}{2}}\jap{v}^{\nu_{\alpha'',\beta'',\sigma''}}\jap{x-(t+1)v}^{\omega_{\alpha'',\beta'',\sigma''}}\derv{''}{''}{''}g}_{L^2([0,T];L^2_xL^2_v)}\\
&\lesssim (1+T)^{2|\beta|(1+\delta)}\eps^{\frac{3}{4}}\times(\sup_{t\in [0,T]}\eps^{\frac{3}{4}}(1+t)^{1+2\delta-|\beta'|(1+\delta)}(1+t)^{-\frac{5}{2}+|\beta'|(1+\delta)})\eps^{\frac{3}{4}}\\
&=\eps^{\frac{9}{4}}(1+T)^{2|\beta|(1+\delta)}.
\end{align*}
\emph{Subcase 1b):} $|\alpha'|\geq 1$.

In this case we have $\nu_{\alpha'',\beta'',\sigma''}\geq\nu_{\alpha,\beta,\sigma}+1$ and $\omega_{\alpha'',\beta'',\sigma''}\geq\omega_{\alpha,\beta,\sigma}$.\\
We use \lref{a_bound_dx} (with the appropriate changes) and \lref{L2_time_in_norm} to obtain
\begin{align*}
 &\norm{\jap{v}^{2\nu_{\alpha,\beta,\sigma}}\jap{x-(t+1)v}^{2\omega_{\alpha,\beta,\sigma}}|\der g|\left|\derv{'}{'}{'}\left(\frac{\bar{a}_{ii}}{\jap{v}}\right)\right||\derv{''}{''}{''} g|}_{L^1([0,T];L^1_xL^1_v)}\\
&\lesssim (1+T)^{2|\beta|(1+\delta)}\norm{(1+t)^{-\frac{1}{2}-\delta-|\beta|(1+\delta)}\jap{v}^{\frac{1}{2}}\jap{v}^{\nu_{\alpha,\beta,\sigma}}\jap{x-(t+1)v}^{\omega_{\alpha,\beta,\sigma}}\der g}_{L^2([0,T];L^2_xL^2_v)}\\
&\quad\times \norm{(1+t)^{1+2\delta-|\beta'|(1+\delta)}\jap{v}^{-2}\derv{'}{'}{'}\left(\frac{\bar{a}_{ii}}{\jap{v}}\right)}_{L^\infty([0,T];L^\infty_xL^\infty_v)}\\
&\quad\times \norm{(1+t)^{-\frac{1}{2}-\delta-|\beta''|(1+\delta)}\jap{v}^{\frac{1}{2}}\jap{v}^{\nu_{\alpha'',\beta'',\sigma''}}\jap{x-(t+1)v}^{\omega_{\alpha'',\beta'',\sigma''}}\derv{''}{''}{''}g}_{L^2([0,T];L^2_xL^2_v)}\\
&\lesssim (1+T)^{2|\beta|(1+\delta)}\eps^{\frac{3}{4}}\times(\sup_{t\in [0,T]}\eps^{\frac{3}{4}}(1+t)^{1+2\delta-|\beta'|(1+\delta)}(1+t)^{-\frac{5}{2}+|\beta'|(1+\delta)})\eps^{\frac{3}{4}}=\eps^{\frac{9}{4}}(1+T)^{2|\beta|(1+\delta)}.
\end{align*}
\emph{Subcase 1c):} $|\sigma'|\geq 1$.

In this case we have $\nu_{\alpha'',\beta'',\sigma''}\geq\nu_{\alpha,\beta,\sigma}$ and $\omega_{\alpha'',\beta'',\sigma''}\geq\omega_{\alpha,\beta,\sigma}+1$.\\
We use \lref{a_bound_Y} (with the appropriate changes) and \lref{L2_time_in_norm} to obtain
\begin{align*}
 &\norm{\jap{v}^{2\nu_{\alpha,\beta,\sigma}}\jap{x-(t+1)v}^{2\omega_{\alpha,\beta,\sigma}}|\der g|\left|\derv{'}{'}{'}\left(\frac{\bar{a}_{ii}}{\jap{v}}\right)\right||\derv{''}{''}{''} g|}_{L^1([0,T];L^1_xL^1_v)}\\
&\lesssim (1+T)^{2|\beta|(1+\delta)}\norm{(1+t)^{-\frac{1}{2}-\delta-|\beta|(1+\delta)}\jap{v}^{\frac{1}{2}}\jap{v}^{\nu_{\alpha,\beta,\sigma}}\jap{x-(t+1)v}^{\omega_{\alpha,\beta,\sigma}}\der g}_{L^2([0,T];L^2_xL^2_v)}\\
&\quad\times \norm{(1+t)^{1+2\delta-|\beta'|(1+\delta)}\jap{v}^{-1}\jap{x-(t+1)v}^{-1}\derv{'}{'}{'}\left(\frac{\bar{a}_{ii}}{\jap{v}}\right)}_{L^\infty([0,T];L^\infty_xL^\infty_v)}\\
&\quad\times \norm{(1+t)^{-\frac{1}{2}-\delta-|\beta''|(1+\delta)}\jap{v}^{\frac{1}{2}}\jap{v}^{\nu_{\alpha'',\beta'',\sigma''}}\jap{x-(t+1)v}^{\omega_{\alpha'',\beta'',\sigma''}}\derv{''}{''}{''}g}_{L^2([0,T];L^2_xL^2_v)}\\
&\lesssim (1+T)^{2|\beta|(1+\delta)}\eps^{\frac{3}{4}}\times(\sup_{t\in [0,T]}\eps^{\frac{3}{4}}(1+t)^{1+2\delta-|\beta'|(1+\delta)}(1+t)^{-\frac{5}{2}+|\beta'|(1+\delta)})\eps^{\frac{3}{4}}=\eps^{\frac{9}{4}}(1+T)^{2|\beta|(1+\delta)}.
\end{align*}
\emph{Case 2:} $|\alpha'|+|\beta'|+|\sigma'|\geq 9$.\\
In each case we have $\nu_{\alpha'',\beta'',\sigma''}\geq\nu_{\alpha,\beta,\sigma}+4$ and $\omega_{\alpha'',\beta'',\sigma''}\geq\omega_{\alpha,\beta,\sigma}+4$.\\
We get the required result by proceeding the same way as in Case 2 of \pref{T5_1} but using the relevant lemmas for $\derv{'}{'}{'} \left(\frac{\bar{a}_{ii}}{\jap{v}}\right)$.
\end{proof}
\begin{proposition}\label{p.T6_2}
Let $|\alpha|+|\beta|+|\sigma|\leq 10$. Then the term $T^{\alpha,\beta,\sigma}_{6,2}$ is bounded as follows for all $T\in [0,T_{\boot})$,
$$T^{\alpha,\beta,\sigma}_{6,2}\lesssim \eps^2(1+T)^{2|\beta|(1+\delta)}.$$
\end{proposition}
\begin{proof}
\emph{Case 1:} $|\alpha'|+|\beta'|+|\sigma'|\leq 8$.

In this case we have $\nu_{\alpha'',\beta'',\sigma''}\geq\nu_{\alpha,\beta,\sigma}$ and $\omega_{\alpha'',\beta'',\sigma''}\geq\omega_{\alpha,\beta,\sigma}$.\\
We use \lref{a2v_bound} and \lref{L2_time_in_norm} to get
\begin{align*}
 &\norm{\jap{v}^{2\nu_{\alpha,\beta,\sigma}}\jap{x-(t+1)v}^{2\omega_{\alpha,\beta,\sigma}}|\der g|\left|\derv{'}{'}{'}\left(\frac{\a v_iv_j}{\jap{v}^2} \right)\right||\derv{''}{''}{''} g|}_{L^1([0,T];L^1_xL^1_v)}\\
&\lesssim (1+T)^{2|\beta|(1+\delta)}\norm{(1+t)^{-\frac{1}{2}-\delta-|\beta|(1+\delta)}\jap{v}^{\frac{1}{2}}\jap{v}^{\nu_{\alpha,\beta,\sigma}}\jap{x-(t+1)v}^{\omega_{\alpha,\beta,\sigma}}\der g}_{L^2([0,T];L^2_xL^2_v)}\\
&\quad\times \norm{(1+t)^{1+2\delta-|\beta'|(1+\delta)}\jap{v}^{-1}\derv{'}{'}{'}\derv{'}{'}{'}\left(\frac{\a v_iv_j}{\jap{v}^2} \right)}_{L^\infty([0,T];L^\infty_xL^\infty_v)}\\
&\quad\times \norm{(1+t)^{-\frac{1}{2}-\delta-|\beta''|(1+\delta)}\jap{v}^{\frac{1}{2}}\jap{v}^{\nu_{\alpha'',\beta'',\sigma''}}\jap{x-(t+1)v}^{\omega_{\alpha'',\beta'',\sigma''}}\derv{''}{''}{''}g}_{L^2([0,T];L^2_xL^2_v)}\\
&\lesssim (1+T)^{2|\beta|(1+\delta)}\eps^{\frac{3}{4}}\times(\sup_{t\in [0,T]}\eps^{\frac{3}{4}}(1+t)^{1+2\delta-|\beta'|(1+\delta)}(1+t)^{-\frac{3}{2}+|\beta'|(1+\delta)})\eps^{\frac{3}{4}}=\eps^{\frac{9}{4}}(1+T)^{2|\beta|(1+\delta)}.
\end{align*}
\emph{Case 2:} $|\alpha'|+|\beta'|+|\sigma'|\geq 8$.

In this case we have $\nu_{\alpha'',\beta'',\sigma''}\geq\nu_{\alpha,\beta,\sigma}+4$ and $\omega_{\alpha'',\beta'',\sigma''}\geq\omega_{\alpha,\beta,\sigma}+4$.\\
We use \lref{a2v_bound} and \lref{L2_time_in_norm} to get
\begin{align*}
 &\norm{\jap{v}^{2\nu_{\alpha,\beta,\sigma}}\jap{x-(t+1)v}^{2\omega_{\alpha,\beta,\sigma}}|\der g|\left|\derv{'}{'}{'}\left(\frac{\a v_iv_j}{\jap{v}^2} \right)\right||\derv{''}{''}{''} g|}_{L^1([0,T];L^1_xL^1_v)}\\
&\lesssim (1+T)^{2|\beta|(1+\delta)}\norm{(1+t)^{-\frac{1}{2}-\delta-|\beta|(1+\delta)}\jap{v}^{\frac{1}{2}}\jap{v}^{\nu_{\alpha,\beta,\sigma}}\jap{x-(t+1)v}^{\omega_{\alpha,\beta,\sigma}}\der g}_{L^2([0,T];L^2_xL^2_v)}\\
&\quad\times \norm{(1+t)^{1+2\delta-|\beta'|(1+\delta)}\jap{v}^{-1}\derv{'}{'}{'}\derv{'}{'}{'}\left(\frac{\a v_iv_j}{\jap{v}^2} \right)}_{L^\infty([0,T];L^2_xL^\infty_v)}\\
&\quad\times \norm{(1+t)^{-\frac{1}{2}-\delta-|\beta''|(1+\delta)}\jap{v}^{\frac{1}{2}}\jap{v}^{\nu_{\alpha'',\beta'',\sigma''}-3}\jap{x-(t+1)v}^{\omega_{\alpha'',\beta'',\sigma''}-3}\derv{''}{''}{''}g}_{L^2([0,T];L^\infty_xL^2_v)}\\
&\lesssim (1+T)^{2|\beta|(1+\delta)}\eps^{\frac{3}{4}}\times(\sup_{t\in [0,T]}\eps^{\frac{3}{4}}(1+t)^{1+2\delta-|\beta'|(1+\delta)}(1+t)^{-\frac{3}{2}+|\beta'|(1+\delta)})\eps^{\frac{3}{4}}=\eps^{\frac{9}{4}}(1+T)^{2|\beta|(1+\delta)}.
\end{align*}

The penultimate inequality follows because we apply Sobolev embedding in space at the cost of two spatial derivatives but we have enough room in our weights to accomodate them.
\end{proof}
\begin{proposition}\label{p.A4}
Let $|\alpha|+|\beta|+|\sigma|\leq 10$. Then the term $\mathcal{A}^{\alpha,\beta,\sigma}_{4}$ is bounded as follows for all $T\in [0,T_{\boot})$,
$$\mathcal{A}^{\alpha,\beta,\sigma}_{4}\lesssim \eps^2(1+T)^{2|\beta|(1+\delta)}.$$
\end{proposition}
\begin{proof}
We use \lref{a_bound} and \lref{L2_time_in_norm} to get
\begin{align*}
&\norm{\jap{v}^{2\nu_{\alpha,\beta,\sigma}-2}\jap{x-(t+1)v}^{2\omega_{\alpha,\beta,\sigma}}\a (\der g)^2}_{L^1([0,T];L^1_xL^1_v)}\\
&\lesssim(1+T)^{|\beta|(1+\delta)}\norm{(1+t)^{-\frac{1}{2}-\delta-|\beta|(1+\delta)}\jap{v}^{\nu_{\alpha,\beta,\sigma}}\jap{x-(t+1)v}^{\omega_{\alpha,\beta,\sigma}}\der g}_{L^2([0,T];L^2_xL^2_v)}^2\\
&\quad\times \norm{(1+t)^{1+2\delta}\jap{v}^{-2}\a}_{L^\infty([0,T];L^\infty_xL^\infty_v)}\\
&\lesssim (1+T)^{|\beta|(1+\delta)}\eps^{\frac{3}{2}}\times (\sup_{t\in[0,T]}\eps^{\frac{3}{4}}(1+t)^{1+2\delta}(1+t)^{-\frac{3}{2}})\\
&\lesssim \eps^{\frac{9}{4}}(1+T)^{|\beta|(1+\delta)}.
\end{align*}
\end{proof}
\begin{proposition}\label{p.A5}
Let $|\alpha|+|\beta|+|\sigma|\leq 10$. Then the term $\mathcal{A}^{\alpha,\beta,\sigma}_{5}$ is bounded as follows for all $T\in [0,T_{\boot})$,
$$\mathcal{A}^{\alpha,\beta,\sigma}_{5}\lesssim \eps^2(1+T)^{2|\beta|(1+\delta)}.$$
\end{proposition}
\begin{proof}
We use \lref{a_bound_Y} and \lref{L2_time_in_norm} to get
\begin{align*}
&\norm{(1+t)\jap{v}^{2\nu_{\alpha,\beta,\sigma}-1}\jap{x-(t+1)v}^{2\omega_{\alpha,\beta,\sigma}-1}\a (\der g)^2}_{L^1([0,T];L^1_xL^1_v)}\\
&\lesssim(1+T)^{|\beta|(1+\delta)}\norm{(1+t)^{-\frac{1}{2}-\delta-|\beta|(1+\delta)}\jap{v}^{\nu_{\alpha,\beta,\sigma}}\jap{x-(t+1)v}^{\omega_{\alpha,\beta,\sigma}}\der g}_{L^2([0,T];L^2_xL^2_v)}^2\\
&\quad\times \norm{(1+t)^{2+2\delta}\jap{x-(t+1)v}^{-1}\jap{v}^{-1}\a}_{L^\infty([0,T];L^\infty_xL^\infty_v)}\\
&\lesssim (1+T)^{|\beta|(1+\delta)}\eps^{\frac{3}{2}}\times (\sup_{t\in[0,T]}\eps^{\frac{3}{4}}(1+t)^{1+2\delta}(1+t)^{-\frac{5}{2}})\\
&\lesssim \eps^{\frac{9}{4}}(1+T)^{|\beta|(1+\delta)}.
\end{align*}
\end{proof}
\begin{proposition}\label{p.A6}
Let $|\alpha|+|\beta|+|\sigma|\leq 10$. Then the term $\mathcal{A}^{\alpha,\beta,\sigma}_{6}$ is bounded as follows for all $T\in [0,T_{\boot})$,
$$\mathcal{A}^{\alpha,\beta,\sigma}_{6}\lesssim \eps^2(1+T)^{2|\beta|(1+\delta)}.$$
\end{proposition}
\begin{proof}
We use \lref{a_bound_2Y} and \lref{L2_time_in_norm} to get
\begin{align*}
&\norm{(1+t)^2\jap{v}^{2\nu_{\alpha,\beta,\sigma}}\jap{x-(t+1)v}^{2\omega_{\alpha,\beta,\sigma}-2}\a (\der g)^2}_{L^1([0,T];L^1_xL^1_v)}\\
&\lesssim(1+T)^{|\beta|(1+\delta)}\norm{(1+t)^{-\frac{1}{2}-\delta-|\beta|(1+\delta)}\jap{v}^{\nu_{\alpha,\beta,\sigma}}\jap{x-(t+1)v}^{\omega_{\alpha,\beta,\sigma}}\der g}_{L^2([0,T];L^2_xL^2_v)}^2\\
&\quad\times \norm{(1+t)^{3+2\delta}\jap{x-(t+1)v}^{-1}\jap{v}^{-2}\a}_{L^\infty([0,T];L^\infty_xL^\infty_v)}\\
&\lesssim (1+T)^{|\beta|(1+\delta)}\eps^{\frac{3}{2}}\times (\sup_{t\in[0,T]}\eps^{\frac{3}{4}}(1+t)^{1+2\delta}(1+t)^{-\frac{7}{2}})\\
&\lesssim \eps^{\frac{9}{4}}(1+T)^{|\beta|(1+\delta)}.
\end{align*}
\end{proof}
\subsection{Putting things together} We now combine the propositions in \sref{errors} to obtain,
\begin{proposition}\label{p.int_est}
Let $|\alpha|+|\beta|+|\sigma|\leq 10$. Then for every $\eta>0$, there is a constant $C_\eta>0$(depending on $\eta$, $d_0$ and $\gamma$) such that the following estimate holds for all $T\in [0,T_{\boot})$:
\begin{align*}
&\norm{\jap{v}^{\nu_{\alpha,\beta,\sigma}}\jap{x-(t+1)v}^{\omega_{\alpha,\beta,\sigma}}\der g}_{L^\infty([0,T];L^2_xL^2_v)}^2\\
&\quad+\norm{(1+t)^{-\frac{1}{2}-\frac{\delta}{2}}\jap{v}^{\frac{1}{2}}\jap{v}^{\nu_{\alpha,\beta,\sigma}}\jap{x-(t+1)v}^{\omega_{\alpha,\beta,\sigma}}\der g}_{L^2([0,T];L^2_xL^2_v)}^2\\
&\leq C_\eta(\eps^2(1+T)^{2|\beta|(1+\delta)}+\sum_{\substack{|\beta'|\leq|\beta|, |\sigma'|\leq |\sigma|\\ |\beta'|+|\sigma'|\leq |\beta|+|\sigma|-1}} \norm{(1+t)^{-\frac{1+\delta}{2}}\jap{v}^{\nu_{\alpha,\beta',\sigma'}}\jap{x-(t+1)v}^{\omega_{\alpha,\beta',\sigma'}}\derv{}{'}{'} g}^2_{L^2([0,T];L^2_xL^2_v)}\\
&\quad+(1+T)^{2(1+\delta)}\sum_{\substack{|\alpha'|\leq|\alpha|+1\\ |\beta'|\leq |\beta|-1}} \norm{(1+t)^{-\frac{1+\delta}{2}}\jap{v}^{\frac{1}{2}}\jap{v}^{\nu_{\alpha',\beta',\sigma}}\jap{x-(t+1)v}^{\omega_{\alpha',\beta',\sigma}}\derv{'}{'}{} g}_{L^2([0,T];L^2_xL^2_v)}^2)\\
&\quad+\eta \norm{(1+t)^{-\frac{1+\delta}{2}}\jap{v}^{\nu_{\alpha,\beta,\sigma}}\jap{x-(t+1)v}^{\omega_{\alpha,\beta,\sigma}}\der g}^2_{L^2([0,T];L^2_xL^2_v)}\\
&\quad+\eta\norm{(1+t)^{-\frac{1+\delta}{2}}\jap{v}^{\frac{1}{2}}\jap{v}^{\nu_{\alpha,\beta,\sigma}}\jap{x-(t+1)v}^{\omega_{\alpha,\beta,\sigma}}\der g}_{L^2([0,T];L^2_xL^2_v)}^2.
\end{align*}
\end{proposition}
\begin{proposition}\label{p.ene_est_1}
Let $|\alpha|+|\beta|+|\sigma|\leq 10$. Then the following estimate holds for all $T\in [0,T_{\boot})$:
\begin{align*}
&\norm{\jap{v}^{\nu_{\alpha,\beta,\sigma}}\jap{x-(t+1)v}^{\omega_{\alpha,\beta,\sigma}}\der g}_{L^\infty([0,T];L^2_xL^2_v)}^2\\
&\quad+\norm{(1+t)^{-\frac{1}{2}-\frac{\delta}{2}}\jap{v}^{\frac{1}{2}}\jap{v}^{\nu_{\alpha,\beta,\sigma}}\jap{x-(t+1)v}^{\omega_{\alpha,\beta,\sigma}}\der g}_{L^2([0,T];L^2_xL^2_v)}^2\\
&\lesssim \eps^2(1+T)^{2|\beta|(1+\delta)}+\sum_{\substack{|\beta'|\leq|\beta|, |\sigma'|\leq |\sigma|\\ |\beta'|+|\sigma'|\leq |\beta|+|\sigma|-1}} \norm{(1+t)^{-\frac{1+\delta}{2}}\jap{v}^{\nu_{\alpha,\beta',\sigma'}}\jap{x-(t+1)v}^{\omega_{\alpha,\beta',\sigma'}}\derv{}{'}{'} g}^2_{L^\infty([0,T];L^2_xL^2_v)}\\
&\quad+(1+T)^{2(1+\delta)}\sum_{\substack{|\alpha'|\leq|\alpha|+1\\ |\beta'|\leq |\beta|-1}} \norm{(1+t)^{-\frac{1+\delta}{2}}\jap{v}^{\frac{1}{2}}\jap{v}^{\nu_{\alpha',\beta',\sigma}}\jap{x-(t+1)v}^{\omega_{\alpha',\beta',\sigma}}\derv{'}{'}{} g}_{L^2([0,T];L^2_xL^2_v)}^2.
\end{align*}
Here, by our convention, if $|\beta|+|\sigma|=0$, then the last two terms on the RHS are not present.
\end{proposition}
\begin{proof}
We just apply \pref{int_est} with $\eta=\frac{1}{4}$. Then the following term is absorbed on the LHS
\begin{align*}
\frac{1}{2}\norm{(1+t)^{-\frac{1+\delta}{2}}\jap{v}^{\frac{1}{2}}\jap{v}^{\nu_{\alpha,\beta,\sigma}}\jap{x-(t+1)v}^{\omega_{\alpha,\beta,\sigma}}\der g}_{L^2([0,T];L^2_xL^2_v)}^2.
\end{align*}
Since $\eta$ is fixed, $C_\eta$ is just a constant depending on $d_0$ and $\gamma$. We thus get the desired inequality.
\end{proof}
\begin{proposition}\label{p.ene_est}
Let $|\alpha|+|\beta|+|\sigma|\leq 10$. Then the following estimate holds for all $T\in [0,T_{\boot})$:
\begin{align*}
&\norm{\jap{v}^{\nu_{\alpha,\beta,\sigma}}\jap{x-(t+1)v}^{\omega_{\alpha,\beta,\sigma}}\der g}_{L^\infty([0,T];L^2_xL^2_v)}^2\\
&+\norm{(1+t)^{-\frac{1}{2}-\frac{\delta}{2}}\jap{v}^{\frac{1}{2}}\jap{v}^{\nu_{\alpha,\beta,\sigma}}\jap{x-(t+1)v}^{\omega_{\alpha,\beta,\sigma}}\der g}_{L^2([0,T];L^2_xL^2_v)}^2\lesssim \eps^2(1+T)^{2|\beta|(1+\delta)}.
\end{align*}
\end{proposition}
\begin{proof}
The proof proceeds by induction on $|\beta|+|\sigma|$.\\
\emph{Step 1: Base Case: $|\beta|+|\sigma|=0$.} Applying \pref{ene_est_1} when $|\beta|+|\sigma|=0$, the last two terms on the RHS are not present. Hence we immediately have
\begin{align*}
&\norm{\jap{v}^{\nu_{\alpha,0,0}}\jap{x-(t+1)v}^{\omega_{\alpha,0,0}}\part^{\alpha}_x g}_{L^\infty([0,T];L^2_xL^2_v)}^2\\
&+\norm{(1+t)^{-\frac{1}{2}-\frac{\delta}{2}}\jap{v}^{\frac{1}{2}}\jap{v}^{\nu_{\alpha,0,0}}\jap{x-(t+1)v}^{\omega_{\alpha,0,0}}\part^{\alpha}_x g}_{L^2([0,T];L^2_xL^2_v)}^2\lesssim \eps^2.
\end{align*}
\emph{Step 2: Inductive step.} Assume by induction that there exists a $B\in \N$ such that whenever $|\alpha|+|\beta|+|\sigma|\leq 10$ and $|\beta|+|\sigma|\leq B-1$,
\begin{align*}
&\norm{\jap{v}^{\nu_{\alpha,\beta,\sigma}}\jap{x-(t+1)v}^{\omega_{\alpha,\beta,\sigma}}\der g}_{L^\infty([0,T];L^2_xL^2_v)}^2\\
&+\norm{(1+t)^{-\frac{1}{2}-\frac{\delta}{2}}\jap{v}^{\frac{1}{2}}\jap{v}^{\nu_{\alpha,\beta,\sigma}}\jap{x-(t+1)v}^{\omega_{\alpha,\beta,\sigma}}\der g}_{L^2([0,T];L^2_xL^2_v)}^2\lesssim \eps^2(1+T)^{2|\beta|(1+\delta)}.
\end{align*}
Now take some multi-indices $\alpha,$ $\beta$ and $\sigma$ such that $|\alpha|+|\beta|+|\sigma|\leq 10$ and $|\beta|+|\sigma|=B$. We will show that the estimate as in the statement of the proposition holds for this choice of $(\alpha,\beta,\sigma)$.\\
By  \pref{ene_est_1} and the inductive hypothesis,
\begin{align*}
&\norm{\jap{v}^{\nu_{\alpha,\beta,\sigma}}\jap{x-(t+1)v}^{\omega_{\alpha,\beta,\sigma}}\der g}_{L^\infty([0,T];L^2_xL^2_v)}^2\\
&\quad+\norm{(1+t)^{-\frac{1}{2}-\frac{\delta}{2}}\jap{v}^{\frac{1}{2}}\jap{v}^{\nu_{\alpha,\beta,\sigma}}\jap{x-(t+1)v}^{\omega_{\alpha,\beta,\sigma}}\der g}_{L^2([0,T];L^2_xL^2_v)}^2\\
&\lesssim \eps^2(1+T)^{2|\beta|(1+\delta)}+\sum_{\substack{|\beta'|\leq|\beta|, |\sigma'|\leq |\sigma|\\ |\beta'|+|\sigma'|\leq |\beta|+|\sigma|-1}} \norm{(1+t)^{-\frac{1+\delta}{2}}\jap{v}^{\nu_{\alpha,\beta',\sigma'}}\jap{x-(t+1)v}^{\omega_{\alpha,\beta',\sigma'}}\derv{}{'}{'} g}^2_{L^2([0,T];L^2_xL^2_v)}\\
&\quad+(1+T)^{2(1+\delta)}\sum_{\substack{|\alpha'|\leq|\alpha|+1\\ |\beta'|\leq |\beta|-1}} \norm{(1+t)^{-\frac{1+\delta}{2}}\jap{v}^{\frac{1}{2}}\jap{v}^{\nu_{\alpha',\beta',\sigma}}\jap{x-(t+1)v}^{\omega_{\alpha',\beta',\sigma}}\derv{'}{'}{} g}_{L^2([0,T];L^2_xL^2_v)}^2\\
&\lesssim \eps^2(1+T)^{2|\beta|(1+\delta)}+\eps^2(\sum_{|\beta'|\leq |\beta|}(1+T)^{2|\beta'|(1+\delta)})+\eps^2(\sum_{|\beta'|\leq |\beta|-1}(1+T)^{2(1+\delta)}(1+T)^{2|\beta'|(1+\delta)})\\
&\lesssim \eps^2(1+T)^{2|\beta|(1+\delta)}.
\end{align*}
We thus get the desired result by induction.
\end{proof}
\pref{ene_est} also completes the proof of \tref{boot}.
\section{Proof of \tref{global}}
\begin{proof}[Proof of \tref{global}] Let 
\begin{align*}
T_{\max}:=\sup\{T\in[0,\infty):&\text{ there exists a unique solution } f:[0,T]\times \R^3\times \R^3 \text{ to } \eref{landau} \text{ with}\\
& f\geq 0,\: f|_{t=0}=f_{\ini} \text{ and satisfying }\eref{local_est} \text{ for }N_{\alpha,\beta}=20-\frac{1}{2}(|\alpha|+|\beta|)\\
& \text{ such that the bootstrap assumption } \eref{boot_assumption_1} \text{ holds}\}.
\end{align*}
Note that by \cref{local_cor}, $T_{\max}>0$.

We will prove that $T_{\max}=\infty$. Assume for the sake of contradiction that $T_{\max}<\infty$.\\
From the definition of $T_{\max}$, we have that the assumptions of \tref{boot} hold for $T_{\boot}=T_{\max}$.\\
Therefore, by \tref{boot} (with $|\sigma|=0$) we get 
\begin{equation}\label{e.uni}
\sum \limits_{|\alpha|+|\beta|\leq 10}\norm{\jap{x-(t+1)v}^{20-\frac{1}{2}(|\alpha|+|\beta|)}\jap{v}^{20-\frac{3}{2}|\alpha|-\frac{1}{2}|\beta|}\part^{\alpha}_x\part^\beta_v(e^{d(t)\jap{v}}f)}_{L^\infty([0,T_{\max});L^2_xL^2_v)}\lesssim \eps.
\end{equation}

Take an increasing sequence $\{t_n\}_{n=1}^\infty\subset [0,T_{\max})$ such that  $t_n \to T_{\max}$. By the uniform bound \eref{uni} and the local existence result \cref{local_cor}, there exists $T_{\sm}$ such that the unique solution exists on $[0,t_n+T_{\sm}]\times \R^3\times \R^3$. In particular, taking $n$ sufficiently large, we have constructed a solution beyod the time $T_{\max}$, up to, $T_{\max}+\frac{1}{2}T_{\sm}$. The solution moveover satisfies \eref{local_est} with $N_{\alpha,\beta}=20-\frac{1}{2}(|\alpha|+|\beta|)$.

We next prove that the estimate \eref{boot_assumption_1} holds slightly beyond $T_{\max}$. To that end we employ \tref{boot},
\begin{equation}\label{e.boot_cont}
\text{the estimate \eref{boot_assumption_1} holds in }[0,T_{\max}) \text{ with } \eps^{\frac{3}{4}} \text{ replaced by }C_{d_0,\gamma}\eps.
\end{equation}

By the local existence result in \cref{local_cor}, for $|\alpha|+|\beta|\leq 10$, $$\jap{x-(t+1)v}^{20-\frac{1}{2}(|\alpha|+|\beta|)}\jap{v}^{20-\frac{3}{2}|\alpha|-\frac{1}{2}|\beta|}\part^\alpha_x\part^\beta_v g(t,x,v)\in C^0([0,T_{\max}+\frac{1}{2}T_{\sm}];L^2_xL^2_v).$$
Since $Y=(t+1)\part_x+\part_v$, for all $|\alpha|+|\beta|+|\sigma|\leq 10$, we also have  $$\jap{x-(t+1)v}^{20-\frac{1}{2}(|\alpha|+|\beta|)-\frac{3}{2}|\sigma|}\jap{v}^{20-\frac{3}{2}(|\alpha|+|\sigma|)-\frac{1}{2}|\beta|}\der g(t,x,v)\in C^0([0,T_{\max}+\frac{1}{2}T_{\sm}];L^2_xL^2_v).$$

Using \eref{uni}, after choosing $ \eps_0$ smaller (so that $\eps$ is sufficiently small) if necessary, there exists $T_{\text{ext}}\in (T_{\max},T_{\max}+\frac{1}{2}T_{\sm}]$ such that \eref{boot_assumption_1} hold upto $T_{\text{ext}}$.\\
This is a contradiction to the definition of $T_{\max}$. Thus we deduce that $T_{\max}=+\infty$.
\end{proof}


\newpage
\begin{thebibliography}{10}


\bibitem{AMUXY12.3}
R.~Alexandre, Y.~Morimoto, S.~Ukai, C.-J. Xu, and T.~Yang.
\newblock The {B}oltzmann equation without angular cutoff in the whole space:
  {II}, {G}lobal existence for hard potential.
\newblock {\em Anal. Appl. (Singap.)}, 9(2):113--134, 2011.

\bibitem{AMUXY12}
R.~Alexandre, Y.~Morimoto, S.~Ukai, C.-J. Xu, and T.~Yang.
\newblock Global existence and full regularity of the {B}oltzmann equation
  without angular cutoff.
\newblock {\em Comm. Math. Phys.}, 304(2):513--581, 2011.

\bibitem{AMUXY12.2}
R.~Alexandre, Y.~Morimoto, S.~Ukai, C.-J. Xu, and T.~Yang.
\newblock The {B}oltzmann equation without angular cutoff in the whole space:
  {I}, {G}lobal existence for soft potential.
\newblock {\em J. Funct. Anal.}, 262(3):915--1010, 2012.


\bibitem{AlLiLi15}
Radjesvarane Alexandre, Jie Liao, and Chunjin Lin.
\newblock Some a priori estimates for the homogeneous {L}andau equation with
  soft potentials.
\newblock {\em Kinet. Relat. Models}, 8(4):617--650, 2015.



\bibitem{AMUXY13}
Radjesvarane Alexandre, Yoshinori Morimoto, Seiji Ukai, Chao-Jiang Xu, and Tong
  Yang.
\newblock Local existence with mild regularity for the {B}oltzmann equation.
\newblock {\em Kinet. Relat. Models}, 6(4):1011--1041, 2013.

\bibitem{AlGa09}
Ricardo~J. Alonso and Irene~M. Gamba.
\newblock Distributional and classical solutions to the {C}auchy {B}oltzmann
  problem for soft potentials with integrable angular cross section.
\newblock {\em J. Stat. Phys.}, 137(5-6):1147--1165, 2009.

\bibitem{ArPe77}
A.~A. Arsen'ev and N.~V. Peskov.
\newblock The existence of a generalized solution of {L}andau's equation.
\newblock {\em \v Z. Vy\v cisl. Mat. i Mat. Fiz.}, 17(4):1063--1068, 1096,
  1977.

\bibitem{Ar11}
Diogo Ars\'enio.
\newblock On the global existence of mild solutions to the {B}oltzmann equation
  for small data in {$L^D$}.
\newblock {\em Comm. Math. Phys.}, 302(2):453--476, 2011.

\bibitem{BaDe85}
C.~Bardos and P.~Degond.
\newblock Global existence for the {V}lasov-{P}oisson equation in {$3$} space
  variables with small initial data.
\newblock {\em Ann. Inst. H. Poincar\'e Anal. Non Lin\'eaire}, 2(2):101--118,
  1985.

\bibitem{BaDeGo84}
C.~Bardos, P.~Degond, and F.~Golse.
\newblock A priori estimates and existence results for the {V}lasov and
  {B}oltzmann equations.
\newblock In {\em Nonlinear systems of partial differential equations in
  applied mathematics, {P}art 2 ({S}anta {F}e, {N}.{M}., 1984)}, volume~23 of
  {\em Lectures in Appl. Math.}, pages 189--207. Amer. Math. Soc., Providence,
  RI, 1986.

\bibitem{BaGaGoLe16}
Claude Bardos, Irene~M. Gamba, Fran{\c{c}}ois Golse, and C.~David Levermore.
\newblock Global solutions of the {B}oltzmann equation over {$\mathbb R^D$}
  near global {M}axwellians with small mass.
\newblock {\em Communications in Mathematical Physics}, 346(2):435--467, Sep
  2016.

\bibitem{BeTo85}
N.~Bellomo and G.~Toscani.
\newblock On the {C}auchy problem for the nonlinear {B}oltzmann equation:
  global existence, uniqueness and asymptotic stability.
\newblock {\em J. Math. Phys.}, 26(2):334--338, 1985.

\bibitem{Bi17}
L\'eo Bigorgne.
\newblock Asymptotic properties of small data solutions of the {V}lasov-{M}axwell
  system in high dimensions.
\newblock {\em arXiv:1712.09698, preprint}, 2017.

\bibitem{Bi18}
L\'eo Bigorgne.
\newblock Sharp asymptotics for the solutions of the three-dimensional massless {V}lasov-{M}axwell system with small data.
\newblock {\em arXiv:1812.09716, preprint}, 2018.

\bibitem{Bi19.1}
L\'eo Bigorgne.
\newblock Sharp asymptotic behavior of solutions of the 3d {V}lasov-{M}axwell system with small data.
\newblock {\em arXiv:1812.11897 , preprint}, 2019.

\bibitem{Bi19.2}
L\'eo Bigorgne.
\newblock Asymptotic properties of the solutions to the {V}lasov-{M}axwell system in the exterior of a light cone.
\newblock {\em arXiv:1902.00764 , preprint}, 2019.

\bibitem{Bi19.3}
L\'eo Bigorgne.
\newblock A vector field method for massless relativistic transport equations and applications.
\newblock {\em arXiv:1907.03121 , preprint}, 2019.

\bibitem{CaSiSn18}
Stephen Cameron, Luis Silvestre, and Stanley Snelson.
\newblock Global a priori estimates for the inhomogeneous {L}andau equation
  with moderately soft potentials.
\newblock {\em Ann. Inst. H. Poincar\'e Anal. Non Lin\'eaire}, 35(3):625--642,
  2018.

\bibitem{Cha19}
Sanchit Chaturvedi.
\newblock {L}ocal existence for the {L}andau Equation with hard potentials.
\newblock {arXiv:1910.11866, preprint}, 2019





\bibitem{De15}
L.~Desvillettes.
\newblock Entropy dissipation estimates for the {L}andau equation in the
  {C}oulomb case and applications.
\newblock {\em J. Funct. Anal.}, 269(5):1359--1403, 2015.

\bibitem{DeVi00}
Laurent Desvillettes and C\'edric Villani.
\newblock On the spatially homogeneous {L}andau equation for hard potentials.
  {I}. {E}xistence, uniqueness and smoothness.
\newblock {\em Comm. Partial Differential Equations}, 25(1-2):179--259, 2000.

\bibitem{DeVi00.2}
Laurent Desvillettes and C\'edric Villani.
\newblock On the spatially homogeneous {L}andau equation for hard potentials.
  {II}. {$H$}-theorem and applications.
\newblock {\em Comm. Partial Differential Equations}, 25(1-2):261--298, 2000.

\bibitem{FaJoSm17}
David Fajman, J\'er\'emie Joudioux, and Jacques Smulevici.
\newblock The stability of the {M}inkowski space for the {E}instein--{V}lasov
  system.
\newblock {\em arXiv:1707.06141, preprint}, 2017.

\bibitem{FaJoSm17.1}
David Fajman, J\'er\'emie Joudioux, and Jacques Smulevici.
\newblock A vector field method for relativistic transport equations with
  applications.
\newblock {\em Anal. PDE}, 10(7):1539--1612, 2017.

\bibitem{Fo10}
Nicolas Fournier.
\newblock Uniqueness of bounded solutions for the homogeneous {L}andau equation
  with a {C}oulomb potential.
\newblock {\em Comm. Math. Phys.}, 299(3):765--782, 2010.

\bibitem{FoGu09}
Nicolas Fournier and H\'el\`ene Gu\'erin.
\newblock Well-posedness of the spatially homogeneous {L}andau equation for
  soft potentials.
\newblock {\em J. Funct. Anal.}, 256(8):2542--2560, 2009.

\bibitem{GlSc88}
R.~T. Glassey and J.~W. Schaeffer.
\newblock Global existence for the relativistic {V}lasov-{M}axwell system with
  nearly neutral initial data.
\newblock {\em Comm. Math. Phys.}, 119(3):353--384, 1988.

\bibitem{GlSt87}
Robert~T. Glassey and Walter~A. Strauss.
\newblock Absence of shocks in an initially dilute collisionless plasma.
\newblock {\em Comm. Math. Phys.}, 113(2):191--208, 1987.

\bibitem{GoImMoVa16}
F.~Golse, Cyril Imbert, Cl\'ement Mouhot, and A.~Vasseur.
\newblock Harnack inequality for kinetic {F}okker--{P}lanck equations with
  rough coefficients and application to the {L}andau equation.
\newblock {\em Annali della Scuola Normale Superiore di Pisa, Classe di Scienze.}, PP. 253-295 | Vol. XIX, issue 1, 2019.

%
\bibitem{Go97}
T.~Goudon.
\newblock Generalized invariant sets for the {B}oltzmann equation.
\newblock {\em Math. Models Methods Appl. Sci.}, 7(4):457--476, 1997.


\bibitem{GrSt11}
Philip~T. Gressman and Robert~M. Strain.
\newblock Global classical solutions of the {B}oltzmann equation without
  angular cut-off.
\newblock {\em J. Amer. Math. Soc.}, 24(3):771--847, 2011.


\bibitem{Guo01}
Yan Guo.
\newblock The {V}lasov-{P}oisson-{B}oltzmann system near vacuum.
\newblock {\em Comm. Math. Phys.}, 218(2):293--313, 2001.

\bibitem{Guo02.1}
Yan Guo.
\newblock The {L}andau equation in a periodic box.
\newblock {\em Communications in Mathematical Physics}, 231(3):391–434,
2002.

\bibitem{Guo02}
Yan Guo.
\newblock The {V}lasov-{P}oisson-{B}oltzmann system near {M}axwellians.
\newblock {\em Comm. Pure Appl. Math.}, 55(9):1104--1135, 2002.

\bibitem{Guo03.2}
Yan Guo.
\newblock Classical solutions to the {B}oltzmann equation for molecules with an
  angular cutoff.
\newblock {\em Arch. Ration. Mech. Anal.}, 169(4):305--353, 2003.

\bibitem{Guo03}
Yan Guo.
\newblock The {V}lasov-{M}axwell-{B}oltzmann system near {M}axwellians.
\newblock {\em Invent. Math.}, 153(3):593--630, 2003.

\bibitem{Guo12}
Yan Guo.
\newblock The {V}lasov-{P}oisson-{L}andau system in a periodic box.
\newblock {\em J. Amer. Math. Soc.}, 25(3):759--812, 2012.

\bibitem{Ha85}
Kamel Hamdache.
\newblock Existence in the large and asymptotic behaviour for the {B}oltzmann
  equation.
\newblock {\em Japan J. Appl. Math.}, 2(1):1--15, 1985.

\bibitem{HeJi17}
Lingbing He and Jin-Cheng Jiang.
\newblock Well-posedness and scattering for the {B}oltzmann equations: soft
  potential with cut-off.
\newblock {\em J. Stat. Phys.}, 168(2):470--481, 2017.

\bibitem{HeSn17}
Christopher Henderson and Stanley Snelson.
\newblock {$C^\infty$} smoothing for weak solutions of the inhomogeneous
  {L}andau equation.
\newblock {\em Arch Rational Mech Anal}, doi:10.1007/s00205-019-01465-7, 2019.

\bibitem{HeSnTa17}
Christopher Henderson, Stanley Snelson, and Andrei Tarfulea.
\newblock Local existence, lower mass bounds, and a new continuation criterion
  for the {L}andau equation.
\newblock {\em Journal of Differential Equations}, https://doi.org/10.1016/j.jde.2018.08.005, 2019.

\bibitem{HeSnTa19.1}
Christopher Henderson, Stanley Snelson, and Andrei Tarfulea.
\newblock Local Solutions of the {L}andau Equation with Rough, Slowly Decaying Initial Data.
\newblock {\em arXiv:1909.05914, preprint}, 2019.

\bibitem{HeSnTa19.2}
Christopher Henderson, Stanley Snelson, and Andrei Tarfulea.
\newblock Local well-posedness of the {B}oltzmann equation with polynomially decaying initial data.
\newblock {\em arXiv:1910.07138, preprint}, 2019.

\bibitem{IlSh84}
Reinhard Illner and Marvin Shinbrot.
\newblock The {B}oltzmann equation: global existence for a rare gas in an
  infinite vacuum.
\newblock {\em Comm. Math. Phys.}, 95(2):217--226, 1984.

%
%
\bibitem{Levermore}
C.~David Levermore.
\newblock Global {M}axwellians over all space and their relation to conserved
  quantites of classical kinetic equations.
\newblock {\em preprint, available online}, 2012.

\bibitem{LiTa17}
Hans Lindblad and Martin Taylor.
\newblock Global stability of {M}inkowski space for the {E}instein--{V}lasov
  system in the harmonic gauge.
\newblock {\em Ann. PDE 3, 9},doi:10.1007/s40818-017-0026-8, 2017.


\bibitem{Lu18}
Jonathan Luk.
\newblock {S}tability of vacuum for the {L}andau equation with moderately soft potentials.
\newblock {Ann. PDE}, 5: 11. https://doi.org/10.1007/s40818-019-0067-2, 2019.

\bibitem{Po88}
Jacek Polewczak.
\newblock Classical solution of the nonlinear {B}oltzmann equation in all
  {${\bf R}^3$}: asymptotic behavior of solutions.
\newblock {\em J. Statist. Phys.}, 50(3-4):611--632, 1988.

\bibitem{Si17}
Luis Silvestre.
\newblock Upper bounds for parabolic equations and the {L}andau equation.
\newblock {\em J. Differential Equations}, 262(3):3034--3055, 2017.

\bibitem{Sm16}
Jacques Smulevici.
\newblock Small data solutions of the {V}lasov-{P}oisson system and the vector
  field method.
\newblock {\em Ann. PDE}, 2(2):Art. 11, 55, 2016.

\bibitem{Sn18}
Stanley Snelson.
\newblock The inhomogeneous {L}andau equation with hard potentials.
\newblock {\em arXiv:1805.10264, preprint}, 2018.

\bibitem{StGu04}
Robert~M. Strain and Yan Guo.
\newblock Stability of the relativistic {M}axwellian in a collisional plasma.
\newblock {\em Comm. Math. Phys.}, 251(2):263--320, 2004.

\bibitem{StGu06}
Robert~M. Strain and Yan Guo.
\newblock Almost exponential decay near {M}axwellian.
\newblock {\em Comm. Partial Differential Equations}, 31(1-3):417--429, 2006.

\bibitem{StGu08}
Robert~M. Strain and Yan Guo.
\newblock Exponential decay for soft potentials near {M}axwellian.
\newblock {\em Arch. Ration. Mech. Anal.}, 187(2):287--339, 2008.

\bibitem{StZh13}
Robert~M. Strain and Keya Zhu.
\newblock The {V}lasov-{P}oisson-{L}andau system in {$\mathbb{R}^3_x$}.
\newblock {\em Arch. Ration. Mech. Anal.}, 210(2):615--671, 2013.

\bibitem{Ta17}
Martin Taylor.
\newblock The global nonlinear stability of {M}inkowski space for the massless
  {E}instein-{V}lasov system.
\newblock {\em Ann. PDE}, 3(1):Art. 9, 177, 2017.

\bibitem{To86}
G.~Toscani.
\newblock On the nonlinear {B}oltzmann equation in unbounded domains.
\newblock {\em Arch. Rational Mech. Anal.}, 95(1):37--49, 1986.

\bibitem{To87}
G.~Toscani.
\newblock {$H$}-theorem and asymptotic trend of the solution for a rarefied gas
  in the vacuum.
\newblock {\em Arch. Rational Mech. Anal.}, 100(1):1--12, 1987.

\bibitem{To88}
G.~Toscani.
\newblock Global solution of the initial value problem for the {B}oltzmann
  equation near a local {M}axwellian.
\newblock {\em Arch. Rational Mech. Anal.}, 102(3):231--241, 1988.

\bibitem{ToBe84}
G.~Toscani and N.~Bellomo.
\newblock Global existence, uniqueness and stability of the nonlinear
  {B}oltzmann equation with almost general gas-particles interaction potential.
\newblock In {\em Proceedings of the conference commemorating the 1st
  centennial of the {C}ircolo {M}atematico di {P}alermo ({I}talian) ({P}alermo,
  1984)}, number~8, pages 419--433, 1985.

\bibitem{Vi98}
C.~Villani.
\newblock On the spatially homogeneous {L}andau equation for {M}axwellian
  molecules.
\newblock {\em Math. Models Methods Appl. Sci.}, 8(6):957--983, 1998.

%
\bibitem{Vi98.2}
C\'edric Villani.
\newblock On a new class of weak solutions to the spatially homogeneous
  {B}oltzmann and {L}andau equations.
\newblock {\em Arch. Rational Mech. Anal.}, 143(3):273--307, 1998.

\bibitem{Wa18.2}
Xuecheng Wang.
\newblock Decay estimates for the 3{D} relativistic and non-relativistic
  {V}lasov--{P}oisson systems.
\newblock {\em arXiv:1805.10837, preprint}, 2018.

\bibitem{Wa18}
Xuecheng Wang.
\newblock Propagation of regularity and long time behavior of the 3{D} massive
  relativistic transport equation {I}: {V}lasov--{N}ordstr\o m system.
\newblock {\em arXiv:1804.06560, preprint}, 2018.

\bibitem{Wa18.1}
Xuecheng Wang.
\newblock Propagation of regularity and long time behavior of the 3{D} massive
  relativistic transport equation {II}: {V}lasov--{M}axwell system.
\newblock {\em arXiv:1804.06566, preprint}, 2018.

\bibitem{Wo18}
Willie Wai~Yeung Wong.
\newblock A commuting-vector-field approach to some dispersive estimates.
\newblock {\em Arch. Math. (Basel)}, 110(3):273--289, 2018.

\bibitem{Wu14}
Kung-Chien Wu.
\newblock Global in time estimates for the spatially homogeneous {L}andau
  equation with soft potentials.
\newblock {\em J. Funct. Anal.}, 266(5):3134--3155, 2014.

\end{thebibliography}
\end{document}